\documentclass{amsart}
\usepackage{amsmath,amsthm,amssymb,amsfonts, amscd}
\usepackage{tikz}
\usepackage{dsfont}
\usetikzlibrary{matrix}
\usepackage{mathdots}

\usepackage{amssymb,latexsym}
\usepackage{color}

\usepackage{amsthm}
\usepackage{eucal}

\usepackage[all,cmtip]{xy}

\newcommand{\wwedge}[1]{\sideset{}{^{#1}}\bigwedge}
\newcommand{\bfv}{\mathbf{v}}
\usepackage{float}
\newcommand{\biddots}{
\begin{tikzpicture}
\filldraw [black] (0,0) circle (1.8pt);
\filldraw [black] (.2,.2) circle (1.8pt);
\filldraw [black] (.4,.4) circle (1.8pt);
\end{tikzpicture}
}
\newcommand{\bddots}{
\begin{tikzpicture}
\filldraw [black] (0,.4) circle (1.8pt);
\filldraw [black] (.2,.2) circle (1.8pt);
\filldraw [black] (.4,0) circle (1.8pt);
\end{tikzpicture}
}
 
 \DeclareFontFamily{OT1}{rsfs}{}
 
\DeclareFontShape{OT1}{rsfs}{n}{it}{<-> rsfs10}{}
\DeclareMathAlphabet{\mathscr}{OT1}{rsfs}{n}{it}

\newcommand{\C}{\mathbb{C}}

\newcommand{\Z}{\mathbb{Z}}

\DeclareFontFamily{OT1}{rsfs}{}

\DeclareFontShape{OT1}{rsfs}{n}{it}{<-> rsfs10}{}
\DeclareMathAlphabet{\mathscr}{OT1}{rsfs}{n}{it}

\newcommand{\Hom}{\mathrm{Hom}}

\newcommand{\R}{\mathbb{R}}
\newcommand{\SO}{\mathrm{SO}}
\newcommand{\Sp}{\mathrm{Sp}}

\newcommand{\OO}{\mathrm{O}}

\swapnumbers
\newtheorem{thm}[subsection]{Theorem}  
\newtheorem{lem}[subsection]{Lemma}         
\newtheorem*{lem*}{Lemma}         
\newtheorem{prop}[subsection]{Proposition}
\newtheorem*{prop*}{Proposition}
\newtheorem{rmk}[subsection]{Remark}

\newtheorem{cor}[subsection]{Corollary}

\theoremstyle{definition}
\newtheorem{defn}[subsection]{Definition}

\numberwithin{equation}{subsection}

\newcommand{\cal}{\mathcal}
\newcommand{\GL}{\mathrm{GL}}

\newcommand{\SU}{\mathrm{SU}}

\newcommand{\vol}{\mathrm{vol}}

\begin{document}

%\tableofcontents

\begin{abstract}
 In Part 1 of  this paper  we construct  a spectral sequence converging to the relative Lie algebra cohomology associated to the action of any subgroup $G$ of the symplectic group  on the polynomial Fock model of the Weil representation, see  \S \ref{spectralsection}. These relative Lie algebra cohomology groups are of interest because they map to the cohomology of suitable arithmetic quotients of the symmetric space $G/K$ of $G$.   We  apply this spectral sequence to the case $G = \SO_0(n,1)$ in Sections \ref{kgeqnsection}, \ref{computationofcplus}, and \ref{computationofcminus} to  compute the relative Lie algebra cohomology groups $H^{\bullet} \big(\mathfrak{so}(n,1), \SO(n); \mathcal{P}(V^k) \big)$. Here  $V = \R^{n,1}$ is  Minkowski space and $\mathcal{P}(V^k)$ is the subspace of $L^2(V^k)$ consisting of all products of polynomials with the Gaussian.  In Part 2 of this paper  we compute the cohomology groups  $H^{\bullet}\big(\mathfrak{so}(n,1), \SO(n); L^2(V^k) \big)$ using spectral theory and representation theory, especially \cite{Li}. In Part 3 of this paper we compute  the maps between the polynomial Fock and $L^2$ cohomology groups  induced by the inclusions $\mathcal{P}(V^k) \subset L^2(V^k)$. 
%In particular, we prove that for $k \leq n$  the groups $H^{\ell} \big(\mathfrak{so}(n,1), \SO(n); \mathcal{P}(V^k) \big)$ are zero in all degrees except $k$ and $n$. In degree $k$ the group is generated over $\mathcal{U}( \mathfrak{sp}(2k,\R))$ by the class $\varphi_k$ of Kudla and Millson, see Section \ref{notation}, Equation \eqref{vaprhikdef}, and this class is a lowest weight vector.  For $k \leq \frac{n}{2}$ the group $H^{\ell}\big(\mathfrak{so}(n,1), \SO(n); L^2(V^k) \big)$ is nonzero if and only if $\ell = k$ or $\ell = n-k$.  In case $k < \frac{n}{2}$ the map from $H^k \big(\mathfrak{so}(n,1), \SO(n); \mathcal{P}(V^k) \big)$ to $H^k\big(\mathfrak{so}(n,1), \SO(n); L^2(V^k) \big)$ is an isomorphism onto the $\mathrm{MU}(k)$-finite vectors, in case $k > \frac{n}{2}$ it is the zero map, and in case $k= \frac{n}{2}$ it is an injection but not a surjection.  In this case, the image is generated over $\mathcal{U}( \mathfrak{sp}(2k,\R))$ by $\phi_{n/2}$, the harmonic projection of the class $\varphi_{n/2}$, and the  cokernel is generated  by the Hodge star $ *\phi_{n/2}$.
\end{abstract}

\title[The relative Lie algebra cohomology for $\SO(n,1)$]{The relative Lie algebra cohomology of the Weil Representation of $\SO(n,1)$}
\author{Nicolas Bergeron, John J. Millson*, Jacob Ralston*}
\date{\today}
\thanks{* Partially supported by NSF grant DMS-1206999}

\maketitle

\section{Introduction}
\bigskip
We let $(V, (,))$ be Minkowski space $\R^{n,1}$ and $e_1, e_2, \ldots, e_{n+1}$ be the standard basis.  Let $V_+$ be the span of $e_1, \ldots, e_n$.  We will consider the connected real Lie group $G = \SO_0(n,1)$ with Lie algebra $\mathfrak{so}(n,1)$ and maximal compact subgroup $K=\SO(n)$ with Lie algebra $\mathfrak{so}(n)$, the subgroup of $G$ that fixes the last basis vector $e_{n+1}$.  Let $\mathcal{P}(V^k)$ be the space of all products of complex-valued polynomials with the Gaussian $\varphi_{0,k}$, see Equation \eqref{gaussian}.  Let $\mathcal{S}_k$ be the $\OO(n)$-invariant complex-valued polynomials on $V^k$ and $\mathcal{R}_k \subset \mathcal{S}_k$ be the $\OO(n)$-invariant complex-valued polynomials on $V_+^k$, see Section \ref{notation}.  We will consider the Weil representation with values in $\mathcal{P}(V^k)$ and $L^2(V^k)$.

%This paper consists of three parts.  In the first, we compute the relative Lie algebra cohomology with values in $\mathcal{P}(V^k)$ using the spectral sequence associated to the filtration of $C^\bullet \big(\mathfrak{so}(n,1), \SO(n); \mathcal{P}(V^k) \big)$ constructed in Section \ref{spectralsection}.  In the second part, we compute the cohomology with values in $L^2(V^k)$.  First, we reduce to computing the relative Lie algebra cohomology with values in the orbits $L^2(G\mathbf{x})$ for $\mathbf{x} \in V^k$ using \cite{Li}.  We then use \cite{Pedon} and \cite{MP} to compute these latter cohomology groups.  In the third part, we calculate the map on cohomology induced by the inclusion of $\mathcal{P}(V^k)$ into $L^2(V^k)$.
We first summarize our results for the cohomology with values in $\mathcal{P}(V^k)$.  In the following theorem, let $\varphi_k$ be the cocycle constructed in the work of Kudla and Millson, \cite{KM2}, see Section \ref{notation}, Equation \eqref{vaprhikdef}.  In what follows, $c_1, \ldots, c_k$ are the cubic polynomials on $V^k$ defined in Equation \eqref{cubic}, $q_1, \ldots, q_n$ are the quadratic polynomials on $V^k$ defined in Equation \eqref{qalphadef}, and $\vol$ is defined in Equation \eqref{voldefn}.

\begin{thm}\label{main}
\hfill

\begin{enumerate}
\item  If $k < n$ then
\begin{equation*}
H^\ell \big(\mathfrak{so}(n,1) ,\SO(n); \mathcal{P}(V^k)\big) = \begin{cases}
\mathcal{R}_k \varphi_k &\text{ if } \ell=k \\
\mathcal{S}_k/(c_1, \ldots, c_k) \varphi_{0,k}\vol &\text{ if } \ell=n \\
0 &\text{ otherwise }
\end{cases}
\end{equation*}
\item If $k = n$ then
\begin{equation*}
H^\ell \big(\mathfrak{so}(n,1) ,\SO(n); \mathcal{P}(V^k)\big) = \begin{cases}
\mathcal{R}_k \varphi_k \oplus \mathcal{S}_k/(c_1, \ldots, c_k) \varphi_{0,k}\vol &\text{ if } \ell=n \\
0 &\text{ otherwise }
\end{cases}
\end{equation*}
\item If $k>n$
\begin{equation*}
H^\ell \big(\mathfrak{so}(n,1) ,\SO(n); \mathcal{P}(V^k)\big) = \begin{cases}
\big( \mathcal{P}_k/(q_1, \ldots, q_n) \big)^K \varphi_{0,k}\vol &\text{ if } \ell=n \\
0 &\text{ otherwise }
\end{cases}
\end{equation*}
\end{enumerate}
\end{thm}

The cohomology groups $H^\ell \big(\mathfrak{so}(n,1) ,\SO(n); \mathcal{P}(V^k)\big)$ are $\mathfrak{sp}(2k, \R)$-modules.  We now describe these modules.  If $k < \frac{n+1}{2}$, then as an $\mathfrak{sp}(2k, \R)$-module $\mathcal{R}_k \varphi_k$ is isomorphic to the space of $\mathrm{MU}(k)$-finite vectors in the holomorphic discrete series representation with parameter $(\frac{n+1}{2}, \cdots,\frac{n+1}{2}) $. If $k < n$, then the cohomology group $H^k\big(\mathfrak{so}(n,1) ,\SO(n); \mathcal{P}(V^k)\big)$ is an irreducible holomorphic representation because it was proved in \cite{KM2} that the class of $\varphi_k$ is a lowest weight vector in $H^k\big(\mathfrak{so}(n,1) ,\SO(n); \mathcal{P}(V^k)\big)$.  On the other hand, the cohomology group $H^n\big(\mathfrak{so}(n,1) ,\SO(n); \mathcal{P}(V^k)\big)$ is never irreducible.  Indeed, if $k < n$ then \\
$H^n\big(\mathfrak{so}(n,1) ,\SO(n); \mathcal{P}(V^k)\big)$ is the direct sum of two nonzero $\mathfrak{sp}(2k,\R)$-modules $H_+^n$ and $H_-^n$ and if $k = n$, then $H^n\big(\mathfrak{so}(n,1) ,\SO(n); \mathcal{P}(V^k)\big)$ is the direct sum of three nonzero $\mathfrak{sp}(2k,\R)$-modules $H_+^n$, $H_-^n$, and $\mathcal{R}_k(V) \varphi_k.$  In the case $k=1$, the authors, with Jens Funke, have shown that $H_+^n$ and $H_-^n$ are each irreducible.

We summarize this theorem in the following chart.  The symbol $\bullet$ means the corresponding group is non-zero.

%\begin{figure} [H]
%\begin{tikzpicture}
%  \matrix (m) [matrix of math nodes,
%    nodes in empty cells,nodes={minimum width=3ex,
%    minimum height=2ex,outer sep=-5pt},
%    column sep=1ex,row sep=1ex]{
%    &&&&&&& \\
%          n   &  &  \bullet & \bullet &  \bullet \bullet \bullet & \bullet & \bullet & \bullet &  \\
%          n-1 &    &  0 & 0 &  \cdots & 0 & \bullet &0 &  \\
%          n-2 &    &  0 & 0 &  \cdots  & \bullet &0 &0 &  \\
%          \vdots  &   &  0 & 0 &  \biddots &0 & 0 &0 &  \\
%          2   &  &  0 & \bullet &  \cdots &0 & 0 &0 &  \\
%          1  &   &  \bullet & 0 &  \cdots &0 & 0 &0 & \\
%          \ell =0   &  &  0  & 0 &  \cdots &0 & 0 &0 &  \\
%    \quad\strut &  & k=1  &  2  &  \cdots &n-2 & n-1 &  n  &  \strut \\};
%\draw[thick] (m-1-2.east) -- (m-9-2.east) ;
%\draw[thick] (m-9-1.north) -- (m-9-9.north) ;
%\end{tikzpicture}
%$$H^\ell(\mathfrak{so}(n,1), \SO(n); \mathcal{P}(V^k))$$
%\end{figure}

\begin{figure} [H]
\begin{tikzpicture}
  \matrix (m) [matrix of math nodes,
    nodes in empty cells,nodes={minimum width=3ex,
    minimum height=2ex,outer sep=-5pt},
    column sep=1ex,row sep=1ex]{
    &&&&&&&&& \\
          n   &  &  \bullet & \bullet &  \bullet \bullet \bullet & \bullet & \bullet & \bullet & \bullet & \bullet \bullet \bullet & \\
          n-1 &    &  0 & 0 &  \cdots & 0 & \bullet &0 & 0 & 0 & \\
          n-2 &    &  0 & 0 &  \cdots  & \bullet &0 &0 & 0 & 0 & \\
          \vdots  &   &  0 & 0 &  \biddots &0 & 0 &0 & \vdots & \vdots & \\
          2   &  &  0 & \bullet &  \cdots &0 & 0 &0 & 0 & 0 & \\
          1  &   &  \bullet & 0 &  \cdots &0 & 0 &0 & 0 & 0 & \\
          \ell =0   &  &  0  & 0 &  \cdots &0 & 0 &0 & 0 & 0 & \\
    \quad\strut &  & k=1  &  2  &  \cdots &n-2 & n-1 &  n  & n+1 & \cdots & \strut \\};
\draw[thick] (m-1-2.east) -- (m-9-2.east) ;
\draw[thick] (m-9-1.north) -- (m-9-11.north) ;
\end{tikzpicture}
$$H^\ell(\mathfrak{so}(n,1), \SO(n); \mathcal{P}(V^k))$$
\end{figure}

We now summarize our results for the cohomology with values in $L^2(V^k)$.  In what follows, $\overline{H}^\ell$ will denote the reduced cohomology in degree $\ell$, that is, the space of $\ell$-cocycles divided by the closure of the space of the exact $\ell$-cocycles.  We will say that the cohomology group $H^\ell$ is reduced if it is equal to $\overline{H}^\ell$.  Note that these cohomology groups are modules for the action of $\mathrm{Mp}(2k, \R)$ and hence for its Lie algebra $\mathfrak{sp}(2k, \R)$.  We will put $K^\prime = \mathrm{MU}(k)$ and if $W$ is a $K^\prime$ module, then $W^{(K^\prime)}$ will denote the submodule of $K^\prime$-finite vectors.  We let $\phi_k$ be the form $(\varphi_{0,\mathbf{x}})_z$ of \cite{KM1}, page 230, the harmonic projection of $\varphi_k$.

To simplify notation in the next theorem and for the rest of this paper, we note that since we will be concerned only with computing the relative Lie algebra cohomology with $\mathfrak{g} = \mathfrak{so}(n,1)$ and $K = \SO(n)$, we will often write $H^\bullet(\mathcal{P}(V^k))$ and $C^\bullet(\mathcal{P}(V^k))$ in place of $H^\bullet(\mathfrak{so}(n,1) ,\SO(n); \mathcal{P}(V^k) )$ and $C^\bullet(\mathfrak{so}(n,1) ,\SO(n); \mathcal{P}(V^k) )$ and similarly for other $(\mathfrak{g}, K)$-modules $\mathcal{V}$.

\begin{thm} \label{L^2mainthm}
\hfill

\begin{enumerate} 
\item Assume that $k >n/2$ and $\ell \notin [\frac{n-1}{2}, \frac{n+1}{2}]$. Then we have:
$$H^\ell (L^2 (V^k)) = \overline{H}^\ell (L^2 (V^k)) = 0.$$

\item Suppose now $k\leq \frac{n}{2}$.  Then for $\ell \notin [\frac{n-1}{2} , \frac{n+1}{2}]$ we have
\begin{enumerate}
\item $H^{\ell} (L^2 (V^k))$ is reduced
\item $H^{\ell} (L^2 (V^k))$ is non-zero if and only if $\ell = k$ or $\ell = n-k$.
\end{enumerate}
Moreover, for $k < \frac{n}{2}$, the cohomology group $H^k (L^2 (V^k))$ is the discrete series representation of $\mathrm{Mp}(2k, \R)$ with parameter $(\frac{n+1}{2}, \ldots, \frac{n+1}{2})$ and hence
$$H^k (L^2 (V^k))^{(K^\prime)} = \mathcal{R}_k(V) \phi_k.$$ 

\item Suppose $n = 2m$.  Then $H^m (L^2 (V^k))$ is reduced and
$H^m (L^2 (V^k))$ is non-zero if and only if $k \geq m$.  Moreover, the cohomology group $H^m (L^2 (V^m))$ is the direct sum of two copies of the discrete series representations of $\mathrm{Mp}_{2m}(\R)$ with parameter $(\frac{n+1}{2}, \ldots, \frac{n+1}{2})$ and
$$H^m (L^2 (V^m))^{(K^\prime)} = \mathcal{R}_m(V) \phi_m \oplus \mathcal{R}_m(V) *\phi_m.$$

\item Suppose $n = 2m+1$.  Then $H^m(L^2(V^k))$ is reduced and $\overline{H}^m(L^2(V^k))$ is non-zero if and only if $k=m$.  Furthermore, $H^{m+1} (L^2(V^k))$ is reduced if and only if $k \leq m$ and
\begin{enumerate}
\item $H^{m+1} (L^2(V^k)) \neq 0$ if and only if $k \geq m$
\item $\overline{H}^{m+1} (L^2(V^k)) \neq 0$ if and only if $k = m.$
\end{enumerate}
Moreover, $H^m (L^2 (V^m))$ is the discrete series representation of $\mathrm{Mp}_{2m}(\R)$ with parameter $(\frac{n+1}{2}, \ldots, \frac{n+1}{2})$ and
$$H^m (L^2 (V^m))^{(K^\prime)} = \mathcal{R}_m(V) \phi_m.$$
\end{enumerate}
\end{thm}

\newpage

We summarize Theorem \ref{L^2mainthm} in the following three charts.  The symbol $\bullet$ means the corresponding group is non-zero.

%\footnotesize{ 
\begin{figure} [H]
\begin{tikzpicture} 
  \matrix (m) [matrix of math nodes,
    nodes in empty cells,nodes={minimum width=3ex,
    minimum height=1ex,outer sep=-5pt},
    column sep=1ex,row sep=1ex]{
          n &    &  0  & 0 &  \cdots  & 0 &0 & 0 & \cdots & \\
          n-1 &    &  \bullet & 0 &  \cdots  & 0 &0 & 0 & \cdots & \\
          n-2  &   &  0 & \bullet &  \cdots  & 0 &0 & 0 & \cdots & \\
          \vdots    & &  0 & 0 &  \bddots  & 0 &0 & 0 & \cdots & \\
          m+1    & &  0 & 0 &  \cdots  & \bullet & 0 & 0 & \cdots & \\
          m     & &  0 & 0 &  \cdots  & 0 & \bullet & \bullet & \bullet \bullet \bullet & \\
          m-1    & &  0 & 0 &  \cdots  & \bullet &0 & 0 & \cdots & \\
          \vdots    & &  0 & 0 &  \biddots  & 0 &0 & 0 & \cdots & \\
          2     & &  0 & \bullet &  \cdots  & 0 &0 & 0 & \cdots & \\
          1     & &  \bullet & 0 &  \cdots  & 0 &0 & 0 & \cdots & \\
          \ell = 0     & &  0  & 0 &  \cdots  & 0 &0 & 0 & \cdots & \\
    \quad\strut & &  k= 1  &  2  &  \cdots  & m-1 &  m & m+1 & \cdots & & \strut \\};
  %\draw[-stealth] (m-3-3.north west) -- (m-2-2.south east);
\draw[thick] (m-1-2.east) -- (m-12-2.east) ;
\draw[thick] (m-12-1.north) -- (m-12-11.north) ;
\end{tikzpicture}

$H^\ell(\mathfrak{so}(2m,1), \SO(2m); L^2(V^k))$
\end{figure}

\begin{figure} [H]
\begin{tikzpicture}
  \matrix (m) [matrix of math nodes,
    nodes in empty cells,nodes={minimum width=3ex,
    minimum height=1ex,outer sep=-5pt},
    column sep=1ex,row sep=1ex]{
          n   &  &  0  & 0 &  \cdots & 0 & 0  & 0 & \cdots & \\
          n-1  &   &  \bullet & 0 &  \cdots & 0 & 0  & 0 & \cdots & \\
          n-2  &   &  0 & \bullet &  \cdots & 0 & 0  & 0 & \cdots & \\
          \vdots  &   &  0 & 0 &  \bddots & 0 & 0 & 0 & \cdots & \\
          m+ 2 &    &  0 & 0 &  \cdots  & \bullet & 0 & 0 &  \cdots & \\
          m+1 &    &  0 & 0 &  \cdots  & 0 & \bullet  & \bullet & \bullet \bullet \bullet & \\
          m  &   &  0 & 0 &  \cdots  & 0 & \bullet & 0 &  \cdots & \\
          m-1 &    &  0 & 0 &  \cdots  & \bullet & 0 & 0 &  \cdots & \\
          \vdots  &   &  0 & 0 &  \biddots & 0 & 0 & 0 & \cdots & \\
          2  &   &  0 & \bullet &  \cdots  & 0& 0 & 0 & \cdots & \\
          1  &   &  \bullet & 0 &  \cdots  & 0& 0  & 0 & \cdots & \\
         \ell = 0  &   &  0  & 0 &  \cdots  & 0 & 0 & 0 & \cdots & \\
    \quad\strut & &  k= 1  &  2  &  \cdots  & m-1 &  m & m+1 & \cdots &  & \strut \\};
  %\draw[-stealth] (m-3-3.north west) -- (m-2-2.south east);
\draw[thick] (m-1-2.east) -- (m-13-2.east) ;
\draw[thick] (m-13-1.north) -- (m-13-11.north) ;
\end{tikzpicture}

$H^\ell(\mathfrak{so}(2m+1,1), \SO(2m+1); L^2(V^k))$
\end{figure}

\begin{figure} [H]
\begin{tikzpicture}
  \matrix (m) [matrix of math nodes,
    nodes in empty cells,nodes={minimum width=3ex,
    minimum height=1ex,outer sep=-5pt},
    column sep=1ex,row sep=1ex]{
          n   &  &  0  & 0 &  \cdots & 0 & 0  & 0 & \cdots & \\
          n-1  &   &  \bullet & 0 &  \cdots & 0 & 0  & 0 & \cdots & \\
          n-2  &   &  0 & \bullet &  \cdots & 0 & 0  & 0 & \cdots & \\
          \vdots  &   &  0 & 0 &  \bddots & 0 & 0 & 0 & \cdots & \\
          m+ 2 &    &  0 & 0 &  \cdots  & \bullet & 0 & 0 &  \cdots & \\
          m+1 &    &  0 & 0 &  \cdots  & 0 & \bullet  & 0 & \cdots & \\
          m  &   &  0 & 0 &  \cdots  & 0 & \bullet & 0 &  \cdots & \\
          m-1 &    &  0 & 0 &  \cdots  & \bullet & 0 & 0 &  \cdots & \\
          \vdots  &   &  0 & 0 &  \biddots & 0 & 0 & 0 & \cdots & \\
          2  &   &  0 & \bullet &  \cdots  & 0& 0 & 0 & \cdots & \\
          1  &   &  \bullet & 0 &  \cdots  & 0& 0  & 0 & \cdots & \\
         \ell = 0  &   &  0  & 0 &  \cdots  & 0 & 0 & 0 & \cdots & \\
    \quad\strut & &  k= 1  &  2  &  \cdots  & m-1 &  m & m+1 & \cdots &  & \strut \\};
  %\draw[-stealth] (m-3-3.north west) -- (m-2-2.south east);
\draw[thick] (m-1-2.east) -- (m-13-2.east) ;
\draw[thick] (m-13-1.north) -- (m-13-11.north) ;
\end{tikzpicture}

$\overline{H}^\ell(\mathfrak{so}(2m+1,1), \SO(2m+1); L^2(V^k))$
\end{figure}
%}

%\normalsize

We now describe the map on cohomology induced by the inclusion $\mathcal{P}(V^k)$ into $L^2(V^k)$.

\begin{thm} \label{part3theorem}
\hfill
\begin{enumerate}
\item Suppose $k < \frac{n}{2}$ then the map from $H^k( \mathcal{P}(V^k))$ to $H^k( L^2(V^k))$ is an isomorphism onto the $K^\prime$-finite vectors.  The lowest weight vector for the $\mathfrak{sp}(2k)$-module $H^k(L^2(V^k))^{(K^\prime)}$ is the image of $\varphi_k$, namely $\phi_k$.

\item Suppose $k = \frac{n}{2}$, then the map from $H^k( \mathcal{P}(V^k))$ to $H^k( L^2(V^k))$ is an isomorphism onto the $K^\prime$-finite vectors $\mathcal{R}_m(V) \phi_m$, the first summand in Theorem \ref{L^2mainthm} (3).  The $K^\prime$-finite vectors in the cokernel are $\mathcal{R}_m(V) *\phi_m$, the second summand in Theorem \ref{L^2mainthm}

\item If $k > \frac{n}{2}$ or $\ell = n$ then the map $H^\ell( \mathcal{P}(V^k))$ to $H^\ell( L^2(V^k))$ is the zero map.
\end{enumerate}
\end{thm}

We now explain why it is important to compute the relative Lie algebra cohomology groups with values in the Weil representation for groups $G$ belonging to dual pairs for the study of the cohomology of  arithmetic quotients of the associated locally symmetric spaces. The key point is that {\it cocycles} of degree $k$  with values in the Weil representation of $\mathrm{Sp}(2k,\R) \times \mathrm{O}(p,q)$ or $\mathrm{U}(a,b) \times \mathrm{U}(p,q)$ give rise (using the theta distribution $\theta$ ) to {\it closed} differential $k$-forms on arithmetic locally symmetric spaces associated to such groups $G$.  This construction gives rise to a map $\theta$ from the relative Lie algebra cohomology of $G$ with values in the Weil representation to the ordinary cohomology of suitable arithmetic quotients $M$ of the symmetric space associated to $G$. Furthermore  if a class $\varphi$ is a  lowest weight vector for the action of $\mathfrak{sp}(2k,\R)$ then $\theta(\varphi)$ is the kernel of an integral operator giving a correspondence between Siegel modular forms of genus $k$ and $H^k(M)$. For all cohomology classes studied so far, the span of their images under $\theta$ is the span of the Poincar\'{e} duals of certain totally geodesic cycles in the arithmetic quotient, the ``special cycles''.  Furthermore, the refined Hodge projection of the map $\theta$ has been shown in \cite{BMM1} and \cite{BMM2} to be onto a certain refined Hodge type for low degree cohomology.  In particular, for $\mathrm{Sp}(2k,\R) \times \OO(n,1)$, it is shown in \cite{BMM1} that this map is onto the $H^k(M)$  for $k < \frac{n}{3}$.  For a description of the map $\theta$ and the induced correspondence between Siegel modular forms and cohomology classes, see the introduction of \cite{FM}, where the induced correspondence is called the geometric theta correspondence.

Further work will be needed to extend these results to the cases of $\SO(p,q)$ and $\SU(p,q)$.  The spectral sequence of Section \ref{spectralsection} exists for general $\SO(p,q)$ and $\SU(p,q)$ and the $E_1$ term will again be a Koszul complex.  However, we do not expect the vanishing results, Propositions \ref{K+vanishing} and \ref{Koszulvanishing}, for the corresponding Koszul cohomology groups to hold.  Also, there are general techniques for computing the relative Lie algebra cohomology with values in $L^2(V^k)$.  We expect there to be analogues of the main techniques used in our computation of the relative Lie algebra cohomology with values in $L^2(V^k)$.
It is important  to generalize  Theorem \ref{part3theorem} to $\mathrm{SO}(p,q)$ and $\mathrm{U}(p,q)$ for general $p$ and $q$ because at present there are many more techniques available for computing cohomology with $L^2$ cofficients than with $\mathcal{P}$ coefficients.

%However, as we will see, there are general techniques for computing with values in $L^2(V^k)$.  Namely the computation reduces to the computation of the relative Lie algebra cohomology with coefficients in the square integrable functions on the orbits which leads to questions about the $L^2$ cohomology of the symmetric space $G/K$ and more generally semisimple symmetric spaces $G/H$ associated to $G$.  Since the map $H^\ell(\mathcal{P}(V^k)) \rightarrow H^\ell(L^2(V^k))$ is an isomorphism onto the $\mathrm{MU}(k)$-finite vectors in $H^\ell(L^2(V^k))$ for $\ell, k < \frac{n}{2}$, it is possible that one can use the $L^2$ cohomology to compute the Fock cohomology in a certain range for $\SO(p,q)$ and $\SU(p,q)$.

We would like to thank Roger Howe for helpful conversations, Bill Casselman for a conversation on smooth vectors and Jens Funke for making his computations for the case $n=1, k=1$ available to us.    We would like to thank Steve Halperin for the proof of Proposition \ref{Steveprop1}.  We would especially like to thank Matei Machedon, who supplied many of the ideas in Part 2.

\newpage

\part{The computation of $H^{\bullet}\big(\mathfrak{so}(n,1), \SO(n); \mathcal{P}(V^k) \big)$}

\section{The relative Lie algebra complexes} \label{notation}

We first establish notation that we will use throughout the paper.  Let \\
$e_1,\ldots,e_{n+1}$ be an orthogonal basis for $V$ such that $(e_i,e_i) =1, 1 \leq i \leq n$ \\
and $(e_{n+1},e_{n+1}) = -1$. We let $x_1,\ldots,x_n, t$ be the corresponding coordinates. 

We have the splitting (orthogonal for the Killing form)
$$\mathfrak{so}(n,1) = \mathfrak{so}(n) \oplus \mathfrak{p}_0$$
and denote the complexification $\mathfrak{p}_0 \otimes \C$ by $\mathfrak{p}$.

We recall that the map $\phi:\bigwedge^2(V) \to \mathfrak{so}(n,1)$ given by
$$\phi(u \wedge v)(w) = (u,w)v - (v,w)u $$
is an isomorphism.  Under this isomorphism the elements $e_i \wedge e_{n+1}, 1 \leq i \leq n$ are a basis for $\mathfrak{p}_0$.  We define $e_{i,n+1} = -e_i \wedge e_{n+1}$ and let $\omega_1,\ldots,\omega_n$ be the dual basis for $\mathfrak{p}_0^*$.  We let $\mathcal{I}_{\ell,n}$ ($\mathcal{I}$ for injections) be the set of all ordered $\ell$-tuples of distinct elements $I= (i_1,i_2,\ldots,i_{\ell})$ from $\{1,2,\ldots,n\}$, that is, the set of all injections from the set $\{1, \ldots, \ell\}$ to $\{1, \ldots, n\}$.  We let $\mathcal{S}_{\ell, n} \subset \mathcal{I}_{\ell,n}$ ($\mathcal{S}$ for stricly increasing) be the subset of strictly increasing $\ell$-tuples.  We define $\omega_I $ for $I \in \mathcal{I}_{\ell,n}$ as above by
$$\omega_I = \omega_{i_1} \wedge \cdots \wedge \omega_{i_{\ell}}.$$
We then define $\vol \in (\wwedge{n}\mathfrak{p}_0^*)^K$ by
\begin{equation} \label{voldefn}
\vol = \omega_1 \wedge \cdots \wedge \omega_n.
\end{equation}

Now let $V^k = \displaystyle \bigoplus_{i=1}^k V$.  We will use $\mathbf{v}$ to denote the element $(v_1,v_2,\cdots,v_k) \in V^k$.  We will often identify $V^k$ with the $((n+1) \times k)$-matrices 
$$\begin{pmatrix}
x_{11} & x_{12} & \cdots & x_{1k} \\
\vdots & \vdots & \ddots & \vdots \\
x_{n1} & x_{n2} & \cdots & x_{nk} \\
t_1 & t_2 & \cdots & t_k
\end{pmatrix}$$
over $\R$ using the basis $e_1, \ldots, e_{n+1}$.  Then $\mathbf{v}$ will correspond to the $(n+1) \times k$ matrix $X$ where $v_j$ is the $j^{th}$ column of the matrix.

Let $V_+$ be the span of $e_1, \ldots, e_n$ and $V_-$ be the span of $e_{n+1}$.  Then we have the splitting $V = V_+ \oplus V_-$ and the induced splitting $V^k = V_+^k \oplus V_-^k$.  We define, for $1 \leq i,j \leq k$, the quadratic function $r_{ij} \in \mathrm{Pol}(V_+^k)$ for $\mathbf{v} \in V_+^k$ by
\begin{equation}
r_{ij}(\mathbf{v}) = (v_i,v_j).
\end{equation}

%For $i,j \in \{1,2,\cdots,k\}$ we define the quadratic polynomial function $\rho_{ij}$ on $V^k$ by
%$$\rho_{ij}(\mathbf{v}) = (v_i,v_j).$$

We let $\mathcal{R}_k = \mathcal{R}_k(V_+)$ be the subalgebra of $\mathrm{Pol}(V_+^k)$ generated by the $r_{ij}$ for $1 \leq i,j \leq k$.  We note that 
$$\mathcal{R}_k = \mathrm{Pol}(V_+^k)^{\mathrm{O}(n)}$$
is the algebra of polynomial invariants of the group $\mathrm{O}(n)$ (the ``First Main Theorem'' for the orthogonal group, \cite{Weyl}, page 53).  Since (as a consequence of our assumption \eqref{k<n/2} immediately below) we will have $k < n$ except in Section \ref{kgeqnsection}, it is a polynomial algebra.  This follows from the ``Second Main Theorem'' for the orthogonal group, \cite{Weyl}, page 75. We will assume that

\begin{equation} \label{k<n/2}
k < n
\end{equation}
for the remainder of Part 1 except for Section \ref{kgeqnsection} where we compute the cohomology for the case $k \geq n$.  

For $K = \SO(n), \OO(n)$ (embedded in $\SO(n,1)$ fixing the last coordinate), or $\OO(n) \times \OO(1)$, any complex $\mathfrak{so}(n, 1) \times K$-module $\mathcal{V}$, and $\sigma : \OO(n,1) \to \mathrm{End}(\mathcal{V})$, we define
$$C^{\bullet}\big(\mathfrak{so}(n,1) ,K; \mathcal{V}\big)$$
by $C^i\big(\mathfrak{so}(n,1) ,K; \mathcal{V}\big) = (\wwedge{i}\mathfrak{p}_0^* \otimes \mathcal{V})^{K}$ and $d = \sum A(\omega_\alpha) \otimes \sigma(e_{\alpha} \wedge e_{n+1})$ as in Borel Wallach \cite{BW}.  Throughout Part 1, the symbol $C$ will denote the complex $C^{\bullet} \big(\mathfrak{so}(n,1) ,\SO(n); \mathrm{Pol}((V \otimes \C)^k)\big)$.

\subsection{The relative Lie algebra complex with values in the Schr\"odinger model}

In this paper, we will be concerned with two models of the Weil representation of the symplectic group $\mathrm{Sp}(2k(n+1), \R)$, the Schr\"odinger model and the Fock model.  For our cohomology computations, we will use the subspace of $U(k(n+1))$-finite vectors.  In the first case, this is the space $\mathcal{P}(V^k)$ and in the second it is $\mathrm{Pol}((V \otimes \C)^k)$.  These two spaces are related by the Bargmann transform, $B: \mathcal{P}(V^k) \to \mathrm{Pol}((V \otimes \C)^k)$ sending the Gaussian (see \eqref{gaussian}) to $1$ and, more generally, Hermite functions to monomials, see \cite{Igusa} Chapter 1 Section 8.
%We now introduce the Schr\"odinger model, $\mathcal{P}(V^k)$, which will connect the two models of focus in Parts 1 and 2.

We define a distinguished element $\varphi_{0,k}$, the Gaussian, of the Schwartz space $\mathcal{S}(V^k)$ by
\begin{equation} \label{gaussian}
\varphi_{0,k}(\mathbf{v}) = \prod_{i=1}^k \exp{ \big((-1/2)(x_{1i}^2 + \cdots + x_{ni}^2 + t_i^2) \big)}.
\end{equation}
Then $\mathcal{P}(V^k)$ is the subspace of the space of complex-valued Schwartz functions on $V^k$ which are products $p(\mathbf{v}) \varphi_{0,k}(\mathbf{v})$ where $p$ is a complex-valued polynomial.

We form the cochain complex $C^{\bullet} \big(\mathfrak{so}(n,1) ,\SO(n); \mathcal{P}(V^k)\big), d$ where $d = \sum_j d^{(j)}$ and
$$d^{(j)} = \sum_{\alpha=1}^n A(\omega_\alpha) \otimes (x_{\alpha,j} \frac{\partial}{\partial t_{j}} + t_{j} \frac{\partial}{\partial x_{\alpha,j}}), \quad 1 \leq j \leq k.$$
Here $A(\omega_\alpha)$ denotes the operation of left exterior multiplication by $\omega_\alpha$.

In Parts 2 and 3 we consider the cochain complex $C^{\bullet} \big(\mathfrak{so}(n,1) ,\SO(n); L^2(V^k)\big)$.  In fact, $L^2(V^k)$ will be replaced by the smooth vectors $L^2(V^k)^\infty$, see the beginning of Part 2, equipped with the same differential.

For the remainder of Part 1 we will work in the Fock model, namely the polynomials on $(V \otimes \C)^k$, denoted $\mathrm{Pol}((V \otimes \C)^k)$.

\subsection{The relative Lie algebra complex with values in the Fock model}

Similar to the above identification, we identify $(V \otimes \C)^k$ with the $((n+1) \times k)$-matrices
$$\begin{pmatrix}
z_{11} & z_{12} & \cdots & z_{1k} \\
\vdots & \vdots & \ddots & \vdots \\
z_{n1} & z_{n2} & \cdots & z_{nk} \\
w_1 & w_2 & \cdots & w_k
\end{pmatrix}$$
over $\C$ using the basis $e_1, \ldots, e_{n+1}$.  Then $\mathbf{v}$ will correspond to the $(n+1) \times k$ matrix $Z$ where $v_j$ is the $j^{th}$ column of the matrix.

The earlier splitting $V = V_+ \oplus V_-$ induces the splitting $(V \otimes \C)^k = (V_+ \otimes \C)^k \oplus (V_- \otimes \C)^k$.  By abuse of notation, we define, for $1 \leq i,j \leq k$, the quadratic function $r_{ij} \in \mathrm{Pol}((V_+ \otimes \C)^k)$ for $\mathbf{v} \in (V_+ \otimes \C)^k$ by
\begin{equation} \label{r_{ij}}
r_{ij}(\mathbf{v}) = (v_i,v_j).
\end{equation}
Here $(\ ,\ )$ denotes the complex bilinear extension of $(\ ,\ )$ to $V \otimes \C$.

We let $\mathcal{R}_k(V_+ \otimes \C) = \mathrm{Pol}((V_+ \otimes \C)^k)^{\OO(n,\C)}$.  We will abuse notation and sometimes use the symbol $\mathcal{R}_k$ in place of $\mathcal{R}_k(V_+ \otimes \C)$.  The meaning of $\mathcal{R}_k$ should be clear from context.  We note that the $r_{ij}, 1 \leq i,j \leq k$, generate $\mathcal{R}_k$.  

In the following we will be concerned with the relative Lie algebra complex where 
$$C^\ell\big(\mathfrak{so}(n,1) ,\SO(n); \mathrm{Pol}((V \otimes \C)^k)\big) \cong \big( \wwedge{\ell}(\mathfrak{p}_0^*) \otimes \mathrm{Pol}((V \otimes \C)^k) \big)^{\SO(n)}$$
and $d = \sum d^{(j)}$ with

\begin{equation} \label{defnofd}
d^{(j)} = \sum_{\alpha=1}^n A(\omega_\alpha) \otimes (\frac{\partial^2}{\partial z_{\alpha,j} \partial w_j} - z_{\alpha,j} w_{j}), \quad 1 \leq j \leq k.
\end{equation}

The reader will show that 
$$C^{\bullet} \big(\mathfrak{so}(n,1) ,\SO(n); \mathrm{Pol}((V \otimes \C)^k)\big) \cong C^{\bullet} \big(\mathfrak{so}(n,1, \C) ,\SO(n, \C); \mathrm{Pol}((V \otimes \C)^k)\big)$$
where $C^{\bullet} \big(\mathfrak{so}(n,1, \C) ,\SO(n, \C); \mathrm{Pol}((V \otimes \C)^k)\big) = \big( \wwedge{\bullet}\mathfrak{p}^* \otimes \mathrm{Pol}((V \otimes \C)^k) \big)^{\SO(n,\C)}$.
%and $\mathfrak{p} = \mathfrak{p}_0 \otimes \C$.  
We will use this isomorphism between cochain complexes throughout the paper.

Note that there is the tensor product map $\mathrm{Pol}((V \otimes \C)^a )\otimes \mathrm{Pol}((V \otimes \C)^b) \to \mathrm{Pol}((V \otimes \C)^{a+b})$ given by 
$$(f_1 \otimes f_2)(\mathbf{v}) = f_1(v_1,\cdots,v_a) f_2(v_{a+1},\cdots,v_{a+b}).$$

The tensor product map induces a bigraded product
\begin{align*}
C^i\big(\mathfrak{so}(n,1) ,\SO(n); \mathrm{Pol}((V \otimes \C)^a)\big) \otimes & C^j\big(\mathfrak{so}(n,1) , \SO(n); \mathrm{Pol}((V \otimes \C)^b)\big) \\
&\to C^{i+j}\big(\mathfrak{so}(n,1) ,\SO(n); \mathrm{Pol}((V \otimes \C)^{a+b})\big)
\end{align*}
which we will call  the outer exterior  product and denote $\wedge$ given by

$$(\omega_{I} \otimes f_I) \wedge (\omega_J \otimes f_J) = (\omega_I \wedge \omega_J)  \otimes (f_I \otimes f_J) = (\omega_I \wedge \omega_J) \otimes f_I f_J.$$
We also have, for $f \in \mathrm{Pol}((V \otimes \C)^k)$, the usual multiplication of functions
$$f(\omega_I \otimes f_I) = \omega_I \otimes f f_I.$$

A key point in the computation of the $k$-coboundaries is a product rule for $d$ relative to the outer exterior product. To state this suppose $\psi$ is an outer exterior product 
$$\psi(\bfv) = \psi_1(v_1) \wedge \psi_2(v_2) \wedge \cdots \wedge \psi_k(v_k)$$
where $\mathrm{deg} (\psi_j) =c_j$.  Then we have
\begin{equation} \label{productrule}
 d \psi = \sum_{\alpha=1}^k (-1)^{\sum_{j=1}^{\alpha -1}c_j} \psi_1 \wedge \cdots \wedge d_\alpha \psi_\alpha \wedge \cdots \wedge \psi_k.
\end{equation}

We define the cocycle $\varphi_1 \in  C^1\big(\mathfrak{so}(n,1), \SO(n); \mathrm{Pol}(V \otimes \C) \big)$ by 
$$\varphi_1 = \sum_{i=1}^n \omega_i \otimes z_i.$$

We then define $\varphi_k \in C^k(\mathfrak{so}(n,1) ,\SO(n); \mathrm{Pol}((V \otimes \C)^k))$ by
\begin{equation} \label{vaprhikdef}
\varphi_k = \varphi_1^{(1)} \wedge \cdots  \wedge \varphi_1^{(k)}.
\end{equation}
Here a superscript $(i)$ on a term in a $k$-fold wedge as above indicates that the term belongs to the $i$-th tensor factor in $\mathrm{Pol}((V \otimes \C)^k)$.  Since $\varphi_1$ is closed and $d$ satisfies \eqref{productrule}, it follows that $\varphi_k$ is also closed.  The cocycle $\varphi_k$ and its analogues for $\SO(p,q)$ and $\mathrm{SU}(p,q)$ played a key role in the work of Kudla and Millson. 

\section{Some decomposition results}

%Let $U$ be a  real vector space of dimension $n$ equipped with a positive definite symmetric bilinear form $(\ , \ )$.  Let $e_1, \ldots, e_n$ be an orthonormal basis and $x_1,\ldots,x_n$ be the corresponding coordinates. 

%Now we again pass to $U^k$, the direct sum of $k$ copies of $U$. We will use $\mathbf{u}$ to denote the element $\mathbf{u}= (u_1,\ldots,u_k)$. Using the basis $e_1,\ldots,e_n$ we identify $U^k$ with $\mathrm{M}(n,k)$. Define 
%\begin{equation} \label{r_{ij}}
%r_{ij}(\mathbf{u}) = (u_i,u_j)
%\end{equation}
For $1 \leq i,j \leq k$, we define the following second order partial differential operator acting on $\mathrm{Pol}((V_+ \otimes \C)^k)$
$$\Delta_{ij} = \sum_{\alpha=1}^n  \frac{\partial^2}{\partial z_{\alpha,i} \partial z_{\alpha,j}}.$$
%Let $\mathcal{R}_k(U)$ be the polynomial ring on $r_{11},r_{12}, \cdots, r_{kk}$.
%We let $\mathrm{Pol}(U^k)$ denote the space of complex-valued polynomial functions on $U^k$ and
We define $\mathcal{H}((V_+ \otimes \C)^k)$ to be the subspace of harmonic (annihilated by all $\Delta_{ij}$) polynomials.  Then we have the classical result, see \cite{KV}, Lemma 5.3, (note that we have reversed their $n$ and $k$)
\begin{thm}\label{generalcase}
\begin{enumerate}
\item The map 
$$p(r_{11},r_{12}, \ldots, r_{kk}) \otimes h(z_{11},z_{12},\ldots,z_{nk})\to p(r_{11},r_{12}, \ldots, r_{kk})h(z_{11},z_{12},\ldots,z_{nk})$$ induces a surjection
$$\mathcal{R}_k(V_+ \otimes \C) \otimes \mathcal{H}((V_+ \otimes \C)^k) \to \mathrm{Pol}((V_+ \otimes \C)^k).$$
\item In case $2k < n$ the map is an isomorphism.
\end{enumerate}
\end{thm}
\begin{rmk}
We will denote this surjection by writing 
$$\mathrm{Pol}((V_+ \otimes \C)^k) = \mathcal{R}_k(V_+ \otimes \C) \cdot \mathcal{H}((V_+ \otimes \C)^k).$$
%Note that in Equation \eqref{k<n/2} we have assumed that $2k < n$.
%Also note that the theorem is usually stated over $\C$.  However, since the harmonics and the invariants are defined over $\R$ it follows that we have the above map $\mathcal{R}_k(U) \otimes \mathcal{H}(U^k) \to \mathrm{Pol}(U^k)$ of real vector spaces. Since it complexifies to a surjection, resp. isomorphism over the complexes in the two cases  it is a surjection resp. isomorphism over  $\R$. 

\end{rmk}

We will need the following three decomposition results.  First, recall $V_+$ is the span of $e_1,\ldots,e_n$.  Then we have
\begin{lem}\label{Minsum}
$$\mathrm{Pol}((V \otimes \C)^k) = \mathrm{Pol}((V_+ \otimes \C)^k) \otimes \C[w_1,\ldots,w_k]$$
\end{lem}
and

%The second decomposition comes from applying Theorem \ref{generalcase} in the case $U=V_+$ to deduce

\begin{lem}\label{posharm}
$$\mathrm{Pol}((V_+ \otimes \C)^k) = \C[r_{11},r_{12}, \ldots, r_{kk}] \cdot  \mathcal{H}((V_+ \otimes \C)^k).$$
\end{lem}

We make the following definition

\begin{defn}
$$\mathcal{S}_k = \mathrm{Pol}((V \otimes \C)^k)^{\OO(n, \C)} = \C[r_{11},r_{12}, \ldots, r_{kk},w_1, \ldots,w_k].$$
\end{defn}

Then we have

\begin{lem}\label{combination}
$$\mathrm{Pol}((V \otimes \C)^k) = \mathcal{S}_k \cdot \mathcal{H}((V_+ \otimes \C)^k).$$
\end{lem}
%\begin{proof}
%Combining the two lemmas we have
%$$\mathrm{Pol}((V \otimes \C)^k) = \C[r_{11},r_{12}, \ldots, r_{kk},w_1, \ldots,w_k] \cdot \mathcal{H}((V_+ \otimes \C)^k),$$
%but $$\rho_{ij} = r_{ij} - w_{i}w_{j}.$$

%\end{proof}

It will be very  useful to us  that as $\SO(n)$-modules we have

\begin{equation}
\mathfrak{p} \cong \mathfrak{p}^* \cong V_+ \otimes \C
\end{equation}
and hence

\begin{equation}\label{gothicp}
\wwedge{i}(\mathfrak{p}^*) \cong \wwedge{i}( V_+ \otimes \C).
\end{equation}

%The three above decomposition results and this last equation are all we need to prove the main part of the  theorem.

\section{The occurrence of the $\OO(n,\C)$-module $\wwedge{\ell}(V_+ \otimes \C)$ in $\mathcal{H} \big( (V_+ \otimes \C)^k \big)$} \label{realrep}
%Here we assume $n>2$. Since multiplication by $\rho_{ij}$ commutes with $d$ we may  think of the cochain complex as a complex of $\mathcal{R}_k(V)$ modules.

%In the theory of $\mathfrak{g},K$ cohomology it is customary (as we have done) to assume that $K$ is connected. However the work of \cite{KV} (and the theory of dual pairs) requires us to use $\mathrm{O}(n, \C)$.  The point of this section is to make the transition between $\SO(n,\C)$ and $\mathrm{O}(n, \C)$ in order to apply the results of \cite{KV}.
%In order to apply \cite{KV} we make the following discussion about the transition between $\SO(n)$ and $O(n)$ and between $\R$ and $\C$. This will also enable us to reconcile differences in notation between our paper and \cite{KV}.  In our discussion of the work of Kashiwara and Vergne we will work with complex-valued polynomial functions (since this is what Kashiwara and Vergne do).   

In this section we compute the $\wwedge{\ell}(V_+ \otimes \C)$ isotypic subspaces for $\OO(n, \C)$ acting on $\mathcal{H} \big( (V_+ \otimes \C)^k \big)$ where we identify $V_+ \otimes \C$ with $\C^n$ and hence $\OO(V_+ \otimes \C)$ with $\OO(n, \C)$ using the basis $e_1, \ldots, e_n$.

 It is standard to parametrize the irreducible representations of $\OO(n, \C)$  by Young diagrams such that the sum of the lengths of the first two columns is less than or equal to $n$, see \cite{Weyl}, Chapter 7, \S 7. 
In \cite{Howe}, Howe defines the depth of an irreducible representation to be the number of rows in the associated diagram.

\begin{lem} \label{vanext}
$$\Hom_{\OO(n, \C)} \big( \wwedge{\ell}(V_+ \otimes \C), \mathcal{H} \big( (V_+ \otimes \C)^k \big) \big) \neq 0 \text{ if and only } \ell \leq k.$$
\end{lem}

\begin{proof}
We will use Proposition 3.6.3 in \cite{Howe}.  $\wwedge{\ell}(V_+ \otimes \C)$ corresponds to the diagram $D$ which is a single column of length $\ell$.  Hence, $\wwedge{\ell}(V_+ \otimes \C)$ has depth $\ell$.  But, Proposition 3.6.3 states that a representation of depth $\ell$ occurs in $\mathcal{H} \big( (V_+ \otimes \C)^k \big)$ if and only if $\ell \leq k$.

\end{proof}

\begin{rmk}
The lemma also follows from Proposition 6.6 ($n$ odd), and  Proposition 6.11 ($n$ even) of Kashiwara-Vergne, \cite{KV}.
\end{rmk}

\begin{lem} \label{intertwine}
\hfill
\begin{enumerate}
\item If $U_1$ is an irreducible representation of $\OO(n, \C)$, then
$$U_2 = \Hom_{\OO(n, \C)} \big( U_1, \mathcal{H} \big( (V_+ \otimes \C)^k \big) \big)$$
is an irreducible representation of $\GL(k, \C)$.
\item Hence, given $U_1$ as above, there exists a unique representation $U_2$ of \\
$\GL(k, \C)$ such that we have an $\OO(n, \C) \times \GL(k, \C)$ equivariant embedding
$$\Psi : U_1 \otimes U_2 \to \mathcal{H} \big( (V_+ \otimes \C)^k \big).$$
\end{enumerate}
\end{lem}

\begin{proof}
Statement (1) follows from Proposition 5.7 of Kashiwara-Vergne \cite{KV}.  Statement (2) follows from statement (1).

\end{proof}

\begin{lem} \label{corext}
In the set-up of the previous lemma, if $U_1$ is $\wwedge{\ell}(V_+ \otimes \C)$, then $U_2$ is $\wwedge{\ell}(\C^k)$ and hence the $\wwedge{\ell} (V_+ \otimes \C)$ isotypic subspace for $\OO(n, \C)$ in $\mathcal{H} \big( (V_+ \otimes \C)^k \big)$ is $\wwedge{\ell}(V_+ \otimes \C) \otimes \wwedge{\ell}(\C^k)$ as an $\OO(n, \C) \times \GL(k, \C)$-module.
\end{lem}

\begin{proof}
This is an immediate consequence of Howe \cite{Howe} Proposition 3.6.3 or \\
Kashiwara-Vergne \cite{KV} Theorem 6.9 ($n$ odd) and Theorem 6.13 ($n$ even).

\end{proof}

In the next section we construct an explicit $\OO(n, \C) \times \GL(k,\C)$-intertwiner
$$\Psi: \wwedge{\ell}(V_+ \otimes \C) \otimes \wwedge{\ell} (\C^k) \to \mathcal{H} \big( (V_+ \otimes \C)^k \big).$$  Thus we will give a direct proof of Lemma \ref{corext} and  the ``if'' part of  Lemma \ref{vanext}. 

\section{The intertwiner $\Psi$}
In what follows we will identify the space $\mathrm{Hom}(\C^k, V_+ \otimes \C)$ with the space of $n$ by $k$ matrices using the basis $e_1,\ldots, e_n$ for $V_+$.  That is,
$$\mathrm{Hom}(\C^k, V_+ \otimes \C) \cong (\C^k)^* \otimes (V_+ \otimes \C) \cong (V_+ \otimes \C)^k \cong M_{n,k}(\C).$$
Let $\ell \leq k$.  Let $J = (j_1,\ldots, j_\ell)$ be in $\mathcal{S}_{\ell,k}$ and $I =(i_1,i_2, \ldots,i_{\ell})$ be in $\mathcal{S}_{\ell,n}$, that is, strictly increasing $\ell$-tuples of elements of $\{1, \ldots, k\}$ and $\{1, \ldots, n\}$ respectively.  Let $\epsilon_1,\ldots, \epsilon_k$ be the standard basis for $\C^k$ and $\alpha_1,\ldots, \alpha_k$ be the dual basis.  Let $\{e_I^*\}$ be the basis of $\wwedge{\ell}(V_+ \otimes \C)^*$ dual to the basis $\{e_I\}$ of $\wwedge{\ell}(V_+ \otimes \C)$.  Now define $e_I$ and $\epsilon_{J}$ by
$$e_I = e_{i_1} \wedge \cdots \wedge e_{i_\ell} \quad \text{and} \quad \epsilon_{J} = \epsilon_{j_1} \wedge \cdots  \epsilon_{j_\ell}.$$
Then we define the intertwiner
\begin{align*}
\Psi: \wwedge{\ell}(V_+ \otimes \C) \otimes \wwedge{\ell} \C^k &\rightarrow \mathrm{Pol}\big( \Hom(\C^k, V_+ \otimes \C) \big) \\
\Psi(e_I \otimes \epsilon_J)(Z) &= e_I^* \big(\wwedge{\ell}(Z) (\epsilon_J) \big).
\end{align*}
Here $Z \in \mathrm{Hom}(\C^k, V_+ \otimes \C)$ and $\wwedge{\ell}(Z) : \wwedge{\ell}(\C^k) \to \wwedge{\ell}(V_+ \otimes \C)$ is the $\ell^{th}$ exterior power of $Z$.  Clearly $\Psi$ is nonzero.

Now define $f_{I,J}(Z)$ to be the determinant of the $\ell$ by $\ell$ minor given by choosing the rows $i_1,i_2,\ldots,i_{\ell}$ and the columns $j_1,\ldots, j_\ell$ of $Z \in \mathrm{Hom}(\C^k, V_+ \otimes \C)$ regarded as an $n$ by $k$ matrix. 
Then we have
\begin{lem}
$$\Psi(e_I \otimes \epsilon_J)(Z) = f_{I,J}(Z).$$
\end{lem}

\begin{lem}
$f_{I,J}(Z)$ is harmonic.
\end{lem}

\begin{proof}
Given a monomial $m$, we have $\Delta_{ij}(m) \neq 0$ if and only if $m = z_{\alpha, i} z_{\alpha, j} m^\prime$ for some $1 \leq \alpha \leq n$ and non-zero monomial $m^\prime$.  But since $f_{I,J}$ is a determinant, it is the sum of monomials each of which has at most one term from a given row.  That is,
$$\frac{\partial^2}{\partial z_{\alpha, i} \partial z_{\alpha, j}} f_{I,J}(Z) = 0$$
for all $i,j, \alpha$ and thus $\Delta_{ij}\big( f_{I,J}(Z) \big) = 0$ term by term.

\end{proof}

\begin{cor}
The intertwiner $\Psi$ maps to the harmonics, that is
$$\Psi: \wwedge{\ell}(V_+ \otimes \C) \otimes \wwedge{\ell} \C^k \rightarrow \mathcal{H}\big( (V_+ \otimes \C)^k \big).$$
\end{cor}

We leave the following proof to the reader.

\begin{lem}
$\Psi$ is $\OO(n, \C) \times \GL(k, \C)$-equivariant.  That is, for $g \in \OO(n, \C)$ and $g^\prime \in \GL(k, \C)$,
\begin{equation}\label{equivariant} 
\Psi \big( \wwedge{\ell}(g)(e_I) \otimes \wwedge{\ell}(g^\prime)(\epsilon_J) \big)(Z) = \Psi(e_I \otimes \epsilon_J)(g^{-1} Z g^\prime).
\end{equation}
\end{lem}

We note that $\Psi$ is equivalent to a bilinear map $\widetilde{\Psi}$ from  $\wwedge{\ell}(V_+ \otimes \C) \times \wwedge{\ell} \C^k$ to $\mathcal{H} \big( (V_+ \otimes \C)^k \big)$ where $\widetilde{\Psi}(\eta, \tau) = \Psi(\eta \otimes \tau)$.  We define
\begin{equation*}
\widetilde{\Psi}^\prime : \wwedge{\ell}(\C^k) \to \mathrm{Hom}_{\OO(n, \C)} \big( \wwedge{\ell}(V_+ \otimes \C), \mathcal{H} \big( (V_+ \otimes \C)^k \big) \big)
\end{equation*}
by
\begin{equation}\label{formulaforpsi}
\widetilde{\Psi}^\prime (\epsilon_J) = \widetilde{\Psi} (\bullet, \epsilon_J) = \sum_{I \in \mathcal{S}_{\ell,n} } f_{I,J} e_I^*.
\end{equation}
We define
\begin{equation*}
\Psi_J = \widetilde{\Psi}^\prime (\epsilon_J) \in \Hom_{\OO(n,\C)}\big( \wwedge{\ell}(V_+ \otimes \C), \mathcal{H} \big( (V_+ \otimes \C)^k \big) \big)
\end{equation*}
and thus
$$\Psi_J(Z) = \sum_{I \in \mathcal{S}_{\ell,n}} f_{I,J}(Z) e_I^*.$$

We now compute $\Hom_{\OO(n,\C)}\big( \wwedge{\ell}(V_+ \otimes \C), \mathcal{H} \big( (V_+ \otimes \C)^k \big) \big)$. Since $\wwedge{\ell} (V_+ \otimes \C) \otimes \wwedge{\ell} (\C^k)$ is an irreducible $\OO(n, \C) \times \GL(k,\C)$-module it follows that $\Psi$ is injective. The image of $\Psi$ is contained in the $\wwedge{\ell} (V_+ \otimes \C)$ isotypic subspace of $\cal{H}(V^k_+)$.  By Lemma \ref{corext}, we know that the $\wwedge{\ell}(V_+ \otimes \C)$-isotypic subspace is isomorphic to this tensor product as an $\OO(n, \C) \times \GL(k,\C)$-module which is irreducible.  Hence $\Psi$ is a nonzero map of irreducible $\OO(n, \C) \times \GL(k, \C)$-modules and hence an isomorphism.  Thus $\widetilde{\Psi}^\prime$ is an isomorphism of $\C$-vector spaces and we have

\begin{lem} \label{psijbasisharm}
$\{\Psi_J : J \in \mathcal{S}_{m,k}\}$ is a basis for the vector space
$$\Hom_{\OO(n,\C)}\big( \wwedge{\ell}(V_+ \otimes \C), \mathcal{H} \big( (V_+ \otimes \C)^k \big) \big).$$
\end{lem}

\begin{prop} \label{Psi_Jbasis}
$\{\Psi_J : J \in \mathcal{S}_{\ell,k}\}$ is a basis for the $\mathcal{S}_k$-module
$$\Hom_{\OO(n,\C)}\big( \wwedge{\ell}(V_+ \otimes \C), \mathrm{Pol} \big( (V \otimes \C)^k \big) \big).$$
\end{prop}

\begin{proof}
By Lemma \ref{psijbasisharm}, we have $\{\Psi_J : J \in \mathcal{S}_{\ell,k}\}$ is a basis for the $\C$-vector space\\
$\Hom_{\OO(n, \C)} \big( \wwedge{\ell}(V_+ \otimes \C), \mathcal{H} \big( (V_+ \otimes \C)^k \big) \big)$.  Since $\Hom_{\OO(n, \C)}(\C, \cdot)$ is exact, the surjection $\mathcal{S}_k \otimes \mathcal{H} \big( (V_+ \otimes \C)^k \big) \to \mathrm{Pol}((V \otimes \C)^k)$ induces a surjection $\phi$ from the $\C$-vector space $\mathcal{S}_k \otimes \Hom_{\OO(n, \C)} \big( \wwedge{\ell}(V_+ \otimes \C), \mathcal{H} \big( (V_+ \otimes \C)^k \big) \big)$ to the $\C$-vector space $\Hom_{\OO(n, \C)} \big( \wwedge{\ell}(V_+ \otimes \C), \mathrm{Pol}((V \otimes \C)^k) \big)$, given by $\phi(f \otimes T) = fT$ and thus \\
$\{\Psi_J : J \in \mathcal{S}_{\ell,k}\}$ spans $\Hom_{\OO(n, \C)} \big( \wwedge{\ell}(V_+ \otimes \C), \mathrm{Pol}((V \otimes \C)^k) \big)$ as an $\mathcal{S}_k$-module.

We now show the elements of this set are independent over $\mathcal{S}_k$.  We claim that if $Z_0 \in \Hom^0(\C^k, V_+ \otimes \C)$, the set of injective homomorphisms, then $\{\Psi_J(Z_0)\}$ is an independent set over $\C$.  Indeed, $\Psi_J(Z_0) = \wwedge{\ell}(Z_0) \epsilon_J \in \wwedge{\ell}(V_+ \otimes \C)$.  And, since $Z_0$ is an injection, so is $\wwedge{\ell}(Z_0)$.  Thus, since $\{\epsilon_J\}$ is an independent set over $\C$ and $\wwedge{m}(Z_0)$ is an injection, $\{\wwedge{\ell}(Z_0) \epsilon_J\}$ is an independent set over $\C$.

Following Equation \eqref{r_{ij}} we have the quadratic $\OO(n, \C)$-invariants $r_{ij}(Z)$ for $Z \in \Hom(\C^k, V_+ \otimes \C)$, where $r_{ij}(Z) = \big( Z(\epsilon_i), Z(\epsilon_j) \big)$ is the inner product of columns.  We will regard $r_{ij}$ both as matrix indeterminates and functions of $X$.  Recall $\mathcal{S}_k = \C[r_{11}, r_{12}, \ldots, r_{kk}, w_1, \ldots, w_k]$.

Now suppose there is some dependence relation over $\mathcal{S}_k$ where $\mathbf{t} \in \C^k$ and we abbreviate $\mathbf{r} = (r_{11}, r_{12}, \ldots, r_{kk})$
$$\sum_J p_J(\mathbf{r}(Z), \mathbf{t}) \Psi_J(Z) = 0.$$
Then for each $(Z_0, \mathbf{t}_0) \in \Hom^0(\C^k, V_+) \times \C^k$, since $\{\Psi_J(Z_0, \mathbf{t}_0)\}$ is independent over $\C$, we have that $p_J(\mathbf{r}(Z_0), \mathbf{t}_0) = 0$ for all $J$ .  And since $k \leq n$, $\Hom^0(\C^k, V_+) \times \C^k$ is dense in $\Hom(\C^k, V_+) \times \C^k$, and thus we have $p_J = 0$ for all $J$.

\end{proof}

\section{Computation of the spaces of cochains}
Recall that $\OO(n) \subset \OO(n,1)$ is the subgroup that fixes the last basis vector $e_{n+1}$.

Recall $\mathcal{I}_{a,n}$ was defined to be the set of all ordered $a$-tuples of distinct elements from $\{1, \ldots, n\}$, equivalently the set of all injective maps from $\{1, \ldots, a \}$ to $\{1, \ldots, n\}$.  We recall $\mathcal{S}_{a, n} \subset \mathcal{I}_{a.n}$ is the set of all strictly increasing $a$-tuples.  Lastly, given $I = (i_1, \ldots, i_a) \in \mathcal{I}_{a,n}$, we define the set $\overline{I} = \{i_1, \ldots, i_a\} \subset \{1, \ldots, n\}$.  Note that this map restricted to $\mathcal{S}_{a,n}$ is a bijection to its image, the set of all subsets of size $a$ of $\{1, \ldots, n\}$.

Let $\ast: \wwedge{\ell} \mathfrak{p}_0^{*} \to \wwedge{n - \ell} \mathfrak{p}_0^{*}$ be the Hodge star operator associated to the Riemannian metric and the volume form $\vol = \omega_1 \wedge \cdots \wedge \omega_n$.  Extend $*$ to $\wwedge{\ell} \mathfrak{p}^*$ to be complex linear.  Hence
\begin{equation} \label{starg}
* \circ g = (\det g) g \circ * \text{ for } g \in \OO(n, \C).
\end{equation}

We now observe that the complex $C = C^{\bullet}(\mathfrak{so}(n,1, \C) ,\SO(n, \C); \mathrm{Pol}((V \otimes \C)^k))$ is the direct sum of two subcomplexes $C_+$ and $C_-$.  Indeed let $\iota \in \mathrm{O}(n, \C)$ be the element satisfying 
\begin{equation} \label{iotadef}
\iota(e_1) = -e_1 \ \text{and} \ \iota(e_j) = e_j, \quad 1 < j \leq n+1.
\end{equation}

Then $\iota \otimes \iota$ acts on the complex $C$ and commutes with $d$.  We define $C_+$ resp. $C_-$ to be the $+1$ resp. $-1$ eigenspace of $\iota \otimes \iota$.  Then we have
$$C = C_+ \oplus C_-.$$
Hence, $H^\bullet(C) = H^\bullet(C_+) \oplus H^\bullet(C_-)$.  By Equation \eqref{starg}, $* \otimes 1$ anticommutes with $\iota \otimes \iota$ and hence 
$$C^{n-\ell}_- = (* \otimes 1) \big(C^\ell_+\big) .$$
Since $C^\ell = C_+^\ell \oplus C_-^\ell$, to compute $C^\ell$ it suffices to compute $C_+^\ell$ and $C_+^{n-\ell}$.  Hence it suffices to compute $C_+$.  We note
\begin{align*}
C_+^\ell &= C^{\ell}\big(\mathfrak{so}(n,1, \C),\SO(n, \C);\mathrm{Pol}((V \otimes \C)^k)\big)^{\iota \otimes \iota} \\
&= C^{\ell}\big(\mathfrak{so}(n,1, \C),\OO(n, \C);\mathrm{Pol}((V \otimes \C)^k)\big).
\end{align*}

\subsection{The computation of $C_+$}

We recall $\mathcal{R}_k = \C[r_{11},r_{12},\ldots, r_{kk}]$ and \\
$\mathcal{S}_k =  \mathcal{R}_k[w_1,\ldots,w_k] =  \C[r_{11},r_{12},\ldots, r_{kk},w_1,\ldots,w_k]$.

We define an isomorphism of $\mathcal{S}_k$-modules
\begin{align*}
F_\ell : \Hom_{\OO(n, \C)}(\wwedge{\ell}(V_+ \otimes \C), \mathrm{Pol}( (V \otimes \C)^k)) &\to \Hom_{\OO(n, \C)}(\wwedge{\ell}\mathfrak{p}, \mathrm{Pol}((V \otimes \C)^k)) \\
e_I^* \otimes p &\mapsto \omega_I \otimes p
\end{align*}
and define, for $J \in \mathcal{S}_{\ell,n}$, $\varPhi_J$ by $F_\ell(\Psi_J) = \varPhi_J$.  Hence
$$\varPhi_J = \sum_{I \in \mathcal{S}_{\ell,n}} \omega_I \otimes f_{I,J}.$$
Then by Lemma \ref{vanext} we have

\begin{lem}
$C_+^\ell = 0$ for $\ell > k$.
\end{lem}

We then have the following consequence of Proposition \ref{Psi_Jbasis}
\begin{prop} \label{varphibasis}
$\{\varPhi_J\}$ is a basis for the $\mathcal{S}_k$-module $C_+^\ell$.
\end{prop}

We now give another description of $\varPhi_J$ as the outer exterior product of $\ell$ copies of $\varphi_1$.  That is,

\begin{equation} \label{varphiJwedgeproduct}
\varphi_{1}^{(j_1)} \wedge \cdots \wedge \varphi_{1}^{(j_\ell)} = \sum_{I \in \mathcal{S}_{\ell,n}} \omega_I  \otimes f_{I,J} = \varPhi_J.
\end{equation}

\begin{rmk}
$$C_+^k =  \mathcal{S}_k \varphi_k$$
\end{rmk}
where $\varphi_k = \varPhi_{1,2, \ldots, k}$ as before.

%\begin{cor} \label{C_+bound}
%\begin{equation*}
%F_p(C_+^\ell) = 0 \text{ for } p < \ell
%\end{equation*}
%\end{cor}

%\begin{cor} \label{C_-bound}
%\begin{equation*}
%F_p(C_-^\ell) = 0 \text{ for } p < n-\ell
%\end{equation*}
%\end{cor}

\subsection{The computation of $C_-$}\label{highdegreecochains}

Since $* \otimes I$ is an isomorphism of $\mathcal{S}_k$-modules from $C_+^\ell \to C_-^{n-\ell}$, we have, abbreviating $(* \otimes I)(\varPhi_J)$ to $(*\varPhi_J)$,

\begin{prop} \label{thecochains}
$$C_+^\ell = \sum_{J \in \mathcal{S}_{\ell, k}} \cal{S}_k \varPhi_J$$
$$C_-^\ell = \sum_{J \in \mathcal{S}_{n-\ell, k}} \cal{S}_k (*\varPhi_J)$$
where if $i > j$ then $\mathcal{S}_{i,j}$ is the empty set.  In particular $C_+^\ell \cong \mathcal{S}_k^{\binom{k}{\ell}}$ for $0 \leq \ell \leq k$ and zero for $\ell > k$, whereas $C_-^\ell \cong \mathcal{S}_k^{\binom{k}{n-\ell}}$ for $\ell \geq n-k$ and zero for $\ell < n-k$.
\end{prop}

\section{The spectral sequence associated to the Weil representation  } \label{spectralsection}
In this section we will construct a  spectral sequence associated to the  relative Lie algebra complex attached to  the action of {\it any} reductive subgroup $G$ of the symplectic group on the polynomial Fock model $\mathcal{P}$. By Proposition \ref{Steveprop1} this spectral sequence  converges to the relative Lie algebra cohomology $H^{\bullet}(\mathfrak{g}, K;\mathcal{P})$.  The Lie algebra $\mathfrak{g}$ of $G$ acts on   $\mathcal{P}$ by quadratic polynomial differential operators.  Hence the action of $\mathfrak{g}$ raises the filtration of $\mathrm{P}$ by polynomial degree at most two.  Hence the exterior differential $d$ raises the induced filtration degree, $p$, of the associated relative Lie algebra complex $C$ by at most two.  But $d$ raises the exterior degree, $\ell$, on $C$ by one. Hence if we redefine the filtration degree to be $ 2 \ell - p$ then the resulting filtration is preserved and decreasing and we obtain the required  spectral sequence.  The convergence of this spectral sequence
follows from  Proposition \ref{Steveprop1} since  the filtration is bounded above by twice the dimension of the symmetric space $G/K$.

Since this spectral sequence will have some exceptional properties, for example the filtration $\{F^{\bullet}\}$ is  bounded above and exhaustive  {\it but not bounded below}, we will first reprove some standard results about spectral sequences, Propositions \ref{Steveprop1}, \ref{grCzeroimpliesCzero} and \ref{generalspectral} in the generality required here.   Our basic reference will be Chapter 2 of \cite{Mc}, especially Theorem 2.1.  However, we warn the reader that the convergence part of \cite{Mc} Theorem 2.1 will not apply to our case, since our filtration is not bounded below.  An expanded version of this section, in particular a proof of Proposition \ref{firstandsecondiso}, may be found in the PhD thesis of the third author, see \cite{Ralston}.

In what follows we will assume the cochain complex $C$ has (cohomological) degrees between $0$ and $n$ for some fixed $n$. 

\subsection{Some general results on spectral sequences}
We first recall that a spectral sequence is a sequence of bigraded (by $\Z \times \Z$) complexes $\{ E^{\bullet,\bullet}_r\, d_r\}$ such that
\begin{equation} \label{definingequationforE}
H^{p,q}(E_r,d_r) = E^{p,q}_{r+1}, \ p,q \in \Z \times \Z, r \geq 0.
\end{equation}
\begin{rmk} Note that $d_r$ is bigraded, hence it will have a bidegree $(a,b)$.
\end{rmk}

We recall the following definitions.
\begin{defn}
A filtration $F^\bullet$ of a cochain complex $C$ is exhuastive if
$$\bigcup_{p \in \Z} F^p C = C $$
and separated if
$$\bigcap_{p \in \Z} F^p C = 0.$$
\end{defn}

Many  occurences of spectral sequences come from the following theorem, see \cite{Mc} Theorem 2.1. Note that the defining formulas for $Z^{p,q}_r$  and $B^{p,q}_r$ on page 33 of \cite{Mc} are not correct as stated but the correct formulas are used throughout pg. 33-35, in particular in the proof of Theorem 2.1.  
\begin{thm}\label{existenceofspecseq}
Suppose $C^{\bullet},d$ is a cochain complex equipped with a decreasing filtration $F^{\bullet} C, - \infty < p < \infty$.  Then there is a spectral sequence $\{ E^{\bullet,\bullet}_r\, d_r\}$ such that $E_0$ is the bigraded complex associated to the filtered complex $F^{\bullet}C$, that is
$$E^{p,q}_0 = F^pC^{p+q}/ F^{p+1}C^{p+q},  -\infty < p < \infty.$$
%\end{enumerate}
\end{thm}
We now define graded subspaces $\overline{Z}^{p,q}_r$ and $\overline{B}^{p,q}_r$ of $E^{p,q}_{r-1}$ by
\begin{equation}
\begin{aligned}
\overline{Z}^{p,q}_r &= \mathrm{ker}\big(d_{r-1}:E^{p,q}_{r-1} \to E^{p+r-1,q-r+2}_{r-1} \bigskip),\\
\overline{B}^{p,q}_r &=\mathrm{im}\big(d_{r-1}:E^{p-r+1,q+r-2}_{r-1} \to E^{p,q}_{r-1} \big).
\end{aligned}
\end{equation}
Hence, by definition we have
\begin{equation*}
E^{p,q}_r = H^{p,q}(E^{p,q}_{r-1}) = \overline{Z}^{p,q}_r / \overline{B}^{p,q}_r .
\end{equation*}
Define subspaces $Z^{p,q}_r$ and $B ^{p,q}_r$ by
\begin{enumerate}
\item $Z^{p,q}_r = \mathrm{ker}\big(d: F^pC^{p+q} \to  F^p C^{p+q+1}/F^{p+r}C^{p+q+1} \big) $
\item $B ^{p,q}_r = \mathrm{im}\big(d:F^{p-r+1}C^{p+q-1} \to F^pC^{p+q}\big)$.
\end{enumerate}
We note two properties of $\{Z^{p,q}_r\}$ and $\{B^{p,q}_r\}$. 
\begin{lem} \label{exhaustsep}
We have
\begin{enumerate}
\item If $F^\bullet$ is exhaustive then $B^{p,q} = \bigcup_{r} B^{p,q}_r$.
\item If $F^\bullet$ is separated  then $Z^{p,q} = \bigcap_{r} Z^{p,q}_r$.
\end{enumerate} 
\end{lem}

We have 

$$ Z^{p,q}_1 \supset Z^{p,q}_2\supset \cdots  \supset Z^{p,q} \supset B^{p,q} \supset \cdots \supset B^{p,q}_2 \supset B^{p,q}_1 .$$

The following proposition is then a complement to Theorem \ref{existenceofspecseq}.

\begin{prop} \label{firstandsecondiso}
We have 
\begin{enumerate}
\item $E^{p,q}_r \cong \frac{Z^{p,q}_r}{B ^{p,q}_r  + Z^{p+1,q-1}_{r-1}}, r \geq 1$. 
\item Under the above isomorphism, the  differential $d_r$  on $E^{p,q}_r$ is induced by the action of $d$ on  $Z^{p,q}_r$ and has bidegree $(r,-r+1)$. 
\end{enumerate}
\end{prop}

\begin{rmk}
In the usual development of the spectral sequence of a filtered complex, formula (1) in Proposition \ref{firstandsecondiso} is taken as the definition of $E^{p,q}_r$, see page 34 of \cite{Mc}.  In this case, equality must be changed to isomorphism in Equation \eqref{definingequationforE}.  We have made the choice above for the sake of brevity.
\end{rmk}

The spectral sequence above converges to the graded vector space associated to the induced filtration of the cohomology. We conclude the general discussion by describing this bigraded vector space.

\subsubsection{The associated graded $\mathrm{gr}(H^\bullet)$}
The filtrations on the cocycles $Z$ and coboundaries $B$ induce a filtration on the cohomology $H$.  A priori, the vector space $\mathrm{gr}^{p,q}(H)$ is a four-fold quotient, but we have the following proposition.

\begin{prop} \label{grH}
\begin{equation*}
\mathrm{gr}^{p,q}(H) \cong \frac{Z^{p,q}}{B^{p,q} + Z^{p+1,q-1}}.
\end{equation*}
\end{prop}

\begin{proof}
By definition,
\begin{equation*}
\mathrm{gr}^{p,q}(H) = \frac{Z^{p,q}/B^{p,q}}{Z^{p+1,q-1}/B^{p+1,q-1}}.
\end{equation*}
By one of the standard isomorphism theorems,
\begin{equation*}
\frac{Z^{p,q}}{B^{p,q} + Z^{p+1,q-1}} \cong \frac{Z^{p,q}/B^{p,q}}{(B^{p,q} + Z^{p+1,q-1})/B^{p,q}}.
\end{equation*}
By another standard isomorphism theorem,
\begin{equation*}
\frac{B^{p,q} + Z^{p+1,q-1}}{B^{p,q}} \cong \frac{Z^{p+1,q-1}}{B^{p,q} \cap Z^{p+1,q-1}}.
\end{equation*}
Finally, observe that $B^{p,q} \cap Z^{p+1,q-1} = B^{p+1,q-1}$.

\end{proof}  

We now give conditions on the filtration of a filtered complex that are sufficient to imply convergence of the associated  spectral sequence.   

\begin{prop} \label{Steveprop1}
Suppose $F^\bullet, C$ is a filtered cochain complex such that the filtration is bounded above and exhaustive.  Then the spectral sequence converges.  That is, for all $p,q$, there is an isomorphism
\begin{equation}
E^{p,q}_\infty \rightarrow gr^{p,q}(H^\bullet).
\end{equation}
\end{prop}

\begin{proof}
Let $p$ be given. Because the filtration is bounded above, for each $(p,q)$ there exists an $r(p)$ so that $Z_r^{p,q} = Z^{p,q}$ for all $r \geq r(p)$.  For $r > \mathrm{max}(r(p), r(p+1))$ we will construct below a surjective map 
\begin{equation} \label{pirmap}
\pi_{r} : E_r^{p,q} \cong \frac{Z_r^{p,q}}{B_r^{p,q} + Z_{r-1}^{p+1, q-1}} \to gr^{p,q}(H^\bullet) \cong \frac{Z^{p,q}}{B^{p,q} + Z^{p+1, q-1}}.
\end{equation}
Since $r > r(p)$, $Z^{p,q}_r =Z^{p,q}$ and we may first define $\widetilde{\pi}_r: Z^{p,q}_r \to Z^{p,q}$ to be the identity map.  We then define $\pi'_r$ to be the induced quotient map
$$\pi'_r:   Z^{p,q}_r  \to \frac{Z^{p,q}}{B^{p,q} + Z^{p+1, q-1}}.$$
Since $r > r(p+1)$, we have $Z^{p+1,q-1}_{r-1} =Z^{p+1,q-1}$.  Moreover, $B_r^{p,q} \subset B^{p,q}$. Hence the map $\pi'_r$ factors through the quotient by $B_r^{p,q} + Z_{r-1}^{p+1, q-1}$ to give the required surjection $\pi_r$. 

Note that there is a quotient map from $E_r^{p,q}$ to $E_{r+1}^{p,q}$ making $\{E_r^{p,q}, r> r(p)\}$ into a direct system. Moreover, the maps $\{ \pi_r: r > r(p)\}$ fit together to induce a morphism from the direct system to  $gr^{p,q}(H^\bullet)$ and hence we obtain a surjective map $\pi_{\infty}: E_{\infty} ^{p,q} \to gr^{p,q}(H^\bullet)$. We claim that $\pi_{\infty}$ is injective. Indeed suppose $x \in  E_{\infty} ^{p,q}$ satisfies $\pi_{\infty}(x) =0$.  Then for some $r$ we have $\pi_r(x) = 0$.  Hence $x \in B^{p,q} + Z^{p+1, q-1}$.  By Lemma \ref{exhaustsep}, we have $x \in B_{r'}^{p,q} + Z^{p+1, q-1}$ for some $r' \geq r$.  Furthermore, since $Z^{p+1,q-1} \subset Z_{r'}^{p+1,q-1}$ we have $x \in B_{r'}^{p,q} +Z_{r'}^{p+1,q-1}$.  Thus $x$ is zero in $E_{r'}^{p,q}$ and hence is zero in the direct limit.

\end{proof}

\subsection{Some consequences of the vanishing of $E_1^{p,q}$}

In the spectral sequence which follows, many of the terms $E_1^{p,q}$ will vanish.  To utilize this feature, we need the following two general propositions from the theory of spectral sequences.  In what follows we assume the filtration $F^{\bullet}$ is bounded above and exhaustive.

The following proposition is an immediate consequence of convergence of the spectral sequence (Proposition \ref{Steveprop1}) since $\mathrm{gr}(H)$ is obtained from $E_1$ by taking successive subquotients.

\begin{prop}\label{grCzeroimpliesCzero}
Suppose $(F^\bullet,M, d)$ is a filtered cochain complex such that $F^\bullet$ is bounded above and exhaustive.  Then $H^\ell(\mathrm{gr}(M))= 0$ implies $H^\ell(M) = 0$.
\end{prop} 

\begin{rmk}
In the case we are studying, $M =C$ has a canonical grading as a vector space.  Hence, in this case, there is a map of graded vector spaces $f: M \to \mathrm{gr}(M)$ which sends $\varphi \in M$ to its leading term.  However, $f$ does not commute with the differential. Hence, there is no map in general (even if $M$ is graded as a vector space)  from the cohomology of $M$ to the cohomology of $\mathrm{gr} (M) $.
\end{rmk} 

\begin{prop} \label{generalspectral}
Let $(F^\bullet,M, d)$ be a filtered cochain complex such that $F^\bullet$ is bounded above and exhaustive.
\begin{enumerate}
\item If $H^{\ell-1}(\mathrm{gr}(M)) = 0$, then there is a well defined map from $H^\ell(M)$ to $H^\ell(\mathrm{gr}(M))$ and it is an injection.
\item If $H^{\ell+1}(\mathrm{gr}(M)) = 0$, then there is a well defined map from $H^\ell(\mathrm{gr}(M))$ to $H^\ell(M)$ and it is a surjection.
\item if $H^{\ell-1}(\mathrm{gr}(M)) = 0$ and $H^{\ell+1}(\mathrm{gr}(M)) = 0$, then the map from $H^\ell(M)$ \\
to $H^\ell(\mathrm{gr}(M))$ is an isomorphism.
\end{enumerate}

\end{prop}

\begin{proof}
We first prove (1). We will construct an inverse system of injective maps
$$\cdots \hookrightarrow E_r^{p,q} \hookrightarrow E_{r-1}^{p,q} \hookrightarrow E_{r-2}^{p,q} \hookrightarrow \cdots \hookrightarrow E_2^{p,q} \hookrightarrow E_1^{p,q}.$$
First note by the hypothesis of (1) we have
\begin{equation} \label{coboundariesvanish}
\overline{B}^{p,q}_r = 0, r \geq 2.
\end{equation}
Next, note that for any spectral sequence $\{E_r, d_r\}$ we have an inclusion
\begin{equation} \label{inclusioncobound}
\overline{Z}^{p,q}_r \hookrightarrow E_{r-1}^{p,q}, r \geq 1.
\end{equation}
But since $E^{p,q}_r = \overline{Z}^{p,q}_r / \overline{B}^{p,q}_r$, by \eqref{coboundariesvanish} we have 
$$E_{r}^{p,q} = \overline{Z}^{p,q}_r, r \geq 2$$
and the inclusion of Equation \eqref{inclusioncobound} becomes 
$$E_{r}^{p,q}\hookrightarrow E_{r-1}^{p,q}, r \geq 2.$$
Thus $\{E_r^{p,q} \}$ is an inverse system of injections which may be identified with a decreasing (for inclusion) sequence of bigraded subspaces of the fixed bigraded vector space $E_1$.  We have constructed the required inverse system. 

The inverse limit of the above sequence is $E^{p,q}_{\infty}$.  Since the inverse limit of an inverse system maps to each member of the system, we have a map $E_\infty^{p,q} \rightarrow E_1^{p,q}$.  In this case the inverse limit is simply the intersection of all the subspaces and the map of the limit is the inclusion of the infinite intersection which is obviously an injection.  Since we have convergence (Proposition \ref{Steveprop1}), $E_\infty^{p,q} \cong \mathrm{gr}^{p,q}(H)$ and $E_1^{p,q} = H^{p,q}(\mathrm{gr}(C))$.  Hence (1) is proved.

%For each $r$ we have a map $Z^{p,q}/B_1^{p,q} \hookrightarrow \overline{Z}_r^{p,q}$.  This induces a map
%$$Z^{p,q}/B_1^{p,q} \hookrightarrow \bigcap_r \overline{Z}_r^{p,q}.$$
%Choose $x \in \bigcap_r \overline{Z}_r^{p,q}$ and $\tilde{x} \in C^{p,q}$ so that under the map $\pi: C^{p,q} \rightarrow C^{p,q}/B_1^{p,q}$, $\pi(\tilde{x}) = x$.  Then $d\tilde{x} \in F^{p+r}C^{p+q+q}$ for all $r$.  Hence, by separation, $d \tilde{x} = 0$ and $\tilde{x} \in Z^{p,q}$.  Thus, this map is onto.  But, because the filtration is exhaustive, $B^{p,q} = B_1^{p,q}$ and hence $Z^{p,q}/B_1^{p,q} = H^{p,q}$.  Hence (1) is proved.

We now prove (2).  We construct a direct system of surjective maps
$$E_1^{p,q} \twoheadrightarrow E_2^{p,q} \twoheadrightarrow \cdots \twoheadrightarrow  E_{r-1}^{p,q} \twoheadrightarrow E_r^{p,q} \twoheadrightarrow E_{r+1}^{p,q} \twoheadrightarrow  \cdots.$$ 

First note by the hypothesis of (2) we have $d_{r-1}|_{E_{r-1}^{p,q}} = 0, r \geq 2, p+q = \ell$ and hence
\begin{equation} \label{everythingisacycle}
E^{p,q}_{r-1}  = \overline{Z}^{p,q}_r = 0, r \geq 2.
\end{equation}
Next, note that for any spectral sequence $\{E_r, d_r\}$ we have a surjection
$$\overline{Z}^{p,q}_r \twoheadrightarrow E_{r}^{p,q}, r \geq 1.$$
Hence, by Equation \eqref{everythingisacycle}, the previous surjection becomes 
$$E_{r-1}^{p,q}\twoheadrightarrow E_{r}^{p,q}, r \geq 2.$$
Hence $\{E_r^{p,q} \}$ is a direct system of surjections which may be identified with a  sequence of bigraded quotient spaces of the fixed bigraded vector space $E_1$. We have constructed the required direct system. 

Since each member of a direct system maps to the direct limit, the space $E_1^{p,q}$ maps to $E_\infty^{p,q}$ and this map is clearly surjective.  As in the proof of (1), since we have convergence (Proposition \ref{Steveprop1}), $E_\infty^{p,q} \cong \mathrm{gr}^{p,q}(H)$ and $E_1^{p,q} = H^{p,q}(\mathrm{gr}(C))$.  Hence (2) is proved.

Lastly, (1) and (2) imply (3).

\end{proof}

\subsection{Construction of the spectral sequence for the case in hand}

We now study the above spectral sequence for the case in hand, $G = \SO(n,1)$, and apply the previous results to it. 

In what follows we will make a change of notation and use $\mathcal{P}_k$ to denote the ring $\mathrm{Pol}((V \otimes \C)^k) $.  The ring $\mathrm{Pol}((V \otimes \C)^k)$ is graded by polynomial degree
$$\mathcal{P}_k = \bigoplus_{i=0}^\infty \mathcal{P}_k(i).$$
This grading of $\mathcal{P}_k$ induces a grading of $C^\ell$, for each $\ell$, called the polynomial grading. We will let $C^\ell(i)$ denote the $i$-th graded summand of $C^\ell$.
%and for $\varphi \in C^\ell$ we will let $\varphi(i)$ denote its $i$-th graded component.
The above grading of $C^\ell$ induces an increasing filtration $F_{\bullet}$  of $C^\ell$ by
$$F_p C^\ell= \bigoplus_{i=0}^p C^\ell(i).$$
The filtration $F_\bullet$ is bounded below but not bounded above and is exhaustive
$$C = \bigcup_{p=0}^\infty F_p C.$$

Note also that $C$ is bigraded by $(\ell, i)$
\begin{equation} \label{Cbigradingorig}
C = \bigoplus_{\ell = 0}^n \bigoplus_{i=0}^\infty C^\ell(i).
\end{equation}

It is clear that $d$ may be written as a direct sum $d = d_2 + d_{-2}$ where $d_2$ increases the polynomial degree by two and $d_{-2}$ lowers the polynomial degree by two, see Equation \eqref{defnofd}. From $d^2 = 0$ we obtain
\begin{lem} \label{doublecomplex}
\hfill

\begin{enumerate}
\item $d_2^2 = 0$
\item $d_{-2}^2 =0$
\item $d_2 d_{-2} + d_{-2} d_2 =0$.
\end{enumerate}
\end{lem} 

In particular we have
\begin{equation} \label{dandfiltrationlevel}
d F_p C^\ell \subset F_{p+2} C^{\ell+1}.
\end{equation}

\begin{rmk}
Relative to the bigrading of $C$ given by Equation \eqref{Cbigradingorig}, $d$ is a bigraded map with
\begin{equation}
d = d_{1,2} + d_{1,-2}.
\end{equation}
\end{rmk}

Since $d$ increases filtration degree, $F_\bullet, C$ is not a filtered cochain complex.  Also, the above filtration of $C$ is increasing whereas the general theory assumes it is decreasing.  We correct this in the next subsection.

%In what follows we will use the standard notation for spectral sequences (e.g. $B_r^{p,q}, Z_r^{p,q}$, and $E_r^{p,q}$).  For definitions of these symbols, see for example McCleary \cite{Mc} pages 33-34.

We now regrade $C$ so that $d$ preserves the filtration and so that the filtration is decreasing.  Since the changes for $C$ will specialize to those for $C_+$ and $C_-$ we will work with $C = C_+ \oplus C_-$.

The vector space underlying the complex $(C, d)$ is bigraded by cochain degree $\ell$ and polynomial degree $p$. As is customary in the theory of spectral sequences, we regrade by complementary degree $q = \ell - p$ and $p$. We change this bigrading according to $(p,q) \rightarrow (p^\prime, q^\prime)$ where $p^\prime = p - 2 \ell$ and $q^\prime = p+q - p^\prime$.  Note that $\ell = p+q = p^\prime + q^\prime$ and $d$ preserves the filtration.  That is,
\begin{equation*}
d F_{p^\prime} C^\ell \subset F_{p^\prime} C^{\ell+1}.
\end{equation*}

The theory of the spectral sequence associated to a filtered cochain complex requires the filtration to be decreasing, however the filtration $F_{p^\prime}$ is increasing.  Accordingly, we pass to the new decreasing filtration $F^{p^{\prime \prime}}$ defined by $F^{p^{\prime \prime}} = F_{-p^{\prime \prime}}$.  As before, $q^{\prime \prime}$ is the complementary degree, so $q^{\prime \prime} = \ell - p^{\prime \prime}$.  Hence
\begin{equation*}
p^{\prime \prime} = 2\ell - p \text{ and } q^{\prime \prime} = p - \ell.
\end{equation*}

Thus, from the bigrading of Equation \eqref{Cbigradingorig} we obtain a new bigrading for $C$
\begin{equation} \label{Cbigrading}
C = \bigoplus_{p^{\prime \prime} = -\infty}^{2n} \bigoplus_{q^{\prime \prime} = -n}^{\infty} C^{p^{\prime \prime}, q^{\prime \prime}}.
\end{equation}

\begin{lem} \label{dnewbigrading}
Relative to the bigrading Equation \eqref{Cbigrading}, we have $d = d_{0,1} + d_{4,-3}$.  In particular, $d$ preserves the new filtration.
\end{lem}

\begin{rmk} \label{d=d_{0,1}}
The differential $d^\prime$ on $E_1$ is induced by the summand $d_{0,1}$ in Lemma \ref{dnewbigrading}.  In fact, we use the negative of $d_{0,1}$.  That is,
$$d^\prime = \sum_{i=1}^k \sum_{\alpha=1}^n \omega_{\alpha} \otimes z_{\alpha i} w_i.$$
\end{rmk}

$F^{\bullet}$ is a decreasing filtration preserved by $d$ and the associated bigraded vector space $E_0^{p'', q''} = \bigoplus_{p'',q''}F^{p''} C^{p'' + q''}$ is  supported in the quadrant $p'' \leq 2n$ and $q'' \geq -n$.  In addition, because the cohomological degree $\ell$ satisfies $ 0 \leq \ell \leq n$ and $\ell = p'' + q''$, $E_0^{p'', q''}$ is supported in the intersection of the above quadrant with the band $0 \leq p'' + q'' \leq n$.

In what follows we will abuse notation and write $p$ and $q$ instead of $p''$ and $q''$.  The following figure shows the support of the bigraded complex $C^{\bullet,\bullet}$ with the new bigrading.

\begin{figure}[H]

\begin{tikzpicture}
%This is the shading inside \ell=0 and \ell=n
\draw[draw=gray!50!white,fill=gray!20!white] 
    plot[smooth,samples=100,domain=-2.5:1.5] (\x,{-\x}) -- 
    plot[smooth,samples=100,domain=3:-1] (\x,{1.5-\x});

%This is the shading to the right of F^p
\draw[draw=gray!50!white,fill=gray!70!white] 
    plot[smooth,samples=100,domain=-1.5:2.5] (1.2,\x) -- 
    plot[smooth,samples=100,domain=2.5:-1.5] (3,\x);

%This is the dark shading of F
\draw[draw=gray!50!white,fill=gray!110!white] 
    plot[smooth,samples=100,domain=1.2:1.5] (\x,{-\x}) -- 
    plot[smooth,samples=100,domain=3:1.2] (\x,{1.5 - \x});

\node at (1.7,-1) {$F^p$};

%This is the line y=-x
\draw[domain=-2.5:2] plot (\x,{-\x});
\node at (-2.1,1.6) {\small{$\ell = 0$}};

%This is the line y=1.5-x
\draw[domain=-1:3] plot (\x,{1.5-\x});
\node at (-.41,2.4) {\small{$\ell = n$}};

%This is the line y=-n
\draw[thick, dashed] (-2.5,-1.5) -- (3.5,-1.5);
\node at (-2,-1.7) {$q = -n$};

%This is the line x=2nhttps://www.sharelatex.com/project/54d11aefb2fef7631b7646b1
\draw[thick, dashed] (3,-2.5) -- (3,2.5);
\node at (3.7,2) {$p = 2n$};

%These are the axes
\draw (-2.5,0)--(3.5,0) node[right]{$p$};
\draw (0,-2.5)--(0,2.5) node[above]{$q$};

%This is a filtration
\draw (1.2,-2.5)--(1.2,2.5);
\node at (1.8,2) {$F^p \rightarrow$};
\end{tikzpicture}
$$C^{p,q} \text{ is supported in the shaded diagonal region}.$$
\end{figure}

Because the filtration is bounded above by $p = 2n$, for each $(p,q)$ there exists an $r(p)$ so that $Z_r^{p,q} = Z^{p,q}$ for all $r \geq r(p)$. In our case it suffices to take $r(p) = 2n - p +1$.  Note that $r(p)$ is a decreasing function of $p$, so $r > r(p)$ implies $r > r(p+k)$ for $k \geq 1$.  Since the complex is bounded below by $q \geq -n$ we also obtain an analogous bound $r(q) = q+n+1$, however, we will use only $r(p)$.
\begin{rmk} For the action of a general reductive $G$ on the polynomial Fock model, the exterior differential may be decomposed as $d = d_{-2} + d_0 + d_2$ and preserves the new filtration.  The support of the resulting bigraded complex will still be contained in the band of the above figure. We leave the details to the reader.  
\end{rmk}

\section{The vanishing of the cohomology of $C$ when $k \geq n$.} \label{kgeqnsection}
In this section, we work under the assumption
\begin{equation}
k \geq n.
\end{equation}
We prove the following theorem, see Equation \eqref{qalphadef} for the definition of the quadratic elements $q_1, \ldots, q_n \in \mathcal{P}(V^k)$.
\begin{thm} \label{kgeqnvanishing}
Assume $k \geq n$.  Then we have 
\begin{equation*}
H^\ell(\mathfrak{so}(n,1), \SO(n); \mathcal{P}_k) = \begin{cases}
0 &\text{ if } \ell \neq n \\
(\mathcal{P}_k/(q_1, \ldots, q_n))^K \vol &\text{ if } \ell = n.
\end{cases}
\end{equation*}
\end{thm}
Define the complex $(A, d_A)$ by
\begin{equation*}
A^\ell = \wwedge{\ell}(\mathfrak{p}^*) \otimes \mathcal{P}_k \text{ and } d_A = \sum_{\alpha = 1}^n \sum_{i=1}^k  A(\omega_\alpha) \otimes z_{\alpha i} w_i.
\end{equation*}
Then $C^\ell = (A^\ell)^K$ and by Remark \ref{d=d_{0,1}} we have, since $K$ is compact,
\begin{equation} \label{E_1CisA}
H^\ell(E_1(C)) = (H^\ell(A))^K.
\end{equation}

\subsection{$A, d_A$ is a Koszul complex.}

Define the quadratic elements $q_1, \ldots, q_n$ of $\mathcal{P}_k$ by
\begin{equation} \label{qalphadef}
q_\alpha = \sum_{i=1}^k z_{\alpha i} w_i \text{ for } 1 \leq \alpha \leq n.
\end{equation}
We note that the $q_\alpha$ are the result of the following matrix multiplication of elements of $\mathcal{P}_k$
\begin{equation*}
\begin{pmatrix}
z_{11} & z_{12} & \cdots & z_{1k} \\
\vdots & \vdots & \ddots & \vdots \\
z_{n1} & z_{n2} & \cdots & z_{nk}
\end{pmatrix}
\begin{pmatrix}
w_1 \\
\vdots \\
w_k
\end{pmatrix}
=
\begin{pmatrix}
q_1 \\
\vdots \\
q_n
\end{pmatrix}.
\end{equation*}
It is clear that
\begin{equation*}
d_A = \sum_{\alpha = 1}^n A(\omega_\alpha) \otimes q_\alpha.
\end{equation*}

We first note that $d_A$ is the differential in the Koszul complex $K(q_1,\ldots,q_n)$ associated to the sequence of the quadratic polynomials $q_1,\ldots,q_n$, see Eisenbud \cite{Eisenbud}, Section 17.2.  To see that the Koszul complex as described in \cite{Eisenbud} is the above complex $A$ we choose $\mathcal{P}_k$ as Eisenbud's ring $R$ and $\mathcal{P}_k^k$ as Eisenbud's module $N$. In our description, since $\mathfrak{p} \cong \C^n$, we are using the exterior algebra $\bigwedge^{\bullet}( (\C^n)^*) \otimes \mathcal{P}_k$.  But the operation of taking the exterior algebra of a module commutes with base change and hence we have $\bigwedge^{\bullet}( (\C^n)^*) \otimes \mathcal{P}_k \cong \bigwedge^{\bullet}(\mathcal{P}_k^n)$. Then we apply Eisenbud's construction with the sequence $q_1,\ldots, q_n$ to obtain the above complex $A$.  We recall that $f_1, \ldots, f_n$ is a regular sequence in a ring $R$ if and only if $f_\alpha$ is not a zero divisor in $R/(f_1, \ldots, f_{\alpha-1})$ for $1 \leq \alpha \leq n$.  The following proposition will be a consequence of the two subsequent lemmas.

\begin{prop} \label{Aregularsequence}
$q_1, \ldots, q_n$ is a regular sequence in $\mathcal{P}_k$.
\end{prop}

The following lemma is Matsumura's corollary to Theorem 16.3 on page 127, \cite{Matsumura}.  It gives a condition under which we may reorder a sequence while preserving regularity.

\begin{lem} \label{Matsumuralemma}
If $R$ is Noetherian and graded and $a_1, \ldots, a_n$ is a regular sequence of homogeneous elements in $R$, then so is any permutation of $a_1, \ldots, a_n$.
\end{lem}

\begin{lem} \label{easyregular}
Let $R = \C[x_1, x_2, \ldots, x_k, y_1, y_2, \ldots, y_k]$.  Then $(x_1 y_1, x_2 y_2, \ldots, x_k y_k)$ is a regular sequence.
\end{lem}

\begin{proof}
We first rewrite $R$ as the tensor product of $n$ polynomial rings
\begin{equation*}
R \cong \C[x_1, y_1] \otimes \C[x_2, y_2] \otimes \cdots \otimes \C[x_n, y_n].
\end{equation*}
Fix $\alpha$ between $1$ and $n$.  We verify that $x_\alpha y_\alpha$ is not a zero divisor in $R_\alpha = R/(x_1 y_1, \ldots, x_{\alpha-1} y_{\alpha-1})$.  Note that
\begin{equation*}
R_\alpha \cong \frac{\C[x_1, y_1]}{(x_1 y_1)} \otimes \frac{\C[x_2, y_2]}{(x_2 y_2)} \otimes \frac{\C[x_{\alpha -1}, y_{\alpha -1}]}{(x_{\alpha -1} y_{\alpha -1})} \otimes \C[x_{\alpha}, y_{\alpha}] \otimes \cdots \otimes \C[x_n, y_n].
\end{equation*}
Let $b_{e,f} = x_\alpha^e y_\alpha^f \in \C[x_\alpha, y_\alpha]$.  Then $\{b_{e,f}\}_{e,f \geq 0}$ is a basis for $\C[x_\alpha, y_\alpha]$.  Now consider the map
\begin{align*}
g : R_\alpha &\to R_\alpha \\
r &\mapsto x_\alpha y_\alpha r.
\end{align*}
We show $g$ is injective and thus $x_\alpha y_\alpha$ is not a zero divisor in $R_\alpha$.  Suppose $r$ is in the kernel of $g$.  Then $r$ has a unique representation
\begin{equation*}
r = \sum_{e,f} a_{1,e,f} \otimes a_{2,e,f} \otimes \cdots \otimes b_{e,f} \otimes a_{\alpha+1,e,f} \otimes \cdots \otimes a_{n,e,f} \text{ where } a_{\beta,e,f} \in \C[x_\beta, y_\beta].
\end{equation*}
Hence
\begin{align*}
g(r) &= \sum_{e,f} a_{1,e,f} \otimes a_{2,e,f} \otimes \cdots \otimes x_\alpha y_\alpha b_{e,f} \otimes \cdots \otimes a_{n,e,f} \\
&= \sum_{e,f} a_{1,e,f} \otimes a_{2,e,f} \otimes \cdots \otimes b_{e+1,f+1} \otimes \cdots \otimes a_{n,e,f} = 0.
\end{align*}
Since $b_{e+1, f+1}$ is a basis for $\C[x_\alpha, y_\alpha]$, we have, for all $e,f \geq 0$,
\begin{equation*}
a_{1,e,f} \otimes a_{2,e,f} \otimes \cdots \otimes a_{\alpha-1,e,f} \otimes a_{\alpha+1,e,f} \otimes \cdots \otimes a_{n,e,f} = 0.
\end{equation*}
Hence $r = 0$.  Thus $x_\alpha y_\alpha$ is not a zero divisor and the lemma is proved.

\end{proof}

We now prove Proposition \ref{Aregularsequence}.
\begin{proof}
First, we examine the sequence $\sigma = (\{z_{\alpha i}\}_{\alpha \neq i}, q_1, \ldots, q_n)$.  That is, let $\sigma$ be the sequence
\begin{equation*}
\sigma = (z_{12}, z_{13}, \ldots, z_{1k}, z_{21}, z_{23}, \ldots, z_{2k}, \ldots, z_{n 1}, z_{n 2}, \ldots, z_{n, n-1}, q_1, q_2, \ldots, q_n).
\end{equation*}
This is the sequence of ``off-diagonal" coordinates $z_{\alpha i}$ for $\alpha \neq i$ followed by the quadratics $q_\alpha$.

It is clear that the ``off-diagonal" $\{z_{\alpha i}\}_{\alpha \neq i}$ form a regular sequence since they are coordinates.  To check that the $q_\alpha$ are regular we work in $\mathcal{P}_k / (\{z_{\alpha i}\}_{\alpha \neq i}) \cong \C[z_{11}, z_{22}, \ldots, z_{nn}, w_1, \ldots, w_n, \ldots, w_k]$.  The image of $q_\alpha$ in this quotient ring is $z_{\alpha \alpha} w_\alpha$.  This a regular sequence by Lemma \ref{easyregular}.

Now we apply Lemma \ref{Matsumuralemma} to reorder $\sigma$ and note that $(q_1, \ldots, q_n, \{z_{\alpha i}\}_{\alpha \neq i})$ is a regular sequence.  Hence $(q_1, \ldots, q_n)$ is a regular seqeuence.

\end{proof}

We now compute the cohomology of $(A, d_A)$
\begin{prop} \label{Avanishing}
\begin{equation*}
H^\ell(A) = \begin{cases}
0 &\text{ if } \ell \neq n \\
\mathcal{P}_k/(q_1, \ldots, q_n) \vol &\text{ if } \ell = n.
\end{cases}
\end{equation*}
\end{prop}

\begin{proof}
Corollary 17.5 of \cite{Eisenbud} (with $M = R$), states that the cohomology of a Koszul complex $K(f_1,\ldots,f_n)$ below the top degree vanishes if $f_1,\ldots,f_n$ is a regular sequence and that in this case the top cohomology $H^n(K(f_1, \ldots, f_n))$ is isomorphic to $R/(f_1, \ldots, f_n)$.

\end{proof}

We now prove Theorem \ref{kgeqnvanishing}.

\begin{proof}

By Equation \eqref{E_1CisA}, we have, since $\vol$ is $K$-invariant,
\begin{equation*}
H^\ell(E_1(C)) = \begin{cases}
0 &\text{ if } \ell \neq n \\
(\mathcal{P}_k/(q_1, \ldots, q_n))^K \vol &\text{ if } \ell = n.
\end{cases}
\end{equation*}
Thus, by Proposition \ref{grCzeroimpliesCzero} and statement $(3)$ of Proposition \ref{generalspectral} we have
\begin{equation*}
H^\ell(\mathfrak{so}(n,1), \SO(n); \mathcal{P}_k) = \begin{cases}
0 &\text{ if } \ell \neq n \\
(\mathcal{P}_k/(q_1, \ldots, q_n))^K \vol &\text{ if } \ell = n.
\end{cases}
\end{equation*}

\end{proof}

We now show that $H^n(C)$ is not finitely generated as an $\mathcal{R}_k$-module.

\begin{prop}
The map from $\C[w_1, \ldots, w_k]$ to $H^n(C)$ sending $f$ to $[f \vol]$ is an injection.
\end{prop}

\begin{proof}
First, note that $f(w_1, \ldots, w_k)$ is $K$-invariant.  Now we show the map is an injection.  There is an inclusion of polynomial algebras
$$\C[w_1, \ldots, w_k] \hookrightarrow \mathcal{P}_k.$$
This map has a right inverse, $\pi$, where $\pi(w_i) = w_i$ and $\pi(z_{\alpha i})=0$.  Then since $\pi(q_\alpha) = \pi(\sum_i z_{\alpha i} w_i) = 0$, $\pi$ descends to a right inverse from $\mathcal{P}_k / (q_1, \ldots, q_\alpha) \to \C[w_1, \ldots, w_k]$.  Hence the map is injective.

\end{proof}

\section{The computation of the cohomology of $C_+$} \label{computationofcplus}
Now that we have settled the case of $k \geq n$ we return to assuming
\begin{equation*}
k < n.
\end{equation*}
We will compute the cohomology of the associated graded complex $E_0 = \mathrm{gr}(C)$.  By Remark \ref{d=d_{0,1}}, our differential is $d_{0,1}$.  We abuse notation and call this operator $d$.
%Recall that we can write $d = d_2 + d_{-2}$.  In what follows we will replace $d$ with $-d$ and hence $d_2$ with $-d_2$.  We will compute the cohomology of the complex using $d_2$ as the differential.  
We will see there is only one non-zero cohomology group and use the results of Section \ref{spectralsection} to compute the cohomology of the original complex.

Throughout this section, $J$ will denote an element of $\mathcal{S}_{\ell,n}$.  To simplify the notation in what follows, we will abbreviate $\omega \otimes \varphi$ to $\varphi \omega$ for $\omega \in \wwedge{\bullet} \mathfrak{p}^*$ and $\varphi \in \mathrm{Pol}((V \otimes \C)^k)$.  In particular,

\begin{equation} \label{defofphij}
\varPhi_J= \sum_{I \in \mathcal{S}_{\ell,n}} f_{I,J} \omega_I.
\end{equation}

Recall we have fixed $k$ with $k < n$ and we have the ring
\begin{equation}
\mathcal{S}_k = \C[r_{11},r_{12},\ldots, r_{kk},w_1,\ldots,w_k].
\end{equation}
%\mathcal{R}_k = & \C[r_{11},r_{12},\ldots, r_{kk}]\\

For $1 \leq i \leq k, J \in \mathcal{S}_{\ell,k}$, if $i \notin \overline{J}$, then we denote by $\{J, i\}$ the element of $\mathcal{S}_{\ell+1, k}$ that corresponds to the set $\overline{J} \cup \{i\}$.  Now we define $\varPhi_{J, i} \in C_+^{\ell+1}$  by

\begin{equation}
\varPhi_{J,i} = \begin{cases} (-1)^{J(i)} \varPhi_{\{J, i\}} & \text{ if } i \notin \overline{J} \\ 0 & \text{ if } i \in \overline{J}.
\end{cases}
\end{equation}
where $J(i)$ is defined as follows.

\begin{defn}
$J(i) = |\{j \in \overline{J} : j < i\}|$ is the number of elements of $J$ less than $i$.
\end{defn}

We remark that the reason for the sign $(-1)^{J(i)}$ in this notation is that we have put the $i$ in the appropriate spot instead of the beginning.  In particular, we have the following lemma.

\begin{lem}
$$\varphi_1^{(i)} \wedge \varPhi_J = \varPhi_{J,i}.$$
\end{lem}

The following formula for $d$ is then immediate.

\begin{prop} \label{donbasisfirstcase}
\begin{equation}
d \varPhi_J =   \sum_{i=1}^k w_i \varphi_1^{(i)} \wedge \varPhi_J = \sum_{i=1}^k w_i \varPhi_{J,i}
\end{equation}
%and if $F \in \mathcal{S}_k$ then 
%\begin{equation}
%d(F \varPhi_J) =   \sum_{i=1}^k (-1)^{J(i)} \big((\frac{\partial}{\partial t_{i}} - 2 t_{i}) F\big) \Phi_{J,i} .
%\end{equation}
\end{prop}

\subsection{The map from $\mathrm{gr}(C_+)$ to a Koszul complex $K_+$}

We define $K_+$ to be the complex given by
\begin{equation}
K_+^\bullet = \wwedge{\bullet}((\C^k)^*) \otimes \mathcal{S}_k \text{ with the differential } d_{K_+} = \sum_i A(dw_i) \otimes w_i.
\end{equation}
Here $w_1, \ldots, w_k$ are coordinates on $\C^k$ and $dw_1, \ldots, dw_k$ are the corresponding one-forms.

We define a map $\Psi_+$ of $\mathcal{S}_k$-modules from the associated graded complex $\mathrm{gr}(C_+)$ to $K_+$ by sending $\varPhi_J$ to $dw_J$.  In particular, this sends $\varphi_1^{(i)} \mapsto dw_i$.  Recall that the degrees $\ell$ such that $C_+^\ell$ is non-zero range from $0$ to $k$.

The following lemma is an immediate consequence of Proposition \ref{donbasisfirstcase}. We leave its verification to the reader.

\begin{lem} $\Psi_+$  is an isomorphism  of cochain complexes, $\Psi_+ : \mathrm{gr}(C_+) \to K_+$.
\end{lem} 

We now compute the cohomology of the complex $K_+, d_{K_+}$. 

\begin{prop}\label{K+vanishing}
\hfill

\begin{enumerate}
\item $H^\ell(K_+) = 0,  \ell \neq k$ 
\item $H^k(K_+) = \mathcal{S}_k / (w_1, \ldots, w_k) dw_1 \wedge \cdots \wedge dw_k \cong \mathcal{R}_k dw_1 \wedge \cdots \wedge dw_k$.
\end{enumerate}
\end{prop}

We first prove statement $(1)$ of Proposition \ref{K+vanishing}.  We first note that 
$$ d_K = \sum_j w_j \otimes A(dw_j)$$
is the differential in the Koszul complex $K_+(w_1,\ldots,w_k)$ associated to the sequence of the linear polynomials $w_1,\ldots,w_k$, see Eisenbud \cite{Eisenbud}, Section 17.2.  To see that the Koszul complex as described in \cite{Eisenbud} is the above complex $K$ we choose $\mathcal{S}_k$ as Eisenbud's ring $R$ and $\mathcal{S}_k^k$ as Eisenbud's module $N$. In our description we are using the exterior algebra $\bigwedge^{\bullet}( (\C^k)^*) \otimes \mathcal{S}_k$.  But the operation of taking the exterior algebra of a module commutes with base change and hence we have $\bigwedge^{\bullet}( (\C^k)^*) \otimes \mathcal{S}_k \cong \bigwedge^{\bullet}(\mathcal{S}_k^k)$. Then we apply Eisenbud's construction with the sequence $w_1,\ldots, w_k$ to obtain the above complex $K_+$.  We recall that $f_1, \ldots, f_k$ is a regular sequence in a ring $R$ if and only if $f_i$ is not a zero divisor in $R/(f_1, \ldots, f_{i-1})$ for $1 \leq i \leq k$.  The following lemma is obvious.

\begin{lem} \label{K+regularsequence}
$w_1, \ldots, w_k$ is a regular sequence in $\cal{S}_k$.
\end{lem}

Statement $(1)$ of Proposition \ref{K+vanishing} then follows from Lemma \ref{K+regularsequence} above and Corollary 17.5 of \cite{Eisenbud} (with $M = R$), which states that the cohomology of a Koszul complex $K(f_1,\ldots,f_k)$ below the top degree vanishes if $f_1,\ldots,f_k$ is a regular sequence.

We next note that statement $(2)$ of Proposition \ref{K+vanishing} follows from \cite{Eisenbud}, Corollary 17.5, which states that if $f_1,\ldots,f_k$ is a regular sequence in the ring $R$ then the top cohomology $H^k(K(f_1, \ldots, f_k))$ is isomorphic to $R/(f_1, \ldots, f_k)$.

We now pass from the above results for $K_+$ to the corresponding results for $C_+$.  

\begin{thm}\label{K+vanCvan}
\hfill

\begin{enumerate}
\item $H^{\ell}(C_+) = 0$ $\ell \neq k$
\item $H^k(C_+) = \mathcal{S}_k / (w_1, \ldots, w_k) \varPhi_k \cong \mathcal{R}_k \varphi_k$.
\end{enumerate}
\end{thm}

\begin{proof}
Since $K_+$ is the associated graded complex of $C_+$, the statement (1) is an immediate consequence of Proposition \ref{K+vanishing} and Proposition \ref{grCzeroimpliesCzero}.

Statement $(2)$ follows by applying statement $(3)$ of Proposition \ref{generalspectral} and noting that $\varPhi_k$ is the form $\varphi_k$ of Kudla and Millson.

\end{proof}

\section{The computation of the cohomology of $C_-$} \label{computationofcminus}
We now compute the cohomology of the associated graded complex $\mathrm{gr}(C_-)$.  As in the previous section, the differential is $d_{0,1}$.  Again, we abuse notation and call this operator $d$.  We will see there is only one non-zero cohomology group and use the results of Section \ref{spectralsection} to compute the cohomology of the original complex.

In Proposition \ref{thecochains} we proved that $\{* \Phi_J : J \in \mathcal{S}_{n-\ell, k} \}$ was a basis for the $\mathcal{S}_k$-module $C_-^\ell$.  Note that in order to obtain a cochain of degree $\ell$, we assume $J \in \mathcal{S}_{n-\ell, k}$ instead of $\mathcal{S}_{\ell,k}$, since by Equation \eqref{defofphij}, for $J \in \mathcal{S}_{\ell, k}, \varPhi_J= \displaystyle \sum_{I \in \mathcal{S}_{\ell,n}} f_{I,J} \omega_I$ and hence

\begin{equation} \label{def*phij}
*\varPhi_J =   \sum_{I \in \mathcal{S}_{\ell,n}}   f_{I,J} (*\omega_I)
\end{equation}
has degree $n-\ell$.  For our later computations, we need to replace the determinant $f_{I,J}$ of \eqref{def*phij} by the monomials $z_{I,J}$ where $z_{I,J} = z_{i_1,j_1} \cdots z_{i_{n-\ell},j_{n-\ell}}$.  In order to do this, we sum over all ordered subsets $\mathcal{I}_{n-\ell, n}$, instead of just $\mathcal{S}_{n-\ell, n}$ (those which are in increasing order), to obtain

\begin{equation} \label{defof*phij}
*\varPhi_J =  \sum_{I \in \mathcal{I}_{n-\ell,n}}   z_{I,J} (*\omega_I).
\end{equation}

Using this basis we may identify $C_-^\ell$ with the direct sum of $\binom{k}{n-\ell}$ copies of $\mathcal{S}_k$.

\subsection{A formula for $d$} \label{dformula}
\hfill

Our goal is to prove Proposition \ref{donbasis}, a formula for $d$ relative to the basis $\{*\varPhi_J\}$.  Recall that for $J \in \mathcal{S}_{n-\ell,k}$, and $1 \leq i \leq k$, we have $J(i)$ is the number of elements in $J$ less than $i$.   For $1 \leq i \leq k$, $J \in \mathcal{S}_{\ell,k}$, if $i \in \overline{J}$, then we denote by $\{J - i\}$ the element of $\mathcal{S}_{\ell-1, k}$ that corresponds to the set $\overline{J} - \{i\}$.  Now we define $*\varPhi_{J - i} \in C_+^{\ell+1}$  by

\begin{equation}
\varPhi_{J-i} = \begin{cases} (-1)^{J(j)} \varPhi_{\{J - i\}} & \text{ if } i \in \overline{J} \\ 0 & \text{ if } i \notin \overline{J}.
\end{cases}
\end{equation}

%In what follows the symbol $J-j$, for $j \in \overline{J}$, will mean the element of $\mathcal{S}_{n-\ell-1, k}$ corresponding to the subset $\overline{J} - \{j\} \in P_{n-\ell-1, k}$. In case $j \notin J$ we define $\varPhi_{J-j}$ to be zero. 

%Similarly, given $I \in I_{n-m,n}, i \in \overline{I}$, there is some $1 \leq \ell \leq n-m$ so that $i = i_\ell$ and we define the symbol $I-i$ to be the element $(i_1, \ldots, \widehat{i_\ell}, \ldots, i_{n-m}) \in I_{n-m-1,n}$, the symbol $\widehat{i_\ell}$ indicating that the term $i_\ell$ is omitted.  We also define $I(i) = \ell$, the number of predecessors of $i$ in $I$.

\begin{prop} \label{donbasis}
Assume $|J| = n-\ell$, then we have
\begin{equation*}
d(*\varPhi_J) = (-1)^{(n-\ell-1)} \big( \sum_{j \in J}   \sum_{i =1}^k w_i r_{ij} *\varPhi_{J-j}    \big).
\end{equation*} 
\end{prop}
The proposition will follow from the next two lemmas. Note first that from the defining formula we have
\begin{equation}\label{firstformulaford}
d(*\varPhi_J) = \sum_{i = 1}^k \sum_{\alpha=1}^n z_{\alpha,i} w_i \omega_\alpha \wedge (*\varPhi_J).
\end{equation}

In what follows we will need to extend the definition of $J-j$ for $J \in \mathcal{S}_{n-\ell,k}$ to elements $I \in \mathcal{I}_{n-\ell,n}$.  Given $I \in \mathcal{I}_{n-\ell,n}$, we define the symbol $I-i_s$ to be the element $(i_1, \ldots, \widehat{i_s}, \ldots, i_{n-\ell}) \in \mathcal{I}_{n-\ell-1,n}$, the symbol $\widehat{i_s}$ indicating that the term $i_s$ is omitted.

We leave the proof of the next lemma to the reader.

\begin{lem} \label{wedgeandstar}
Given $I \in \mathcal{I}_{n-\ell,n}$,
\begin{align*}
\omega_{\alpha} \wedge *(\omega_I) = &(-1)^{(n-\ell-1)}* \iota_{e_{\alpha, n+1}}(\omega_I)\\
= & (-1)^{(n-\ell-1)}\sum_{s=1}^{n-\ell} (-1)^{s-1}\delta_{\alpha, i_{s}}  *  (\omega_{I-i_s}).\\
\end{align*}

Here $\delta_{\alpha, i_{s}}$ is the Kronecker delta 
$$\delta_{\alpha, i_{s}} = \begin{cases} 1 & \ \text{if} \  \alpha = i_{s}  \\ 0 & \ \text{if} \  \alpha  \neq i_{s} 
\end{cases}$$
\end{lem}

\begin{lem} \label{mainformulaford}
\begin{equation*}
\sum_{\alpha=1}^n z_{\alpha, i} A(\omega_\alpha) *\varPhi_J = \sum_{s=1}^{n-\ell} (-1)^{s-1} r_{i j_s} *\varPhi_{J-j_s}
\end{equation*}
\end{lem}

\begin{proof}
By Lemma \ref{wedgeandstar},
\begin{align*}
\sum_{\alpha=1}^n z_{\alpha, i} &A(\omega_\alpha) *\varPhi_J = (-1)^{n-\ell-1}* \sum_{\alpha=1}^n z_{\alpha, i} \iota_{e_{\alpha, n+1}} (\varPhi_J) \\
%= &(-1)^{n-\ell-1}* \frac{\varphi_{0,k}}{\varphi_0^{(J)}} \sum_{\alpha=1}^n x_{\alpha, i} \iota_{e_{\alpha, n+1}} (\varphi_1^{(J)}) \\
= &(-1)^{n-\ell-1}* \sum_{\alpha=1}^n z_{\alpha, i} \iota_{e_{\alpha, n+1}} ( \varphi_1^{(j_1)} \wedge \cdots \wedge \varphi_1^{(j_s)} \wedge \cdots \varphi_1^{(j_{n-\ell})} ) \\
= &(-1)^{n-\ell-1}* \sum_{s=1}^{n-\ell} (-1)^{s-1} \big( \varphi_1^{(j_1)} \wedge \cdots \wedge \sum_{\alpha = 1}^n z_{\alpha,i} \iota_{e_{\alpha, n+1}}(\varphi_1^{(j_s)}) \wedge \cdots \varphi_1^{(j_{n-\ell})} \big)
\end{align*}
where the second equality is by Equation \eqref{varphiJwedgeproduct}.  However,
\begin{align*}
\sum_{\alpha = 1}^n z_{\alpha,i} \iota_{e_{\alpha, n+1}}(\varphi_1^{(j_s)}) = & \sum_{\alpha = 1}^n z_{\alpha,i} \sum_{\beta=1}^n z_{\beta, j_s} \iota_{e_{\alpha, n+1}}(\omega_\beta) \\
= &\sum_{\alpha = 1}^n \sum_{\beta=1}^n z_{\alpha,i} z_{\beta, j_s}  \delta_{\alpha, \beta}\\
= &r_{i,j_s}.
\end{align*}
We see that in the $j_s^{th}$ slot we have replaced $\varphi_1$ by $(-1)^{s-1} r_{i,j_s}$ and the lemma follows.

\end{proof}

Proposition \ref{donbasis} follows by substituting the formula of Lemma \ref{mainformulaford} into Equation \eqref{firstformulaford}.

\subsection{The map from $\mathrm{gr}(C_-)$ to a Koszul complex $K_-$}

Define the cubic polynomials $c_j \in \cal{S}_k$ by
\begin{equation}\label{cubic}
c_j = \sum_{i=1}^k r_{ij} w_i, 1 \leq j \leq k.
\end{equation} 

We note that the $c_i$ are the result of the following matrix multiplication of elements of $\mathcal{S}_k$
\begin{equation*}
\begin{pmatrix}
r_{11} & \cdots & r_{1k} \\
\vdots & \ddots & \vdots \\
r_{k1} & \cdots & z_{kk}
\end{pmatrix}
\begin{pmatrix}
w_1 \\
\vdots \\
w_k
\end{pmatrix}
=
\begin{pmatrix}
c_1 \\
\vdots \\
c_k
\end{pmatrix}.
\end{equation*}
We then define $K_-$ to be the complex given by

\begin{equation} \label{newcomplexK-}
K_-^{\bullet} =  \wwedge{\bullet}(( \C^k)^*) \otimes \cal{S}_k  
\ \text{with the differential} \ d_{K_-} = \sum_{j=1}^k  A(dw_j) \otimes c_j.
\end{equation}

In order to obtain an isomorphism of complexes we need to shift degrees according to the following definition.

\begin{defn}
Let $M$ be a cochain complex, $j$ an integer, then we define the cochain complex $M[j]$ by $(M[j])^i = M^{j+i}$.
\end{defn}

We define a map $\Psi_-$ from the associated graded complex $\mathrm{gr}(C_-)[n-k]^{\ell}$ to $K_-^{\ell}$ by sending $* \varPhi_J$ to $* dw_J$.  Here by $*$ we mean the Hodge star for the standard Euclidean metric on $\R^k$ extended to be complex linear. Note that the degrees $\ell$ such that $C_-[n-k]^\ell$ is nonzero range from $0$ to $k$.

The following lemma is an immediate consequence of Proposition \ref{donbasis}. We leave its verification to the reader.

\begin{lem} $\Psi_-$  is an isomorphism  of cochain complexes 
$$\Psi_- : \mathrm{gr}(C_-)[n-k] \to K_-.$$
\end{lem} 

We now compute the cohomology of the complex $K_-, d_{K_-}$. 

\begin{prop}\label{Koszulvanishing}
\hfill

\begin{enumerate}
\item $H^\ell(K_-) = 0,  \ell \neq k$ 
\item $H^k(K_-) = \mathcal{S}_k / (c_1, \ldots, c_k) \vol$.
\end{enumerate}
\end{prop}

We mimic the proof of Proposition \ref{Avanishing}.  We note $d_{K_-}$ is the differential in the Koszul complex $K(c_1,\ldots,c_k)$ associated to the sequence of the cubic polynomials $c_1,\ldots,c_k$, \cite{Eisenbud}, Section 17.2.  It remains to show $c_j$ is a regular sequence.

\begin{lem} \label{regularsequence}
The sequence $(c_1, \ldots, c_k)$ is a regular sequence in $\mathcal{S}_k$.
\end{lem}

\begin{proof}
We follow the method of proof of Proposition \ref{Aregularsequence}.  By Lemma \ref{easyregular} we see that $(r_{11} w_1, \ldots, r_{kk} w_k)$ is a regular sequence in $\C[r_{11}, r_{22}, \ldots, r_{kk}, w_1, \ldots, w_k]$.  We now examine the sequence $\sigma = (\{r_{ij}\}_{i < j}, c_1, \ldots, c_k)$ of ``super-diagonal" $\{r_{ij}\}_{i<j}$ followed by $c_1, \ldots, c_k$.  That is,
\begin{equation*}
\sigma = (r_{12}, r_{13}, \ldots, r_{1k}, r_{23}, \ldots, r_{2k}, \ldots, r_{k-1,k}, c_1, \ldots, c_k).
\end{equation*}
It is clear that the ``super-diagonal" $r_{ij}$ form a regular sequence since they are coordinates.  To check if the $c_i$ are regular we work in
$$\mathcal{S}_k / (\{r_{ij}\}_{i < j}) \cong \C[r_{11}, r_{22}, \ldots, r_{kk}, c_1, \ldots, c_k].$$
The image of $c_i$ in this quotient ring is $r_{ii} w_i$ which form a regular sequence.

Now we apply Lemma \ref{Matsumuralemma} to reorder $\sigma$ and note that $(c_1, \ldots, c_k, \{r_{ij}\}_{i < j})$ is a regular sequence.  Hence $(c_1, \ldots, c_k)$ is a regular seqeuence.

\end{proof}

Proposition \ref{Koszulvanishing} then follows by the same argument that appears after Lemma \ref{K+regularsequence}.

%from Lemma \ref{regularsequence} above and Corollary 17.5 of \cite{Eisenbud} (with $M = R$), which states the the cohomology of a Koszul complex $K(f_1,\ldots,f_k)$ below the top degree vanishes if $f_1,\ldots,f_k$ is a regular sequence.

%We next note that statement $(2)$ of Proposition \ref{Koszulvanishing} follows from  \cite{Eisenbud}, Corollary 17.5. 

We now pass from the above results for $K_-$ to the corresponding results for $C_-$.  

\begin{thm}\label{KvanCvan}
\hfill

\begin{enumerate}
\item $H^{\ell}(C_-) = 0$ $\ell \neq n$
\item $H^n(C_-) \cong \mathcal{S}_k / (c_1, \ldots, c_k)$.
\end{enumerate}
\end{thm}

\begin{proof}
Since $K_-$ is the associated graded complex of $C_-[n-k]$, statement (1) is a consequence of Proposition \ref{Koszulvanishing} and Proposition \ref{grCzeroimpliesCzero}.

Statement $(2)$ follows by applying statement $(3)$ of Proposition \ref{generalspectral}.

\end{proof}

\subsection{Infinite generation of $H^n(C)$ as an $\mathcal{R}_k$-module}

We will now demonstrate that $H^n(C)$ is not finitely generated as an $\mathcal{R}_k$-module.  In what follows, recall $\vol =\omega_1 \wedge \cdots \wedge \omega_n$.

\begin{prop}
The map from $\C[w_1, \ldots, w_k]$ to $H^n(C_-)$ sending $f$ to $[f \vol]$ is an injection.  Furthermore, $\C[w_1, \ldots, w_k]$ generates $H^n(C_-)$ over $\mathcal{R}_k$.
\end{prop}

\begin{proof}
There is an inclusion of polynomial algebras
$$\C[w_1, \ldots, w_k] \hookrightarrow \mathcal{S}_k.$$
This map has a right inverse, $\pi$, where $\pi(w_i) = w_i$ and $\pi(r_{ij})=0$.  Then since $\pi(c_j) = \pi(\sum_i r_{ij} w_i) = 0$, $\pi$ descends to a right inverse from $\mathcal{S}_k / (c_1, \ldots, c_k) \to \C[w_1, \ldots, w_k]$.  Hence the map is injective.

The second statement is obvious since $\C[w_1, \ldots, w_k]$ generates $\mathcal{S}_k$ as an $\mathcal{R}_k$-module.

\end{proof}

\begin{rmk}
Note that 
\begin{enumerate}
\item If $k \neq n$, then
$$H^n(C) = H^n(C_-) = \mathcal{S}_k / (c_1, \ldots, c_k) [\vol].$$
\item If $k=n$ then
$$H^n(C) = \mathcal{S}_n / (c_1, \ldots, c_n) [\vol] \oplus \mathcal{R}_n \varphi_n.$$
\end{enumerate}
\end{rmk}

\subsection{The decomposability of $H^n(C)$ as a $\mathfrak{sp}(2k, \R)$-module} Define \\
$\iota^\prime \in O(n,1)$ by $\iota^\prime(e_j) = e_j, 1 \leq j \leq n$ and $\iota^\prime(e_{n+1}) = -e_{n+1}$.  Then since $\iota^\prime \otimes \iota^\prime$ acts on $\big( \wwedge{\ell}\mathfrak{p}^* \otimes \mathrm{Pol}(V^k) \big)^{\SO(n)}$ and commutes with $d$, it acts on $H^n(C)$.   Since $\iota^\prime \otimes \iota^\prime$ has order two, we get the eigenspace decomposition into the $-1$ and $+1$ eigenspaces

$$H^n(C) = H^n(C)_- \oplus H^n(C)_+.$$

\begin{lem}
$H^n(C)_-$ and $H^n(C)_+$ are nonzero.
\end{lem}

\begin{proof}
Note that $\iota^\prime(\vol) = (-1)^n \vol$ and if $p(w_1, \ldots, w_k)$ is homogenous of degree $a$, then $\iota^\prime(p) = (-1)^a p$.  Hence $\iota^\prime \otimes \iota^\prime ([\vol \otimes p]) = (-1)^{n+a} [\vol \otimes p]$.

\end{proof}

Since the action of $\iota^\prime$ on $\mathrm{Pol}(V^k)$ commutes with the action of $\mathfrak{sp}(2k, \R)$, the above decomposition of $H^n(C)$ is invariant under $\mathfrak{sp}(2k, \R)$.

\section{A simple proof of nonvanishing of $H^n(C)$.}\label{nonvanishingtop}
In what follows we let $G$ be a connected, noncompact, and semisimple Lie group with maximal compact $K$.  We let $n = \mathrm{dim}(G/K)$ and $\mathcal{V}$ be a  $(\mathfrak{g}, K)$-module with $\mathcal{V}^*$ the dual.
\begin{lem}
Suppose either

$(1)$ $\mathcal{V}$ is a topological vector space and $K$ acts continuously. Furthermore, assume there exists a nonzero $\mathfrak{g}$-invariant continuous linear functional $\alpha \in (\mathcal{V}^*)^{\mathfrak{g}}$\\
or

$(2)$ there exists a $\mathfrak{g}$-invariant linear functional $\alpha$ and a $K$-invariant vector $v \in \mathcal{V}$ such that $\alpha(v) \neq 0$ (no topological hypotheses needed).

\vspace{.25cm}

Then $H^n(\mathfrak{g}, K; \mathcal{V}) \neq 0$.
\end{lem}
\begin{proof}
We will first assume (2). We let $\vol$ be the  element in
$\wedge^n(\mathfrak{p}^*)$ which is of unit length for the metric induced by the Killing form and of the correct orientation. Then $\vol \otimes v$ is invariant under the product group $K \times K$, hence is invariant under the diagonal and hence gives rise to an $n$-cochain with values in 
$\mathcal{V}$ which is automatically a cocycle. Let $[\vol \otimes v]$ be the corresponding cohomology class. 
Now $\alpha$ induces a map on cohomology
\begin{equation*}
\alpha_* : H^n(\mathfrak{g}, K; \mathcal{V}) \to H^n(\mathfrak{g}, K; \R).
\end{equation*}
But $H^n(\mathfrak{g}, K; \R)$ is the ring of invariant differential $n$-forms on $D = G/K$.  Thus $H^n(\mathfrak{g}, K; \R) = \R [\vol]$.  Finally, we have
$$ \alpha_*[\vol \otimes v] = [\vol \otimes \alpha(v)] = \alpha(v) [\vol] \neq 0.$$
Hence, $[\vol \otimes v ] \neq 0$.

Now we reduce (1) to (2). 
 Since $\alpha \neq 0$ there is some $v \in \mathcal{V}$ such that $\alpha(v) \neq 0$. Let $dk$ be the Haar measure on $K$ normalized so that $\int\nolimits_{K} dk = 1$. We define the projection  $ p: \mathcal{V} \to \mathcal{V}^{K}$ by
\begin{equation*} \label{projection} 
 p(v) = \int_{K} k \cdot v dk.
\end{equation*}
The reader will verify since $\alpha$ is $K$ invariant and $\alpha(v)$ is continuous in $V$  that
\begin{equation*} \label{pinvariant}
 \alpha(p(v)) = \alpha(v).
\end{equation*}
Hence 
$$\alpha(p(v)) \neq 0 \ \text{and} \ p(v) \in \mathcal{V}^{K}.$$
Now the result follows from the argument of  case (2).

\end{proof}

We will apply $(2)$ for the following examples.  Let $G = \SO_0(p,q)$  (resp $\mathrm{SU}(p,q)$), $K = \SO(p) \times \SO(q)$ (resp $\mathrm{S}(\mathrm{U}(p) \times \mathrm{U}(q))$), $V = \R^{p,q}$ (resp $\C^{p,q}$) and $\mathcal{V} = \mathcal{P}(V^k)$ or $\mathcal{S}(V^k)$.  Then if $\alpha = \delta_0$, the Dirac delta distribution at the origin, $\alpha$ is a non-zero element of $(\mathcal{V}^*)^\mathfrak{g}$.  Hence, we have proved

\begin{thm} \label{topforweil}
For $\mathcal{V} = \mathcal{P}(V^k)$ or $\mathcal{S}(V^k)$
\begin{enumerate}
\item $H^{pq}(\mathfrak{so}(p,q), \SO(p) \times \SO(q); \mathcal{V}) \neq 0$
\item $H^{2pq}(\mathfrak{u}(p,q), \mathrm{U}(p) \times \mathrm{U}(q); \mathcal{V}) \neq 0$.
\end{enumerate}
\end{thm}

In the case of this paper we see that $[\varphi_{0,k} \vol] \in H^n\big(\mathfrak{so}(n,1) ,\SO(n); \mathcal{P}(V^k)\big)$ is non-zero.

We give one more example which uses (1).     Then (choosing $\alpha$ to be the Dirac delta function at the origin of $V$) we have
\begin{thm}\label{generaltheorem}
Let $G$ be a connected linear semisimple Lie group, $K$ a maximal compact subgroup and $n = \dim(G/K)$.  Let $V$ be a finite dimensional representation and $\mathcal{S}(V)$ be the Schwartz space of $V$.
$$H^n(\mathfrak{g}, K; \mathcal{S}(V)) \neq 0.$$
\end{thm}

We conclude with a problem.  Let $G$ and $\mathcal{V}$ be as above Theorem \ref{generaltheorem}.  
Then the action of $\C$ on $\mathcal{V}$ by vector space multiplication induces an action of the $\C$-algebra $H^{\bullet}(\mathfrak{g}, K;\C)$ on $H^{\bullet}(\mathfrak{g}, K; \mathcal{V})$. 
\smallskip

\par{\bf Problem.}
Describe the structure of $H^{\bullet}(\mathfrak{g}, K; \mathcal{V})$  as a  $H^{\bullet}(\mathfrak{g}, K;\C)$-module.

\section{The extension of the theorem to the two-fold cover of $\mathrm{O}(n,1)$ }

It is important to give the analogues of the above results when we replace the connected group $\SO_0(n,1)$ by the covering group $\widetilde{\mathrm{O}(n,1)}$ (with four components) of $\OO(n,1)$ and hence the maximal compact $\SO(n)$ in $\SO_0(n,1)$ by the maximal compact subgroup $\widetilde{K} = \widetilde{\mathrm{O}(n) \times \mathrm{O}(1)}$ of $\widetilde{\mathrm{O}(n,1)}$.  Here  $\widetilde{\mathrm{O}(n,1)}$ denotes the total space of the restriction to $\mathrm{O}(n,1)$ of the pull-back of the metaplectic cover of $\Sp( 2k(n+1),\R )$ under the inclusion of the dual pair $\mathrm{O}(n,1) \times \Sp(2k,\R)$ into $\Sp( 2k(n+1),\R)$. Let $\varpi_k$ be the restriction of the Weil representation of $\mathrm{Mp}( 2k(n+1),\R)$ to $\widetilde{\mathrm{O}(n,1)}$ under the embedding $\widetilde{\mathrm{O}(n,1)} \to \mathrm{Mp} (2k(n+1),\R)$.  The following lemma is a consequence of the result of Section 4 of \cite{BMM2}). We believe it is more enlightening to state the following lemma in terms of a general orthogonal group.  Note that the required results for $\mathrm{O}(p,q)$ follow from those of \cite{BMM2} for $\mathrm{U}(p,q)$ by restriction. 

\begin{lem}\label{Weildescends}
\hfill

\begin{enumerate}
\item
The central extension $\widetilde{\mathrm{O}(p,q)} \to \mathrm{O}(p,q)$ is the pull-back under $\det_{\OO(p,q)}^k: \OO(p,q) \to \C^{\ast}$ of the twofold extension $\C^{\ast} \to \C^{\ast}$ given by taking the square. Hence, the group $\widetilde{\mathrm{O}(p,q)}$, has the character  $\det_{\mathrm{O}(p,q)}^{k/2}$, the square-root of $\det_{\mathrm{O}(p,q)}^k$. 
\item
The character ${\det}_{\mathrm{O}(p,q)}^{k/2}$  is ``genuine'' (does not descend to the base of the cover)  if and  only if $k$ is odd.
\item For both even and odd $k$, the twisted Weil representation $\varpi_k \otimes  {\det}_{\mathrm{O}(p,q)}^{k/2}$ descends to $\mathrm{O}(p,q)$.
\item The induced action of $K = \OO(p) \times \OO(q)$ by $\varpi_k \otimes {\det}_{\mathrm{O}(p,q)}^{k/2}$ on the vaccuum vector $\psi_0$ (the constant polynomial $1$) in the Fock model $\mathrm{Pol}((V \otimes \C)^k)$ is given by
\begin{equation*}
\big(\varpi_k \otimes {\det}_{\mathrm{O}(p,q)}^k\big)(k_+, k_-) \big(\psi_0) = {\det}_{\mathrm{O}(q)}(k_-)^k \psi_0.
\end{equation*}
\end{enumerate}
\end{lem}
Applying  items (1),(2),(3) and (4) to the  case in hand, we obtain 
\begin{prop} \label{actionofK}
The action of $K = \OO(n) \times \OO(1)$ on $\mathrm{Pol}((V \otimes \C)^k)$ under  the restriction of the Weil representation twisted by ${\det}_{\mathrm{O}(n,1)}^{k/2}$ is given by 
\begin{equation*}
\big(\varpi_k \otimes {\det}_{\OO(n,1)}^k\big)(k_+, k_-) \big(\varphi \big)(\mathbf{v})  = {\det}_{\OO(1)}(k_-)^k \  \varphi(k_+^{-1} k_-^{-1} \mathbf{v}).
\end{equation*}
\end{prop}

The cohomology groups of interest to us now are the groups
$$H^\ell\big(\mathfrak{so}(n,1) ,\widetilde{K}; \mathrm{Pol}((V \otimes \C)^k) \otimes \det^{k/2}\big).$$

The goal in this subsection is to prove
\begin{thm}\label{maintwo}
\hfill
\begin{enumerate}
\item If $k < n$,
\begin{equation*}
H^\ell \big(\mathfrak{so}(n,1) , \widetilde{\mathrm{O}(n) \times \mathrm{O}(1)}; \rm{Pol}((V \otimes \C)^k)\otimes {\det}^{\frac{k}{2}} \big) = \begin{cases}
\mathcal{R}_k \varphi_k &\text{ if } \ell = k \\
0 &\text{ otherwise.}
\end{cases}
\end{equation*}
\item  If $k=n$,
\begin{equation*}
H^{\ell} \big(\mathfrak{so}(n,1) , \widetilde{\mathrm{O}(n) \times \mathrm{O}(1)}; \mathrm{Pol}((V \otimes \C)^k)\otimes {\det}^{\frac{k}{2}} \big) =
\begin{cases}
\mathcal{R}_n \varphi_n &\text{ if } \ell = n \\
0 &\text{ otherwise.}
\end{cases}
\end{equation*}
%\item Items $\mathrm{(1)}$ and $\mathrm{(2)}$ together imply that
%$$H^n \big(\mathfrak{so}(n,1) ,\widetilde{\mathrm{O}(n) \times \mathrm{O}(1)}; \rm{Pol}((V \otimes \C)^k)\otimes {\det}^{\frac{k}{2}} \big)= \begin{cases} 0, \ k \neq n\\ \mathcal{R}_n \varphi_n,\  k = n.\end{cases}$$
\item If $k > n$,
\begin{equation*}
H^{\ell} \big(\mathfrak{so}(n,1) , \widetilde{\mathrm{O}(n) \times \mathrm{O}(1)}; \mathrm{Pol}((V \otimes \C)^k)\otimes {\det}^{\frac{k}{2}} \big) =
\begin{cases}
\text{nonzero} &\text{ if } \ell = n \\
0 &\text{ otherwise.}
\end{cases}
\end{equation*}
\end{enumerate}
% \smallskip
\end{thm}

%Theorem \ref{maintwo} will be a consequence of the following proposition and  two lemmas.   We have (see Section 4 of \cite{BMM2})
%\begin{lem} \label{twistdescends}
%\begin{enumerate}
%\item For both even and odd $k$, the Weil representation twisted by $\det^{k/2}$ descends to $\mathrm{O}(n,1)$.
%\item Let  $K = \mathrm{O}(n) \times \mathrm{O}(1)$. Then the action of $K = \OO(n) \times \OO(1)$ on $\mathrm{Pol}((V \otimes \C)^k)$ induced by the Weil representation twisted by $\det^{k/2}$ is given by 
%\begin{equation*}
%(k_+, k_-) \varphi(\mathbf{v}) = {\det}_{\OO(1)}(k_-)^k \varphi(k_-^{-1} k_+^{-1}. \mathbf{v})
%\end{equation*}
%\end{enumerate}
%\end{lem}

We now prove Theorem \ref{maintwo}. 
Put $K = \OO(n) \times \OO(1)$.  Then from (3) of Lemma \ref{Weildescends} the restriction of the twisted Weil representation of $\widetilde{K}$ descends to $K$ and we have

\begin{equation} \label{twistedactioncohomologygroups}
C^\ell\big(\mathfrak{so}(n,1) ,\widetilde{K}; \mathrm{Pol}((V \otimes \C)^k) \otimes {\det}^{k/2}\big) = C^\ell\big(\mathfrak{so}(n,1) ,K; \mathrm{Pol}((V \otimes \C)^k) \otimes {\det}^{k/2}\big).
\end{equation}

Recall the notation $C^\ell(\mathrm{Pol}((V \otimes \C)^k)) = C^\ell(\mathfrak{so}(n,1), \SO(n); \mathrm{Pol}((V \otimes \C)^k) )$.  Note $(\mathrm{O}(n) \times \mathrm{O}(1))/ \SO(n)\cong \Z/2 \times \Z/2$ and apply Proposition \ref{actionofK} to the right-hand side of  Equation \eqref{twistedactioncohomologygroups}  to obtain 
\begin{equation*}
C^\ell(\mathfrak{so}(n,1), K; \mathrm{Pol}((V \otimes \C)^k) \otimes {\det}^{k/2} ) = \big( C^\ell( \mathrm{Pol}((V \otimes \C)^k) ) \otimes {\det}_{\OO(1)}^k \big)^{\Z/2 \times \Z/2}.
\end{equation*}
On the right-hand side of the above equation, we have extended the action of $\SO(n)$ on $(V \otimes \C)^k$ to the action of $\OO(n)$ given by
$$k \varphi(\mathbf{v}) = \varphi(k^{-1} \mathbf{v}).$$

Now recall that $C^\ell( \mathrm{Pol}((V \otimes \C)^k) = C_+^\ell \oplus C_-^\ell$ and hence
$$C^\ell( \mathrm{Pol}((V \otimes \C)^k \otimes {\det}_{\OO(1)}^k)^{\Z/2 \times \Z/2} = (C_+^\ell \otimes {\det}_{\OO(1)}^k)^{\Z/2 \times \Z/2} \oplus (C_-^\ell \otimes {\det}_{\OO(1)}^k)^{\Z/2 \times \Z/2}.$$
Since the element $(1,0) \in \Z/2 \times \Z/2$ acts by the element $\iota \otimes \iota$ (see Equation \eqref{iotadef}) and $C_-$ is defined to be the $-1$ eigenspace of the action of $\iota \otimes \iota$, we have $C_-^{\Z/2 \times \Z/2} = 0$ and hence $(C_- \otimes {\det}_{\OO(1)}^k)^{\Z/2 \times \Z/2} = 0$.  Hence
\begin{equation*}
C^\ell(\mathfrak{so}(n,1), K; \mathrm{Pol}((V \otimes \C)^k) \otimes {\det}^{k/2} ) = (C_+^\ell \otimes {\det}_{\OO(1)}^k)^{\Z/2 \times \Z/2}.
\end{equation*}
Hence we have
\begin{equation*}
H^\ell\big(\mathfrak{so}(n,1) ,\widetilde{K}; \mathrm{Pol}((V \otimes \C)^k) \otimes {\det}^{k/2}\big) = (H^\ell(C_+) \otimes {\det}_{\OO(1)}^k)^{\Z/2 \times \Z/2}.
\end{equation*}
%The cochain groups on the right of Equation \eqref{twistedactioncohomologygroups} are obtained from those of Theorem \ref{main} by taking invariants with respect to the component group $(\mathrm{O}(n) \times \mathrm{O}(1))/ \SO(n)\cong \Z/2 \times \Z/2$ with the action on the coefficients given by $1 \boxtimes \det_{\OO(1)}^k$.  

\begin{rmk} \label{phiko(1)}
Note that $\varphi_1$ is invariant under $\OO(n)$ and transforms under $\OO(1)$ by ${\det}_{\OO(1)}$. Since $\varphi_k$ is the $k$-fold exterior wedge of $\varphi_1$ with itself, it follows that  $\varphi_k$ is invariant under $\OO(n)$ and transforms under $\OO(1)$ by ${\det}_{\OO(1)}^k$.  This is the reason for twisting the Fock model by ${\det}^{k/2}$.
\end{rmk}

\begin{lem}
$$H^\ell(C_+) = (H^\ell(C_+) \otimes {\det}_{\OO(1)}^k)^{\Z/2 \times \Z/2}.$$
\end{lem}

\begin{proof}
By Theorem \ref{K+vanCvan},
$$H^\ell(C_+) = \begin{cases} \mathcal{R}_k \varphi_k &\text{ if }\ell = k \\
0 &\text{ otherwise. } \end{cases}$$
By Remark \ref{phiko(1)}, $\varphi_k$ transforms by ${\det}_{\OO(1)}^k$ and hence $H^k(C_+)$ transforms by ${\det}_{\OO(1)}^k$ and the lemma follows.

\end{proof}

As an immediate consequence of the previous lemma we have 

\begin{equation*}
H^\ell \big(\mathfrak{so}(n,1) ,\mathrm{O}(n) \times \mathrm{O}(1); \rm{Pol}((V \otimes \C)^k)\otimes {\det}^{\frac{k}{2}} \big) =
\begin{cases}
\mathcal{R}_k \varphi_k \text{ if } \ell = k\\
0 \text{ otherwise.}
\end{cases}
\end{equation*}
and 
statements (1) and (2) of Theorem \ref{maintwo} follow.  

We now prove statement (3).  Thus, we have $k > n$. The vanishing part (for $\ell <n$) of statement (3) follows from the vanishing statement in Theorem \ref{kgeqnvanishing}.  It remains to prove nonvanishing in degree $n$. To this end, we must exhibit classes in $H^n(A)$ which, once twisted by $\det^\frac{k}{2}$, are $\OO(n) \times \OO(1)$-invariant.  By Proposition \ref{Avanishing} we have (this is before we take invariance),
\begin{equation*}
H^\ell(A) = \begin{cases}
0 &\text{ if } \ell \neq n \\
\mathcal{P}_k/(q_1, \ldots, q_n) \vol &\text{ if } \ell = n.
\end{cases}
\end{equation*}
It remains to find a nonzero element which is invariant.  First, a definition and a lemma.  Define ${\det}_+$ to be the determinant of the upper-left $(n \times n)$-block of coordinates.  That is, ${\det_+}$ is the determinant of the $n \times n$-matrix $(z_{\alpha i})_{1 \leq \alpha, i \leq n}$.

\begin{lem} \label{injecttoA}
The map from $\C[w_{n+1}, \ldots, w_k]$ to $\mathcal{P}_k/(q_1, \ldots, q_n)$ sending $f$ to $f {\det}_+$ is an injection.
\end{lem}

\begin{proof}
We will show that this map has kernel zero.  Suppose $f(w_{n+1}, \ldots, w_k) {\det}_+ \in (q_1, \ldots, q_n)$.   Now pass to the quotient where we have divided by the ``off-diagonal" $z_{\alpha i}$.  Then
\begin{align*}
f(w_{n+1}, \ldots, w_k) {\det}_+ \mapsto f(w_{n+1}, \ldots, w_k) z_{11} z_{22} \cdots z_{nn}
q_\alpha \mapsto w_\alpha z_{\alpha \alpha}.
\end{align*}
By assumption, $f \prod_\alpha z_{\alpha \alpha} \in (w_1 z_{11}, w_2 z_{22}, \ldots, w_n z_{nn}) \subset (w_1, \ldots, w_n)$.
Hence $f(w_{n+1}, \ldots, w_k) \prod_\alpha z_{\alpha \alpha} \in (w_1, \ldots, w_n)$ and thus $f = 0$.

\end{proof}

By Lemma \ref{injecttoA} the cohomology classes $f(w_{n+1}, \ldots, w_k) {\det}_+ \vol \in H^n(A)$ are nonzero.  We now check if they are invariant.

Let $f$ be homogenous of degree $a$.  Then for $(k_+, k_-) \in \OO(n) \times \OO(1)$ we have, by Lemma \ref{Weildescends},
$$(k_+ k_-) f {\det}_+ \vol = \det(k_-)^{a+k+n} f {\det}_+ \vol.$$
Thus, if $a+k+n$ is even this element is invariant and Statement (3) is proved.  In fact, we have shown

\begin{prop}
For $k > n$, if $k+n$ is even (resp. odd), then the even (resp. odd) degree polynomials $f \in \C[w_{n+1}, \ldots, w_k]$ inject into\\
$H^n \big(\mathfrak{so}(n,1) , \widetilde{\mathrm{O}(n) \times \mathrm{O}(1)}; \mathrm{Pol}((V \otimes \C)^k)\otimes {\det}^{\frac{k}{2}} \big)$ by the map $f \mapsto f det_+ \vol.$
\end{prop}

%Suppose $n$ and $k$ have the same parity.  Define $\det_+$ to be the determinant of the upper-left $(n \times n)$-block of coordinates.  Then we claim that $\det_+ \otimes \vol$ is invariant under $\OO(n)$. Indeed, $\det_+$ transforms by $\det_{\OO(1)}^k$ and $\vol$ transforms by $\det_{\OO(1)}^n$.  Hence this element is invariant.  Now we must show this is not in the ideal generated by the $q_\alpha$.  But $(q_1, \ldots, q_n) \subset (w_1, \ldots, w_k)$ and $\det_+ \notin (w_1, \ldots, w_k)$.

%Now suppose $n$ and $k$ have opposite parity.  Then $w_{n+1} \det_+ \otimes \vol$ is invariant under $\OO(n)$, $w_{n+1} \det_+$ transforms by $\det_{\OO(1)}^{k+1}$, and $\vol$ transforms by $\det_{\OO(1)}^n$.  Hence this element is invariant.  Now we must show this is not in the ideal generated by the $q_\alpha$.  Suppose $w_{n+1} \det_+ \in (q_1, \ldots, q_n)$.  Now pass to the quotient where we have divided by the ``off-diagonal" $z_{\alpha i}$.  Then
%$$w_{n+1} {\det}_+ \mapsto w_{n+1}\prod_\alpha z_{\alpha \alpha}.$$
%By assumption, $w_{n+1} \prod_\alpha z_{\alpha \alpha} \in (\bar{q_1}, \ldots, \bar{q_n}).$  But $(\bar{q_1}, \ldots, \bar{q_n}) \subset (w_1, \ldots, w_n)$ and clearly $w_{n+1} \prod_\alpha z_{\alpha \alpha} \notin (w_1, \ldots, w_n)$.  We have a condtradiction and hence the class is non-zero.

\newpage

\part{The computation of $H^{\bullet}\big(\mathfrak{so}(n,1), \SO(n); L^2(V^k) \big)$}

The purpose of Part 2 of this paper is to prove the following theorem

\begin{thm} \label{L^2mainNicolas}
\hfill

\begin{enumerate} 
\item Assume that $k >n/2$ and $\ell \neq \frac{n}{2}, \frac{n+1}{2}$. Then we have:
$$H^\ell (\mathfrak{so}(n,1) , \mathrm{SO} (n) ; L^2 (V^k)) = \overline{H}^\ell (\mathfrak{so}(n,1) , \mathrm{SO} (n) ; L^2 (V^k)) = 0.$$

\item Suppose now $k\leq \frac{n}{2}$.  Then for $\ell \neq \frac{n+1}{2}$ we have $H^{\ell} (\mathfrak{so}(n,1) , \mathrm{SO} (n) ; L^2 (V^k))$ is reduced.  Furthermore, $H^{\ell} (\mathfrak{so}(n,1) , \mathrm{SO} (n) ; L^2 (V^k))$ is non-zero if and only if $\ell = k$ or $\ell = n-k$.

\item Suppose $n = 2m$.  Then $H^m (\mathfrak{so}(n,1) , \mathrm{SO} (n) ; L^2 (V^k))$ is reduced.  Furthermore, $H^m (\mathfrak{so}(n,1) , \mathrm{SO} (n) ; L^2 (V^k))$ is non-zero if and only if $k \geq m$.

%\item Suppose $n = 2m+1$.  Then
%\begin{equation*}
%\overline{H}^m(\mathfrak{so}(n,1), \SO(n); L^2(V^k)) \cong \overline{H}^{m+1} (\mathfrak{so}(n,1), \SO(n); L^2(V^k)).
%\end{equation*}
%Furthermore, $\overline{H}^m(\mathfrak{so}(n,1), \SO(n); L^2(V^k))$ is non-zero if and only if $k=m$.

\item Suppose $n = 2m+1$.  Then $H^m(\mathfrak{so}(n,1), \SO(n); L^2(V^k))$ is reduced and $\overline{H}^m(\mathfrak{so}(n,1), \SO(n); L^2(V^k))$ is non-zero if and only if $k=m$.  Furthermore, $H^{m+1} (\mathfrak{so}(n,1), \SO(n); L^2(V^k))$ is reduced if and only if $k \leq m$ and
\begin{enumerate}
\item $H^{m+1} (\mathfrak{so}(n,1), \SO(n); L^2(V^k)) \neq 0$ if and only if $k \geq m$
\item $\overline{H}^{m+1} (\mathfrak{so}(n,1), \SO(n); L^2(V^k)) \neq 0$ if and only if $k = m.$
\end{enumerate}

\end{enumerate}
\end{thm}

%Here we denote by $\overline{H}^\ell (\mathfrak{so}(n,1) , \mathrm{SO} (n) ; L^2 (V^k))$ the {\it reduced} cohomology.

%\begin{rem}  The Hodge star operator induces an isomorphism
%$$*: H^{\ell} (\mathfrak{so}(n,1) , \mathrm{SO} (n) ; L^2 (V^k)) \to H^{n-\ell} (\mathfrak{so}(n,1) , \mathrm{SO} (n) ; L^2 (V^k))$$
%\end{rem}

%Let $\mathbf{x} \in V^k$. 
%In Propostion \ref{Mateisprop2} of this section we will need the direct integral decompositions of   $L^2(G_{\mathbf{x}} \backslash G)$.  We have
%\begin{thm}
%There exists a projection-valued measure $\mu_{\mathbf{x}}$ on the space of unitary representations $\widehat{G}$  of $G$ such that
%$$L^2(G_{\mathbf{x}} \backslash G) \cong \int_{\widehat{G}} d \mu.$$
%\end{thm}
%We define a finite subset of $\widehat{G}$ by  $\mathcal{A}_{\ell} = \{A_{\ell}, \ell \in \mathbb{N}\}$ 
%is the set of unitary representations of $G$ with cohomology in degree $\ell$. 

\section{Smooth vectors in $L^2(V)$}
In this section we will derive some properties of smooth vectors in $L^2(V)$ for the action of a Lie group $G$ induced by a linear, measure-preserving action on $V$.

We begin with some definitions.  Suppose a semi-simple Lie group $G$, with Lie algebra $\mathfrak{g}$ and maximal compact group $K$, acts continously on a Hilbert space $\mathcal{V}$.  Then we use the symbols $C^\ell(\mathfrak{g}, K; \mathcal{V})$ (resp. $H^\ell(\mathfrak{g}, K; \mathcal{V})$) to be $C^\ell(\mathfrak{g}, K; \mathcal{V}^\infty)$ (resp. $H^\ell(\mathfrak{g}, K; \mathcal{V}^\infty)$) where $\mathcal{V}^\infty \subset \mathcal{V}$ is the subspace of smooth vectors, see Bump \cite{Bump}.  To justify this definition, see Theorem 2.7 of Borel \cite{Borel2}.  In this theorem, Borel shows that if $\overline{C}$ is the complex obtained by taking the closure of $d$ on $C^\bullet(\mathfrak{g}, K; \mathcal{V}^\infty)$, then the inclusion $C^\bullet(\mathfrak{g}, K; \mathcal{V}^\infty) \rightarrow \overline{C}$ induces an isomorphism of cohomology.    In this paper, $\mathcal{V}$ will be $L^2(X)$ where $X$ is either $V^k$, an open cone in $V^k$, or an orbit of $G$ in $V^k$.  Then the smooth vectors have the property that their restriction to almost every orbit is a smooth square integrable function by Lemma \ref{Fubini}. 

%Let $\mathbf{x} \in V^k$, then $L^2(G\mathbf{x})^\infty$ is the set of smooth functions $f$ on $G\mathbf{x}$ such that all derivatives are square integrable.

\begin{lem}\label{Fubini}
Suppose $\alpha \in C^{\ell}(L^2(V))$.  Then for almost every $x \in V$ the cochain $\alpha$ admits a measurable and square integrable restriction to $Gx$  giving rise to an element again denoted $\alpha \in  C^{\ell}( L^2(Gx))$.
\end{lem}

\begin{proof}
The set $\{ \omega_I \in \wedge^{\ell}(\mathfrak{p}^*): \ I \in \mathcal{S}_{\ell,n} \}$ is a basis of $\wedge^{\ell}(\mathfrak{p}^*)$ where $\{ \omega_i: 1 \leq i \leq n \}$ is a basis for $\mathfrak{p}^*$.  We may write
$$\alpha = \sum\limits_{I} f_I \omega_I.$$
The $L^2$-norm of $\alpha$ is the integral over $V$ of $\sum\limits_{I} f^2_I$.  Hence we have $f_I \in L^2(V)$, for all $I$.  Fix $I$.  Since $f_I$ is measurable and the measure on $V$ is a product measure, by \cite{Rudin}, Theorem 7.12, pg. 145, the restriction to almost every orbit is  measurable (see the discussion following Lemma 2).  We may then  apply Fubini's Theorem to $f_I^2$ to conclude that for almost every orbit the restriction of $f_I$ to that orbit is square-integrable. The lemma follows by intersecting the domains of the $f_I$'s.

\end{proof}

Suppose now $V$ is a finite dimensional vector space, $f$ a smooth vector in $L^2(V)$ and $x \in \mathfrak{g}$.  Let $X$ be the associated vector field on $V$, that is, $X$ acts on a vector $v \in V$ by
\begin{equation*}
X(v) = \frac{d}{dt}|_{t=0} exp(-tx) v.
\end{equation*}

\begin{lem}
Suppose $f \in L^2(V)$ is a smooth vector and fix $x \in \mathfrak{g}$.  Define a smooth vector $p$ by
\begin{equation*}
p = \frac{d}{dt}|_{t=0} exp(-tx)f
\end{equation*}
where this limit is taken using the $L^2$ norm on $V$.  Then
\begin{enumerate}
\item for almost every $v \in V$, $f(exp(-tx)v)$ is a smooth function in $t$, hence $Xf$ is a well defined smooth vector in $L^2(V)$
\item $Xf = p$
\item let $v \in V$ and suppose $f$ is well defined on the orbit $Gv$.  Then $f \in C^\infty(Gv)$ and consequently we have
\begin{equation*}
(Xf)|_{Gv} = X (f|_{Gv}).
\end{equation*}
\end{enumerate}
\end{lem}

\begin{proof}
Since $f$ is a smooth vector, by the theorem of Dixmier Malliavin, \cite{DM}, \cite{Ca}, we have $f$ is a finite sum of convolutions with compactly supported smooth functions $\varphi_j$, that is
\begin{equation*}
f = \sum_j \varphi_j* h_j
\end{equation*}
where
\begin{equation*}
(\varphi_j* h_j)(v) = \int_G \varphi_j(g) h_j(g^{-1}v) dg.
\end{equation*}

In what follows it suffices to consider to consider the case
\begin{equation*}
f = \varphi* h
\end{equation*}
where $\varphi$ is a compactly supported smooth function on $G$ and $h$ a smooth vector in $L^2(V)$.  Then
\begin{equation*}
Xf = (X \varphi)* h
\end{equation*}
and (1) is proved.

We will prove the analogue of (2) on the group $G$ with $L^2$ replaced with $L^1$.  Define $\tau_t$ by
\begin{equation*}
\tau_t(\varphi)(g) = \varphi(exp(-tx)g).
\end{equation*}
Let $\psi = X \varphi$ and define $\varphi_t \in L^1(G)$ by
\begin{equation*}
\varphi_t = \frac{\tau_t \varphi - \varphi}{t} - \psi.
\end{equation*}
We want to show that $\lim_{t \rightarrow 0} \varphi_t = 0$ in $L^1(G)$.  Since $\varphi$ is smooth on $G$, the pointwise limits of the one-paramter family of functions are zero everywhere, that is
\begin{equation*}
\lim_{t \rightarrow 0} \varphi_t(g) = \lim_{t \rightarrow 0} [\frac{\tau_t \varphi(g) - \varphi(g)}{t} - \psi(g)] = 0 \text{ for all $g$}.
\end{equation*}

Since $\varphi$ has compact support, $Y$, we may assume the family $\varphi_t$ is supported in the compact set $C = \{exp(tx)y : t \in [-1, 1], y \in Y \}$.  If $M$ is an upper bound for the values of $\psi$, then by the mean value theorem we have
$$|\frac{\tau_t \varphi - \varphi}{t} - \psi| \leq 2M.$$
Hence the family $\varphi_t$ is bounded by $2M$ on the compact set $C$ and is zero elsewhere.  Then the above limit converges in $L^1$ by the Lebesgue Dominated Convergence Theorem.

Now we want to prove
$$Xf = \lim_{t \rightarrow 0} \frac{\tau_t(f) - f}{t} = p$$
where the limit is taken in $L^2$.  Note
$$||\frac{\tau_t(f) - f}{t}||_{L^2(V)} = ||\varphi_t* h||_{L^2(V)}.$$
By Minkowski's Inequality we have
\begin{align*}
||\varphi_t* h||_{L^2(V)} &= \sqrt{ \int_{V} \bigg( \int_G \varphi_t(g) h(g^{-1}v) dg \bigg)^2 dv} \\
&\leq \int_G \sqrt{ \int_V \bigg( \varphi_t(g) h(g^{-1}v) \bigg)^2 dv} dg \\
&= \int_G |\varphi_t(g)| dg \sqrt{ \int_V h^2(w) dw} = ||\varphi_t||_{L^1(G)} ||h||_{L^2(V)}
\end{align*}
where the last line is obtained by making the change of variable $w = g^{-1} v$ and noting that the action of $G$ preserves measure.  Since $\varphi_t \rightarrow 0$ in $L^1(G)$, we have the convolution $\varphi_t * h \rightarrow 0$ in $L^2(V)$.

\end{proof}

As a consequence of the two previous lemmas, we obtain

\begin{prop} \label{restriction}
Let $\alpha$ be a smooth vector in $C^\ell(L^2(V^k))$.  Then for $\mathbf{x} \in V^k$ with $\alpha|_{G\mathbf{x}}$ smooth we have
\begin{equation*}
(d \alpha)|_{G \mathbf{x}} = d (\alpha|_{G\mathbf{x}}).
\end{equation*}
\end{prop}

Note that by Lemma \ref{Fubini}, almost every $\mathbf{x} \in V^k$ satisfies the conditions for Proposition \ref{restriction}.

\section{Notation}

We define the Gram matrix of $\mathbf{x}$ to be the $k$ by $k$ matrix defined by
$$(\mathbf{x}, \mathbf{x}) = ( (x_i, x_j) ).$$
When $\mathbf{x}$ is clear, we will sometimes abbreviate $(\mathbf{x}, \mathbf{x})$ by $\beta$.

If $k \leq n$ we define the (open) cones $\Omega_{k-1,1}$ and  $\Omega_{k,0}$ in $V^k$ by 
\begin{align*}
\Omega_{k-1,1} =& \{\mathbf{x} \in V^k : (\mathbf{x}, \mathbf{x}) \text{ has signature } (k-1, 1) \} \\
\Omega_{k,0} =& \{\mathbf{x} \in V^k : (\mathbf{x}, \mathbf{x}) \text{ is positive definite}\}.
\end{align*}
Since $V^k - (\Omega_{k-1,1} \cup \Omega_{k,0})$ has measure zero, we have the following Hilbert sum decomposition:
\begin{equation}
L^2(V^k) = L^2(\Omega_{k,0}) \oplus L^2(\Omega_{k-1,1}).
\end{equation}
If the value of $k$ is clear we will abbreviate $\Omega_{k,0}$ to $\Omega_+$ and $\Omega_{k-1,1}$ to $\Omega_-$.

For the rest of this section we will use $\Omega$ to denote either of the two cones $\Omega_+$ or $\Omega_-$.  It follows from the transitivity of the action of $G \times \GL(k,\R)$ on $\Omega$ that the invariant densities $\vol_{\Omega}$ on the cones are of the form function of $\beta$ times the tensor product of the two invariant densities on the two factors.  By Proposition \ref{restriction} we may restrict cocycles on $\Omega$ to cocycles on orbits $G \mathbf{x}$ for almost every $\mathbf{x} \in \Omega$.  We use this throughout the rest of the paper without further comment.

Throughout Part 2, it will be useful to have notation for the space and time coordinates of an element $\mathbf{x} \in V^k$.  First, given $x \in V$ we define $(x)_0$ to be the projection of $x$ onto $V_+$
$$(x)_0 = x + (x, e_{n+1}) e_{n+1}.$$
We extend this definition to $V^k$ by defining $\mathbf{x}_0$ to be the projection onto $V_+^k$.  That is, given $\mathbf{x} = (x_1, \ldots, x_k)$, we define
$$\mathbf{x}_0 = ((x_1)_0, \ldots, (x_k)_0).$$

More generally, for $x,z \in V$ with $(z,z)<0$, we define $(x)_z$, the projection onto $z^\perp$, by
$$(x)_z = x + (x, z)z.$$
We extend this definition to $V^k$ by defining $\mathbf{x}_z$ to be the projection onto $(z^\perp)^k$.  That is, given $\mathbf{x} = (x_1, \ldots, x_k)$, we define
$$\mathbf{x}_z = ((x_1)_z, \ldots, (x_k)_z).$$

We define the vector $\mathbf{t} \in \R^k$ by
$$\mathbf{t} = (t_1, \ldots, t_k) \text{ where } t_i = -(x_i, e_{n+1}), 1 \leq i \leq k.$$

If $k > n$ there is only one generic cone $\Omega_{n,1}$ in the space $V^k$. It is the open set where the $k$-tuple $\mathbf{x}=  (x_1, \ldots ,x_k)$ has maximal rank namely $n+1$.  
The Gram matrix of $\mathbf{x}$ will be similar to the diagonal matrix with entries $1,...,1, -1, 0,...,0$  ($n$ ones and $k-(n+1)$ zeroes). Then 
\begin{equation}
L^2 (V^k ) = L^2 (\Omega_{n,1} ).
\end{equation}
%Here again, if the value of $k$ is clear we will abbreviate $\Omega_{n,1}$ to $\Omega_-$. 

We will explain how to deal with the case $k>n$ later. For the time being and until further notice we will suppose that $k \leq n$. 

%Let $\beta: V^k \to \mathrm{Sym}_k$ be the map given by
%$$\beta(\mathbf{x}) = \big( (x_i, x_j) \big).$$
Given a symmetric $k$ by $k$ matrix $\beta \in \mathrm{Sym}_k$ we define the sphere $\mathbb{S}_{k, \beta}(V)$ of radius $\beta$ by 
$$\mathbb{S}_{k, \beta}(V) = \{\mathbf{x} \in V^k : (\mathbf{x}, \mathbf{x}) = \beta \}.$$
In what follows we will consider only the case in which $\beta$ is nondegenerate, hence has signature either $(k,0)$ or $(k-1,1)$.   We abbreviate $\mathbb{S}_{k, I_k}(V)$ by $\mathbb{S}_{k}(V)$ and $\mathbb{S}_{k, I_{k-1,1}}(V)$ by $\mathbb{S}_{k-1,1}(V)$.

Choose $\mathbf{x} \in \mathbb{S}_{k, \beta}(V)$.  Let $G_{\mathbf{x}}$ be the subgroup of $G$ that fixes $\mathbf{x}$.   Hence if $(\mathbf{x}, \mathbf{x})$ is definite we have 
$$G_{\mathbf{x}} \cong \SO(n-k,1)$$
and if $(\mathbf{x}, \mathbf{x})$ is indefinite we have 
$$G_{\mathbf{x}} \cong \SO(n-k+1).$$
The map $f: G \to \mathbb{S}_{k, \beta}(V)$ given by 
$$f(g) = g^{-1} \mathbf{x}$$
induces a $G$-equivariant isomorphism also denoted $f$ 
$$f: G_{\mathbf{x}}\backslash G \to \mathbb{S}_{k, \beta}(V).$$
In case $(\mathbf{x}, \mathbf{x})$ has signature $(k,0)$ choose a cocompact torsion-free discrete subgroup $\Gamma_{\mathbf{x}}$ of $G_{\mathbf{x}}$.  In case $(\mathbf{x}, \mathbf{x})$ has signature $(k-1,1)$  define $\Gamma_{\mathbf{x}}$ to be the trivial group. For the rest of this section we will often abbreviate $G_{\mathbf{x}}$ to $H$ and $\Gamma_{\mathbf{x}}$ to $\Gamma_H$.  

We emphasize that we are abusing notation here and $H$ depends on $\mathbf{x}$ as do the transfer maps $T_{k-1,1}$ and $T_{k,0}$ and the tube $E$ that follow.

Let $\mathbb{P}_k$ be the space of positive definite symmetric $k$ by $k$ matrices. We note that the space $\mathbb{P}_{k-1,1}$ of  symmetric $k$ by $k$ matrices of signature $(k-1,1)$ is diffeomorphic to the product $\mathbb{P}_{k-1} \times D_k$ where $D_k$ is hyperbolic $k$-space.  We have fiber bundle maps
$$p_+: \Omega_+ \to \mathbb{P}_k$$
and 
$$p_-: \Omega_- \to \mathbb{P}_{k-1,1} $$
with fibers the orbits of $G$ on the respective cones.

Both fiber bundles have contractible bases so they are trivial. We give an explicit trivialization $F_+: \mathbb{S}_k \times  \mathbb{P}_k \to \Omega_+$ by
\begin{equation} \label{postriv}
F_+( \mathbf{x}, \beta) = \mathbf{x} \sqrt{\beta}.
\end{equation}

We can similarly give an explicit trivialization $F_-: \mathbb{S}_{k,k-1} \times \mathbb{P}_{k-1,1}  \to \Omega_-$.
% by
%\begin{equation}\label{negtriv}
%F_-( \mathbf{x}, \beta) = \mathbf{x} \sqrt{\beta}
%\end{equation}
%where we abusively denote by $\sqrt{\beta}$ the unique matrix s.t. 

\section{The transfer maps from cochains with values in $L^2(H \backslash G)$ to $H$-invariant forms on hyperbolic space} \label{transfersection}

In this section we will  construct $G$-equivariant  isomorphisms which will preserve the $L^2$ inner products (up to positive scalar multiples)
\begin{equation}
T_{k,0}: C^\bullet \big( L^2(G_{\mathbf{x}}\backslash G) \big) 
    \rightarrow A^\bullet_{(2)}(\Gamma_{\mathbf{x}} \backslash D)^{G_{\mathbf{x}}}
\end{equation}
for the case in which $\beta= ( (x_i,x_j))$ is positive definite and
\begin{equation}
T_{k-1,1}: C^\bullet \big( L^2(G _{\mathbf{x}}\backslash G)) \big)  \rightarrow A^\bullet_{(2)}(D)^{G_{\mathbf{x}}}
\end{equation}
for the case in which $\beta= ( (x_i,x_j))$ has signature $(k-1,1)$.  The maps $T_{k,0}$ and $T_{k-1,1}$ will be referred to as transfer maps.

We will first construct $T_{k,0}$.  Define the ``tube" $E$ by
$$E = \Gamma_H  \backslash D.$$
Then the above map $f$ induces a map $q : \Gamma_H \backslash G \to H \backslash G$.  Let $p$ be the projection $\Gamma_H \backslash G \rightarrow E$. Note that $p$ and $q$ are fibrations with compact fibers, hence proper maps.  Then we have the following diagram

\begin{equation} \label{diamond}
\xymatrix{
  & \Gamma_H \backslash G \ar[dl]_p \ar[dr]^q \\
  E = \Gamma_H \backslash D \ar[dr] && H \backslash G \ar[dl]\\
  & H \backslash G / K
}\hspace{6em}% adjust to suit
\end{equation}
Recall that if $X, Y$ are topological spaces, $f: X \to Y$ is a proper map and $\mu$ is a measure on $X$, then $f_*(\mu)$ is the measure on $Y$ so that $f_*\mu(\varphi) = \mu(f^*\varphi)$ for all $\varphi \in C^0_c(Y)$.  Let $\vol_X$ be a smooth $G$-invariant density for $X$ equal to one of $\Gamma_H \backslash G$, $\Gamma_H \backslash D$, or $H \backslash G$.  Note that $\vol_X$ is unique up to a constant multiple.  Since $p_*$ (resp. $q_*$) commutes with the action of $G$, $p_* \vol_{\Gamma_H \backslash G}$ (resp. $q_* \vol_{\Gamma_H \backslash G}$) is a $G$-invariant measure.  We obtain

\begin{lem}
There exist constants $c_1, c_2$ so that
$$ p_* \vol_{\Gamma_H \backslash G} = c_1 \vol_{\Gamma_H \backslash D} \text{ and } q_* \vol_{\Gamma_H \backslash G} = c_2 \vol_{H \backslash G}.$$
\end{lem}

Let $\sigma_\ell : K \rightarrow \mathrm{Aut}(\wwedge{\ell} \mathfrak{p}^*)$ be given by $\sigma_\ell(k) = \wwedge{\ell} \mathrm{Ad}^*(k^{-1})|\wwedge{\ell}\mathfrak{p}^*$.  Let $F \in \Hom_K \big( \wwedge{\ell} \mathfrak{p}, C_c^\infty(H \backslash G) \big) = \bigg( \wwedge{\ell} \mathfrak{p}^* \otimes C_c^\infty(H \backslash G) \bigg)^K$.  Then we can write
$$ F = \sum_{I \in \mathcal{S}_{\ell,n}} \omega_I \otimes f_I, \quad f_I \in C_c^\infty(H \backslash G).$$
Hence we may identify $F$ with a map
$$F : H \backslash G \rightarrow \wwedge{\ell}\mathfrak{p}^*$$
satisfying
\begin{equation} \label{Fequivariance}
F(Hgk) = \sigma_\ell(k) F(Hg).
\end{equation}

We will use the symbol $(,)$ for the Riemannian metric on $\wwedge{\ell}\mathfrak{p}^*$ induced by the Riemannian metric on $\mathfrak{p}$.  Then $\{\omega_I : I \in \mathcal{S}_{\ell,n} \}$ is an orthonormal basis.  We define the pointwise squared norm $||F||^2$ of $F$ by
$$||F||^2(Hg) = ( F(Hg), F(Hg) ).$$

Then $||F||^2$ is right-$K$ invariant and descends to a function $F^\prime$ on $H \backslash G / K$.  We note that for $F = \sum_{I \in \mathcal{S}_{\ell,n}} \omega_I \otimes f_I$, we have
$$F^\prime(Hgk) = \sum_{I \in \mathcal{S}_{\ell,n}} |f_I|^2.$$

Now let $\widetilde{F} = q^* F$.  Then $\widetilde{F}$ satisfies
$$\widetilde{F}(\gamma g k) = \sigma_\ell(k) \widetilde{F}(g), \gamma \in \Gamma_H, k \in K.$$
Hence $\widetilde{F}$ descends to an $H$-invariant $\ell$-form $\overline{F}$ on $E = \Gamma_H \backslash D$.

We note that the pointwise squared norm $||\overline{F}||^2$ satisfies
$$||\overline{F}||^2 = \sum_{I \in \mathcal{S}_{\ell,n}} |f_I|^2.$$

Again $||\overline{F}||^2$ descends to a function denoted $\overline{F}^\prime$ on $H \backslash G / K$ and we have
\begin{equation} \label{equalondoublecoset}
\overline{F}^\prime = F^\prime.
\end{equation}

From $\eqref{equalondoublecoset}$ we obtain
\begin{equation} \label{pullbacksofnorms}
p^*||\overline{F}||^2 = q^*||F||^2.
\end{equation}

We define the transfer map $T_{k,0}$ by
$$T_{k,0}(F) = \overline{F}.$$

\begin{prop} \label{comparisonofnorms}
There exists a positive constant $c$ such that
$$||F||_{L^2(H \backslash G)}^2 = c||T_{k,0}(F)||^2_{L^2(E)} = c||\overline{F}||_{L^2(E)}^2.$$
\end{prop}

\begin{proof}
In what follows, $c^\prime$ will denote a generic constant.
\begin{align*}
||\overline{F}||_{L^2(E)}^2 &= c^\prime \int_E ||\overline{F}||^2 p_*( \vol_{\Gamma_H \backslash G}) \\
&= c^\prime \int_{\Gamma_H \backslash G} p^* ||\overline{F}||^2 \vol_{\Gamma_H \backslash G} \\
&= c^\prime \int_{\Gamma_H \backslash G} q^* ||F||^2 \vol_{\Gamma_H \backslash G} \\
&= c^\prime \int_{H \backslash G} ||F||^2 q_*(\vol_{\Gamma_H \backslash G}) = c^\prime ||F||_{L^2(H \backslash G)}^2.
\end{align*}
Put $c = \frac{1}{c^\prime}$.

\end{proof}

Now we will construct $T_{k-1,1}$.  Let $\mathbf{e} = (e_1, \ldots, e_{k-1}, e_{n+1})$, so $(\mathbf{e}, \mathbf{e}) = I_{k-1,1}$.  Now $G_\mathbf{e}$, the isotropy of $\mathbf{e}$, is isomorphic to $\SO(n-k+1)$ embedded in the lower right block of $\SO(n)$.  We will now abbreviate $\SO(n-k+1)$ to $H$.  Now we have the following diagram

\begin{equation} \label{diamond2}
\xymatrix{
  & G \ar[dl]_p \ar[dr]^q \\
  D=G/K \ar[dr] && H \backslash G \ar[dl]\\
  & H \backslash G / K
}\hspace{6em}% adjust to suit
\end{equation}

Let $F \in \Hom_K \big( \wwedge{\ell} \mathfrak{p}, C_c^\infty(H \backslash G) \big) = \bigg( \wwedge{\ell} \mathfrak{p}^* \otimes C_c^\infty(H \backslash G) \bigg)^K$.  Then we can write
$$ F = \sum_{I \in \mathcal{S}_{\ell,n}} \omega_I \otimes f_I, f_I \in C_c^\infty(H \backslash G).$$
Again we may identify $F$ with a map
$$F : H \backslash G \rightarrow \wwedge{\ell}\mathfrak{p}^*$$
satisfying
\begin{equation} \label{Fequivariance2}
F(Hgk) = \sigma_\ell(k) F(Hg).
\end{equation}

Again, we let $\widetilde{F} = q^* F$.  Then $\widetilde{F}$ satisfies
$$\widetilde{F}(h g k) = \sigma_\ell(k) \widetilde{F}(g), h \in H, k \in K.$$
Hence $\widetilde{F}$ descends to an $H$-invariant $\ell$-form $\overline{F}$ on $D$.

We note that the pointwise squared norm $||\overline{F}||^2$ satisfies
$$||\overline{F}||^2 = \sum_{I \in \mathcal{S}_{\ell,n}} |f_I|^2.$$

Again $||\overline{F}||^2$ descends to a function denoted $\overline{F}^\prime$ on $H \backslash G / K$ and we have
\begin{equation} \label{equalondoublecoset2}
\overline{F}^\prime = F^\prime.
\end{equation}

From $\eqref{equalondoublecoset}$ we obtain
\begin{equation} \label{pullbacksofnorms2}
p^*||\overline{F}||^2 = q^*||F||^2.
\end{equation}

We define the transfer map $T_{k-1,1}$ by
\begin{equation}
T_{k-1,1}(F) = \overline{F}.
\end{equation}

The proof of the following proposition is analogous to the proof of Proposition \ref{comparisonofnorms}
\begin{prop}
There exists a positive constant $c$ such that
$$||F||_{L^2(H \backslash G)}^2 = c||T_{k-1,1}(F)||^2_{L^2(D)} = c||\overline{F}||_{L^2(D)}^2$$
\end{prop}

\section{Spectral gaps and the vanishing of $L^2$-cohomology} \label{spectralgapsection}
\hfill

In this section we relate how uniform (over points in the cone) lower bounds  on the  spectrum of the Casimir operator $C$ on the square integrable functions on almost every orbit $G \mathbf{x}$ implies vanishing results for the relative Lie algebra cohomology with values in $L^2(V^k)$.

Our goal in this section is to prove vanishing theorems for relative Lie algebra cohomology with values in $L^2(\Omega _+ \cup \Omega_-)$ and hence with values in $L^2(V^k)$ since $V^k - (\Omega _+ \cup \Omega_-)$ has measure zero.   We will use the transfer map to reduce problems on cochains with values in $L^2$ of an orbit to problems on square-integrable forms on hyperbolic space.  To illustrate the use of the transfer map  we will  begin with a simple lemma. 
\begin{lem}
The spectrum of the Casimir on the relative  $\ell$-cochains \\
$\Hom_K(\wwedge{\ell}\mathfrak{p}, L^2(\Omega _+ \cup \Omega_-))$ is contained in $[0,\infty), 0 \leq \ell \leq n$.
\end{lem}
\begin{proof}
It suffices to prove that for all $\mathbf{x} \in \Omega$ the spectrum of the Casimir on $\Hom_K(\wwedge{\ell}\mathfrak{p}, L^2(G \mathbf{x}))$ is contained in $[0,\infty), 0 \leq \ell \leq n$.  Under the maps $T_k$ (resp. $T_{k-1,1}$), $\Hom_K (\wwedge{\ell}\mathfrak{p}, L^2(G \mathbf{x}))$ maps isometrically into $\ell$-forms on hyperbolic space $D$, resp. the tube $E = \Gamma_{\mathbf{x}} \backslash D$, and the Casimir maps to the Hodge Laplacian by Kuga's Lemma.  But the spectrum of the Laplacian on forms on $D$,  respectively $E$, is non-negative.

\end{proof}

 Let $\Omega$ represent either or the two cones $\Omega_+$ or $\Omega_-$.  We then have 

\begin{prop} \label{Mateisprop1}
Suppose the spectrum of the Casimir on $C^{\ell}(L^2(G \mathbf{x}))$ is contained in $[\epsilon, \infty)$ with $\epsilon > 0$ independent of $\mathbf{x}$ for one and hence every $\mathbf{x}$.  Then there is a bounded  operator $\delta \mathcal{G}$ from $C^{\ell}( L^2(\Omega))$ to $C^{\ell-1}(L^2(\Omega))$ such that  for any cocyle  $\alpha \in C^{\ell}(L^2(\Omega))$ we have
$$ d (\delta \mathcal{G} \alpha) = \alpha.$$

Hence
$$H^{\ell}(\mathfrak{so}(n,1), \SO (n) ; L^2(\Omega)) =0.$$
\end{prop}
\begin{proof} We will treat the case of $\Omega = \Omega_-$ and leave the case $\Omega= \Omega_+$ to the reader. 

Choose $\mathbf{x} \in \Omega_-$.  
By assumption, the spectrum of the Casimir operator $\mathrm{C}$ on  $C^{\ell}(L^2(G \mathbf{x}))$ is contained in $[\epsilon, \infty)$ with $\epsilon > 0$ independent of $\mathbf{x}$. Let $\mathcal{G}|_{\mathbf{x}}$ be the inverse of the Casimir $\mathrm{C}$ on $C^{\ell}(L^2(G \mathbf{x}))$. Then $\mathcal{G}|_{\mathbf{x}}$  commutes with the  differential $d$ and satisfies
\begin{equation} \label{firstGreenestimate}
||\mathcal{G}|_{\mathbf{x}}(\alpha) ||_{L^2(G\mathbf{x})} \leq \frac{1}{\epsilon} ||\alpha||_{L^2(G \mathbf{x})} \text{ and } \mathrm{C} \mathcal{G}|_{\mathbf{x}} \alpha = \alpha  .
\end{equation}
The critical point is that the constant  $\frac{1}{\epsilon}$ in the  above estimate is  independent of $\mathbf{x}$.

Let  $\alpha   \in C^{\ell}(C^{\infty}_c(\Omega_-))$ be a cocycle. By Lemma \ref{Fubini}, the cochain $\alpha$ is defined, measurable and square-integrable on almost every orbit $G \mathbf{x}$.  We will refer to the set of $\mathbf{x}$ such that $\alpha$ has the above three properties as the domain of $\alpha$.  Choose $\mathbf{x} \in \Omega_-$ in the domain of $\alpha$  and consider $\alpha(\mathbf{x}) \in C^{\ell}(C^{\infty}_c(\Omega_-))$.  Put $\beta = \mathcal{G}|_{\mathbf{x}}(\alpha) \in C^{\ell}(C^{\infty}_c(\Omega_-))$.  Then letting $\mathbf{x}$ vary we may extend $\beta$ to  every orbit  of $\Omega_{k-1,1}$ by the formula
$$\beta(\mathbf{x}) =  \mathcal{G}|_{\mathbf{x}}(\alpha)(\mathbf{x}).$$
From the inequality \eqref{firstGreenestimate} for every $\mathbf{x}$ and every $\alpha \in C^{\ell}(C^{\infty}_c(\Omega_-))$ we have 
\begin{equation} \label{secondGreenestimate}  
||\mathcal{G} \alpha ||_{L^2(G\mathbf{x})} \leq \frac{1}{\epsilon} ||\alpha||_{L^2(G \mathbf{x})} \text{ and } \mathrm{C} \mathcal{G}|_{\mathbf{x}} \alpha = \alpha  .
\end{equation}
Each side of the inequality \eqref{secondGreenestimate} may be considered as a function on the space of $G$ orbits $G \backslash \Omega_- \cong \mathbb{P}_{k-1} \times D_k$.  The $L^2(\Omega_-)$-norm of each side is obtained by integrating each side with respect to the push-forward measure on this quotient space.  Since $\frac{1}{\epsilon}$ is a constant function on this quotient space, after performing these integrations we obtain  
 \begin{equation}\label{firstextension}
||\mathcal{G}(\alpha)||_{L^2(\Omega_-)} \leq \frac{1}{\epsilon} ||\alpha||_{L^2(\Omega_-)} .
\end{equation}
Hence the extension $\beta$ satisfies $\beta \in C^{\ell}(L^2 (\Omega_{k-1,1})).$

Now for each $\mathbf{x}$ use the transfer $T_{k-1,1}$ to map $\alpha$ and $\beta$ to square integrable forms on hyperbolic space depending on the parameter $\mathbf{x}$.  We again use     $\mathcal{G}$, the Green operator, to denote the right-inverse of the Hodge Laplacian $\Delta$. We note
$$ T_{k-1,1} \circ \mathcal{G}|_{\mathbf{x}} \circ T^{-1}_{k-1,1} =\mathcal{G}$$
and we have the estimate inherited from Equation \eqref{firstGreenestimate}
\begin{equation}\label{ellipticestimate} 
||\mathcal{G}\eta||_{L^2(D)} \leq \frac{1}{ \epsilon}  || \eta ||_{L^2(D)}.
\end{equation}

For $\mathbf{x}$ and $\alpha$  as above we let
$$\omega  = T_{k-1,1}( \alpha ).$$
Then $\omega $ is a smooth compactly supported  closed $\ell$-form depending on $\mathbf{x} \in \Omega_-$.  Define a smooth $\ell$-form $\eta $ depending on $\mathbf{x} \in \Omega_-$ by 
$$\eta = T_{k-1,1}( \beta) = T_{k-1,1}(\mathcal{G}|_{\mathbf{x}}(\omega)) =
\mathcal{G}(\omega).$$
Hence  $\eta$ satisfies the equation
$$ \Delta (\eta ) = \omega.$$
Define an $(\ell-1)$-form $\nu$ by 
$$\nu = d^* \eta.$$
We claim that 
\begin{enumerate}
\item $d \nu = \omega$
\item $||\nu||_{L^2(D)} = ||d^* \mathcal{G}(\omega) ||_{L^2(D)}\leq \frac{1}{\epsilon} ||\omega||_{L^2(D)}.$
\end{enumerate}
We first prove (1).  Note first that since $d$ commutes with $\mathcal{G}$ we have $ d \mathcal{G} \omega = 0$.  We then have 
\begin{equation*}
\omega = \Delta(\eta) = (d d^* + d^* d)( \mathcal{G}(\omega)) = d d^*( \mathcal{G}(\omega)) = d( d^*( \mathcal{G}(\omega))) = d \nu
\end{equation*}
We now prove (2). We use $ ((\ , \ ))$ to denote the inner product on $A^{\ell-1}_{(2)}(D)$ and $ || \ ||$ to denote the associated norm . Then we have 
\begin{align*}
|| \nu||^2 = &||d^*( \mathcal{G}(\omega)) ||^2 
= (( d^*( \mathcal{G}(\omega)),d^*( \mathcal{G}(\omega)))) \\
= & ((d d^* ( \mathcal{G}(\omega)),  \mathcal{G}(\omega))
= (( \omega,  \mathcal{G}(\omega))) \\
\leq &|| \omega||^2||\mathcal{G}(\omega) ||^2
\leq \frac{1}{\epsilon} || \omega||^2.
\end{align*}
Finally we will need to know that the image of closed $L^2$ $\ell$-forms under $d^* \mathcal{G}$ is contained in the domain of $d$.  But we have for $\omega$ closed
\begin{equation}\label{domainestimate}
|| d d^*  \mathcal{G} \omega||^2  = || \Delta   \mathcal{G} \omega||^2 
=  || \omega ||^2 . 
\end{equation}
In case $\omega$ depends on $\mathbf{x}$ the above equation holds for all $\mathbf{x}$. 

Now define 
$$\delta: C^{\ell}(C^{\infty}_c(G \mathbf{x}) ) \to C^{\ell}(C^{\infty}_c(G \mathbf{x}))$$
by
$$\delta = T^{-1} _{k-1,1} \circ d^* \circ T_{k-1,1}.$$
Also define $\gamma(\mathbf{x})$ by
$$ \gamma = T^{-1} _{k-1,1}(\nu) = \delta \mathcal{G}(\alpha).$$
We have 
\begin{enumerate}
\item $d \gamma = \alpha $
\item $\delta \mathcal{G} \alpha||_{L^2(G \mathbf{x})} \leq \frac{1}{\epsilon}||\alpha||_{L^2(G \mathbf{x})} .$
\end{enumerate}
The previous estimate then becomes
\begin{equation} \label{nexttolastestimate}
|| \gamma||^2 _{L^2(G\mathbf{x})} = || \delta \mathcal{G} \alpha||^2 _{L^2(G\mathbf{x})}  \leq  \frac{1}{\epsilon} || \alpha||^2 _{L^2(G\mathbf{x})}. 
\end{equation}

We integrate both sides of Equation \eqref{lastestimate} over the space of orbits $G \backslash \Omega_-$ using the argument proving the inequality \eqref{firstextension}  to obtain
\begin{equation}\label{lastestimate}
||  \delta \mathcal{G} \alpha||^2 _{L^2(\Omega_{k-1,1})} \leq \frac{1}{\epsilon} || \alpha||^2 _{L^2(\Omega_{k-1,1})}. 
\end{equation}
Hence $\delta \mathcal{G}$ is a bounded operator in the $L^2$-norm from cocycles in $C^{\ell}(C^{\infty}_c(\Omega_-))$ to $C^{\ell-1}( C^{\infty}_c(\Omega_-))$. Hence it extends to a bounded operator from $L^2(\Omega_-)$-valued $(\ell-1)$-cocycles to $L^2(\Omega_-)$-valued-cochains. 

It remains to prove that $\delta \mathcal{G}$ maps cocycles to the domain of $d$.  But  the equation $d (\delta \mathcal{G} \alpha)=\alpha$ for smooth forms extends by continuity to square integrable forms. Hence $\delta \mathcal{G} \alpha$ is in the domain of $d$. 

\end{proof}

We now have the  following extension of  Proposition \ref{Mateisprop1}. Again let $\Omega$ represent either or the two cones $\Omega_+$ or $\Omega_-$. Then we have the following

\begin{prop} \label{Mateisprop2}
Suppose for almost all  $\mathbf{x} \in \Omega$ there exists $\epsilon > 0$ independent of $\mathbf{x}$ such that the spectrum of the Casimir on  $C^{\ell}(L^2(G \mathbf{x})))$, is contained in  $\{0\} \cup [\epsilon, \infty)$.  Let $\mathcal{H}$ be the projection on the zero eigenspace of $C$ in $L^2(\Omega)$.  Then there is a bounded operator $\delta \mathcal{G}$ from $C^{\ell}( L^2(\Omega))$ to $C^{\ell-1}( L^2(\Omega))$ such that for any cocycle $\alpha$ we have
$$d \delta \mathcal{G} \alpha = \alpha - \mathcal{H}(\alpha)$$
for any cocyle  $\alpha \in C^{\ell}(L^2(\Omega))$.  Hence any class $[\alpha]$ in $H^{\ell}(L^2(\Omega))$ has a harmonic representative. Such a representative is smooth along the orbits where it is defined.  Hence
$$H^{\ell}(L^2(\Omega)) = 0 \iff H^{\ell}(L^2(G \mathbf{x}))=0$$
for some and hence every $\mathbf{x} \in \Omega$. 
\end{prop}

\begin{proof}
Let $\mathcal{G}$ be the right inverse to the restriction of the Casimir $C$  to the orthogonal complement of the zero eigenspace of $C$ in $\Hom_K(\wedge^{\ell}(\mathfrak{p}),  L^2(G\mathbf{x})$.   Extend $\mathcal{G}$ to all of $C^{\ell}(L^2(\Omega))$ by defining it to be zero on zero eigenspace of $C$. Let $\mathcal{H}_{\ell}$ be the projection on the zero eigenspace of  $C$.  Then we have
$$C \circ \mathcal{G} = I - \mathcal{H}_{\ell}.$$ 
Then for every $\alpha \in C^{\ell}(C^{\infty}_c(\Omega))$ we have
\begin{equation} \label{secondpropGreenestimate}
||\mathcal{G}|_{\mathbf{x}}(\alpha) ||_{L^2(G\mathbf{x})} \leq \frac{1}{\epsilon} ||\alpha||_{L^2(G \mathbf{x})} \text{ and } \mathrm{C} \mathcal{G}|_{\mathbf{x}} \alpha = \alpha  .
\end{equation}

Now use the transfer $T_{k-1,1}$ resp. $T_{k,0}$ to transfer $\mathcal{G}$ and $\mathcal{H}_{\ell}$ to $D$ resp. $E$ where they become the familiar objects from Hodge theory.  We then have the operator $d^* \mathcal{G}$ such that on the orthogonal complement of the harmonic forms we have the estimate (independent of $\mathbf{x}$)
\begin{equation}
||d^*( \mathcal{G}(\omega)) ||^2  \leq \frac{1}{\epsilon} || \omega||^2.
\end{equation}

We transfer back to the cone with $\delta$ the transfer of $d^*$ and obtain the operator 
$\delta \mathcal{G}$ satisfying for all cocycles $\alpha$ and all $\mathbf{x}$
\begin{equation} \label{newnexttolastestimate}
|| \delta \mathcal{G} \alpha||^2 _{L^2(G\mathbf{x})}  \leq  \frac{1}{\epsilon} || \alpha||^2 _{L^2(G\mathbf{x})}. 
\end{equation}

As before we integrate both sides of Equation \eqref{newnexttolastestimate} over the space of orbits $G \backslash \Omega_-$ using the argument proving the inequality \eqref{firstextension}  to obtain
\begin{equation}\label{finalestimate}
||  \delta \mathcal{G} \alpha||^2 _{L^2(\Omega_{k-1,1})} 
\leq \frac{1}{\epsilon} || \alpha||^2 _{L^2(\Omega_{k-1,1})}. 
\end{equation}

\end{proof}

Finally we have one more application  of the technique which we will apply in the cases $\ell = \frac{n-1}{2}$ and $\ell = \frac{n+1}{2}$. 

\begin{prop} \label{Mateisprop3}
Suppose the  discrete spectrum of $\Delta$ on ${A}_{(2)}(D)$ is empty.  Then the reduced cohomology $\overline{H}^\ell(L^2(\Omega_-)) = 0$.
\end{prop}

\begin{proof}
We may assume the spectrum is $[0, \infty)$ otherwise the nonreduced cohomology is zero by Propostion \ref{Mateisprop1}.  Let $\epsilon > 0$ and let $P_{[\epsilon, \infty)}$ be the projection coming from spectral theory.  Let $\mathcal{G}_\epsilon$ be the right inverse of the restriction of $\Delta$ to the image of $P_{[\epsilon, \infty)}$.  Now let $\omega$ be closed and
$$\omega_\epsilon = P_{[\epsilon, \infty)} \omega.$$
Since the discrete spectrum is the empty set, the family  $\omega_{\epsilon}$ converges to $\omega$ as $\epsilon$ goes to zero.
Put $\eta_\epsilon = d^* \mathcal{G}\omega_\epsilon$.  Then $d \eta_\epsilon = \omega_\epsilon$ and hence $\omega_\epsilon$ is a family of exact forms converging to $\omega$.  

\end{proof}

\begin{prop} \label{Mateisprop4}
Suppose the spectrum of $\Delta$ on $C^\ell(L^2(G\mathbf{x}))$ is contained in $\{0\} \cup [\epsilon, \infty)$ for some $\epsilon > 0$.  Then $d(C^\ell(L^2(G\mathbf{x})))$ is closed in $C^{\ell+1}(L^2(G\mathbf{x}))$ and hence
\begin{equation*}
H^{\ell+1}(L^2(G\mathbf{x})) = \overline{H}^{\ell+1}(L^2(G\mathbf{x})).
\end{equation*}

\end{prop}

\begin{proof}
We note that since $d$ commutes with the spectral projection $P_{[\epsilon, \infty)}$, we have
\begin{equation*}
d(C^\ell(L^2(G\mathbf{x}))) = d(P_{[\epsilon, \infty)}C^\ell(L^2(G\mathbf{x}))) = P_{[\epsilon, \infty)} d(C^\ell(L^2(G\mathbf{x}))).
\end{equation*}
But $P_{[\epsilon, \infty)} d(C^\ell(L^2(G\mathbf{x}))) \subset P_{[\epsilon, \infty)} C^{\ell+1}(L^2(G\mathbf{x}))$ which is a closed subspace.  Now suppose $\beta$ is closed and in the closure of the image of $d$.  Then $\beta$ is in \\ 
$P_{[\epsilon, \infty)} C^{\ell+1}(L^2(G\mathbf{x}))$ and hence $\mathcal{G}\beta$ is defined.  Since $\beta$ is closed we have
%Now suppose $\beta \in C^{\ell+1}(L^2(G\mathbf{x}))$ is closed and in the closure of the image of $d$.  Then we have $\beta \in P_{[\epsilon, \infty)} C^{\ell+1}(L^2(G\mathbf{x}))$.  Thus $\mathcal{G} \beta$ is defined and we have
%$$\Delta \mathcal{G} \beta = \beta.$$
%Since $\beta$ is closed we have
$$d \delta \mathcal{G} \beta = \beta.$$

\end{proof}

\section{The geometry of the tubes $E_\mathbf{x}$} \label{tubessection}

We will review certain facts about the geometry of these tubes. Let $\mathbf{x} \in \Omega_{k,0}$ and $U = \mathrm{span} (\mathbf{x})$.  First by \cite{KM1}, pg. 216-217,  the intersection of the orthogonal complement of the span $U$ of $\mathbf{x}$ with $D$ is a totally-geodesic hyperbolic $(n-k)$-space $D_U$ and $C_U = \Gamma_{\mathbf{x}} \backslash D_U$ is a compact totally-geodesic hyperbolic $(n-k)$-manifold in $E_\mathbf{x}$. Moreover there is a topologically-trivial fiber bundle map $\pi:E_\mathbf{x} \to C_U$ with fibers the exponential of the fibers of the normal bundle to $C_U$ in $E_\mathbf{x}$.  The fibers of $\pi$ are totally-geodesic embeddings of hyperbolic space $D^k$.  Thus we may think of $E_\mathbf{x}$ as a trivial $k$-dimensional vector bundle over $C_U$. 

In order to apply \cite{MP} we need to compactify the tube.  Let $S^{n-1}_{\infty}$ be the sphere at infinity in the standard compactification $\overline{D}^n$ of hyperbolic space as a closed $n$-ball and let $S^{n-k-1}_{\infty}$ be the sphere at infinity in the standard compactification as an $n-k$-ball of the totally-geodesically embedded hyperbolic $n-k$-space $\overline{D}_{\mathbf{x}}$.  Then $\Gamma_{\mathbf{x}}$ acts on $\overline{D}^n$ with limit set $S^{n-k-1}_{\infty}$.  Then $\Gamma_{\mathbf{x}}$ acts properly discontinuously on $\overline{D}^n - S^{n-k-1}_{\infty}$ and the quotient is a compact manifold $\overline{E}_\mathbf{x}$ with boundary the product $ (\Gamma_{\mathbf{x}} \backslash D_{\mathbf{x}}) \times S^{n-k-1}_{\infty}$. 
In  this compactification we have compactified each fiber of $\pi$ by adding a $k-1$-sphere at infinity.

Since $\Gamma_\mathbf{x}$ is a cocompact discrete subgroup of $G_\mathbf{x}$, it has no parabolic elements.  Hence, in the notation of \cite{MP}, we have $M = E_\mathbf{x}$ and $\partial_c M = \emptyset$.  Since by \cite{MP}, page 519, $\partial M = \partial_r M \cup \partial_c M$, we have
\begin{equation*}
\partial M = \partial_r M = \partial_r^k M, \text{ all } k.
\end{equation*}
Also,
\begin{equation*}
\partial_r M = \overline{\partial_r M}.
\end{equation*}
For each $1 \leq k \leq n$ we choose a cocompact discrete torsion-free subgroup $\Gamma_k \subset \SO(n-k,1)$ and put $E_k = \Gamma_k \backslash D$.

\section{A spectral gap for square integrable forms on $E_\mathbf{x}$.}

For $\mathbf{x} \in V^k$ we define the space of harmonic cochains $\mathcal{H}^\ell(L^2(G \mathbf{x}))$ to be the subspace of cochains in $C^\ell(L^2(G \mathbf{x}))$ which are annihilated by the Casimir.  We make an analogous definition for cochains with values in $L^2(\Omega_{k-1,1})$ and $L^2(\Omega_{k,0})$.

 The following remark will be important in what follows. 
 \begin{rmk}\label{independenceofC}
 Recall, see Equation \eqref{postriv}, there is  an explicit trivialization by
\begin{align*}
F_+: \mathbb{S}_k \times  \mathbb{P}_k &\to \Omega_+ \\
( \mathbf{x}, \beta) &\mapsto \mathbf{x} \sqrt{\beta}.
\end{align*}
It is important to observe that under this trivialization the Casimir operator $C$ does not depend on the coordinate $\beta$ in the second factor. This is because $C$ commutes with the action of $\GL(k,\R)$. In particular the spectrum of the Casimir in $C^\ell(L^2(G \mathbf{x}))$ coincides with the spectrum of the Casimir in \\
$C^\ell(L^2(\SO (n-k, 1) \backslash \SO (n,1)))$. Note that this implies in particular that it is independent of $\mathbf{x}$.
\end{rmk}

We now have 
\begin{prop} \label{Nicolaslowerbound1}
Let $\ell$ be an integer such that $\ell \neq \frac{n-1}{2}, \frac{n+1}{2}$ and let $\mathbf{x} \in \Omega_{k,0}$. Then the spectrum of the Casimir in $C^\ell(L^2(G \mathbf{x}))$ is contained in $\{0 \} \cup [\epsilon, \infty)$ with $\epsilon >0$ and independent of $\mathbf{x}$.
\end{prop}
\begin{proof} By the remark immediately above, it remains to find a spectral gap for the  spectrum of $C$ on $C^\ell(L^2(G \mathbf{e}))$ where $\mathbf{e} = (e_1,\cdots, e_k)$.  The latter is the union of the essential spectrum and the discrete spectrum.  Being discrete, the discrete spectrum is certainly contained in $\{0 \} \cup [\epsilon, \infty)$ for some positive $\epsilon$. Hence it remains to prove the existence of a nonzero lower bound for the essential spectrum. 

We will use again the transfer map $T_{k,0}$. To conclude the proof we will in fact prove that the essential spectrum of the Laplacian on $A_{(2)}^\ell (E)$ is bounded away from $0$. Here $E$ is the  tube $E_{\mathbf{e}} = \Gamma_{\mathbf{e}} \backslash D$; the spectral gap will be independent of the choice of $\Gamma_{\mathbf{e}}$.  Now the proposition follows by  \cite{MP}, Lemma 5.6 and Lemma 5.19. 

\end{proof}

\begin{rmk}
\hfill

\begin{enumerate}
\item The Proposition \ref{Nicolaslowerbound1} is false when $\ell = \frac{n -1}{2}, \frac{n +1}{2}$. In that case the spectrum of the Laplacian on $E$ is $[0,\infty)$ by Lemma 5.23 of \cite{MP}. 

\item In Proposition \ref{Nicolaslowerbound1} one can in fact take $\varepsilon = n-2\ell-2$ as long as $\ell \leq (n-4)/2$. This can be deduced from \cite{inventiones} using \cite{BLS}. 
\end{enumerate}
%This should also follows from the Plancherel formula for %$L^2(\SO (n-k, 1) \backslash \SO (n,1))$ but we have not %checked it. This $\varepsilon$ is optimal (see the proof %of \cite[Proposition 6.8]{inventiones}). It could %probably be deduced from the Plancherel formula that one %can still take $\varepsilon = n-2\ell-2$ as long as $\ell %< (n-2)/2$ and $\varepsilon = 1/4$ when $\ell = (n-2)/2$ %or $n/2$. 
\end{rmk}

Now by applying Proposition \ref{Mateisprop2}, we obtain

\begin{prop} \label{tubevanishing}
Let $\ell \neq \frac{n-1}{2}, \frac{n +1}{2}$. 
\begin{enumerate}
\item Suppose for some $\mathbf{x} \in \Omega_{k,0}$ the group $H^{\ell}(\mathfrak{so}(n,1), \SO (n);L^2(G \mathbf{x})) = 0$.  Then $H^{\ell}(\mathfrak{so}(n,1), \SO (n);L^2(\Omega_{k,0})) = 0$.
\item Suppose for some $\mathbf{x} \in \Omega_{k,0}$ the group $H^{\ell}(\mathfrak{so}(n,1), \SO (n);L^2(G \mathbf{x})) \neq 0$.  Then $H^{\ell}(\mathfrak{so}(n,1), \SO (n);L^2(\Omega_{k,0})) \neq 0$. 
\end{enumerate}
\end{prop}
\begin{proof}
By Proposition \ref{Mateisprop2} and Proposition \ref{Nicolaslowerbound1}, $H^\ell(L^2(G \mathbf{x})) \cong \mathcal{H}^\ell(L^2(G \mathbf{x}))$ and $H^\ell(L^2(\Omega_{k,0})) = \mathcal{H}^\ell(L^2(\Omega_{k,0}))$. To prove (1), note that since $\GL(k,\R)$ acts transitively on the orbits, the group $H^{\ell}(L^2(G \mathbf{x}))$ is independent of $\mathbf{x}$.  To prove (2), note that any square integrable harmonic form on an orbit can be extended to a square integrable harmonic form on $\Omega_{k,0}$ supported in a neighborhood of that orbit.  The proposition follows.

\end{proof}

\section{The vanishing of the $H^{\ell}( L^2(V^k))$ away from the middle dimensions for $k \geq \frac{n+1}{2}$}

\begin{thm} \label{T:klarge}
Suppose $k \geq \frac{n+1}{2}$, then
$$H^\ell(\mathfrak{so}(n,1), \SO (n); L^2(V^k)) = 0 \text{ for } \ell \notin [\frac{n-1}{2}, \frac{n+1}{2}].$$
\end{thm}
\begin{proof}
We first show
$$H^{\ell}( \mathfrak{so}(n,1), \SO (n); L^2(\Omega_{k-1,1})) = 0 \text{ for } \ell \notin [\frac{n-1}{2}, \frac{n+1}{2}].$$

We use the transfer map.  By \cite{Pedon}, page 68, the spectrum of the Laplacian on square integrable $\ell$-forms on $D$ is contained in $[(\frac{n-1}{2} - \ell)^2, \infty)$.  Hence the spectrum of $C$ on cochains $C^\ell(L^2(G \mathbf{x}))$ for any $\mathbf{x} \in \Omega_{k-1,1}$ is contained in $[(\frac{n-1}{2} - \ell)^2, \infty)$.  The proposition follows by Proposition \ref{Mateisprop1}.

Now we show 
$$H^{\ell}(\mathfrak{so}(n,1), \SO (n); L^2(\Omega_{k,0})) = 0 \text{ for } \ell \notin [\frac{n-1}{2}, \frac{n+1}{2}].$$
We may assume $\ell < \frac{n}{2}$.  Choose $\mathbf{x} \in \Omega_{k,0}$.  By Proposition \ref{tubevanishing} it suffices to prove $H^\ell(L^2(G \mathbf{x})) = 0$.  By using the transfer map it suffices to show $H_{(2)}^\ell(E_\mathbf{x})^{G_\mathbf{x}} = 0$.  By \cite{MP}, Theorem 3.13,
$$H_{(2)}^\ell(E_\mathbf{x})^{G_\mathbf{x}} = H_c^\ell(E_\mathbf{x})^{G_\mathbf{x}}.$$
By the Thom Isomorphism Theorem,
$$H_c^\ell(E_\mathbf{x})^{G_\mathbf{x}} = H^{\ell-k}(C_U)^{G_\mathbf{x}}.$$
But, $\ell-k \leq \frac{n}{2} - \frac{n+1}{2} < 0$.  Hence $H_c^\ell(E_\mathbf{x}) = 0$ and so $H_c^\ell(E_\mathbf{x})^{G_\mathbf{x}} = 0$.

\end{proof}

\section{The computation of $H^{\ell}( L^2(V^k ))$ away from the middle dimensions for $k \leq \frac{n}{2}$  }

\begin{thm} \label{T:ksmall}
Suppose $k \leq \frac{n}{2}$. Then for $\ell \notin [\frac{n-1}{2}, \frac{n+1}{2}]$ we have:
$$H^{\ell} (\mathfrak{so}(n,1) , \mathrm{SO} (n) ; L^2 (V^k)) = \overline{H}^{\ell} (\mathfrak{so}(n,1) , \mathrm{SO} (n) ; L^2 (V^k)).$$
Furthermore, $H^{\ell} (\mathfrak{so}(n,1) , \mathrm{SO} (n) ; L^2 (V^k))$ is non-zero if and only if $\ell = k$ or $\ell = n-k$. Finally, $H^k(\mathfrak{so}(n,1), \SO(n); L^2(V^k))$ is the discrete series representation of $\mathrm{Mp}(2k, \R)$ with parameter $(\frac{n+1}{2}, \cdots, \frac{n+1}{2})$.
\end{thm}
\begin{proof}
We first show 
$$H^{\ell}( \mathfrak{so}(n,1), \SO (n); L^2(\Omega_{k-1,1})) = 0 \text{ for } \ell \notin [\frac{n-1}{2}, \frac{n+1}{2}].$$

The proof is similar to that of Theorem \ref{T:klarge}:  here again the spectrum of the Casimir operator on $C^{\ell}(L^2(G \mathbf{x}))$ for $\ell \notin [\frac{n-1}{2}, \frac{n+1}{2}]$ is contained in $[\epsilon, \infty)$ with $\epsilon> 0$ and independent of $\mathbf{x}$. To see that we can again use the transfer map $T_{k-1,1}$ and prove the corresponding statement for the square integrable forms on hyperbolic space.  By \cite{Pedon}, Corollary 4.4, the spectrum of $\Delta$ on square integrable $\ell$-forms is bounded away from zero unless $n$ is even and $\ell = \frac{n}{2}$ or $n$ is odd and $\ell = \frac{n-1}{2}$ or $\ell= \frac{n+1}{2}$.  We conclude using Proposition \ref{Mateisprop1}.

%\begin{rem}
%For a fixed $\mathbf{x}$ we could apply Theorems A and B in \cite{Borel1} which give the above results for the $L^2$-cohomology of $D$.  However we need the uniform spectral gaps to deal with the parameter $\mathbf{x} \in \Omega_{k-1,1}$. 
%\end{rem}

Now we consider the cone $\Omega_{k,0}$.  In this case we use the transfer map $T_{k,0}$ and Proposition \ref{tubevanishing} to reduce the problem to the study of the $L^2$-cohomology of the tube $E = \Gamma_{\mathbf{x}} \backslash D^n$.  Suppose $\ell < \frac{n}{2}$. Then we have
$$H^{\ell}( \mathfrak{so}(n,1), \SO (n); L^2(\Omega_{k,0})) \neq 0 \iff \ell =k .$$

To see this we first apply Proposition \ref{Nicolaslowerbound1} and Proposition \ref{Mateisprop2} to obtain that $H^{\ell}(L^2(\Omega_{k,0}))$ is nonzero if and only if  $H^{\ell}(L^2(G \mathbf{x}))$ is nonzero. The transfer map $T_{k,0}$ gives an isomorphism between $H^{\ell}(L^2(G \mathbf{x}))$ and the $G_\mathbf{x}$-invariant elements of the $L^2$-cohomology of the tube $E$.

By \cite{MP}, Theorem 3.13, the reduced $L^2$-cohomology of the tube $E$ in degrees less than $\frac{n}{2}$ is isomorphic to the compactly-supported cohomology $H_c^\bullet(E)$.  We may apply the Thom Isomorphism Theorem to obtain
$$H^\ell_c(E) \cong H^{\ell-k}(C_U).$$ 
Since the Thom isomorphism is $G_\mathbf{x}$-equivariant we have an isomorphism
$$H^\ell_c(E)^{G_\mathbf{x}} \cong H^{\ell-k}(C_U)^{G_\mathbf{x}}.$$
But the only invariant forms on hyperbolic $n-k$ space are in degrees $0$ and $n-k$.  Hence either $\ell - k = 0$ or $\ell - k = n-k$.  Hence either $\ell = k$ or $\ell = n$.  But by hypothesis  $\ell < \frac{n}{2}$ and hence $\ell =k$.

Finally using Poincar\'e duality we conclude that $H^{\ell} (L^2 (V^k))$ is non-zero if and only if $\ell = k$ or $\ell = n-k$. The fact that 
$$H^{\ell} (\mathfrak{so}(n,1) , \mathrm{SO} (n) ; L^2 (V^k)) = \overline{H}^{\ell} (\mathfrak{so}(n,1) , \mathrm{SO} (n) ; L^2 (V^k))$$
follows from the fact that the $0$ eigenvalue of the Casimir is isolated. 

It remains to identify $H^k(L^2(V^k)) $ as a $\frak{sp}(2k,\C)$-module. First recall that, by Borel Wallach \cite{BW} Prop 3.1, if $\ell  < \frac{n-1}{2}$, then there exists a unique irreducible unitary representation $J_\ell$ of $\SO(n,1)$ such that $H^\ell (J_\ell ) \neq 0$. And we have:
$$H^i(\mathfrak{so}(n,1) , \mathrm{SO} (n); J_\ell) = C^i(\mathfrak{so}(n,1) , \mathrm{SO} (n); J_\ell) = \begin{cases} 0 \text{ if } i \neq \ell, n-\ell \\ \C \text{ if } i = \ell, n-\ell. \end{cases}$$
It follows from (the proof of) Proposition \ref{Nicolaslowerbound1} that $J_\ell$ occurs (with finite multiplicity) in $L^2 (\Omega_{k,0})$ --- and hence in $L^2 (V^k)$ --- if and only if $k=\ell$.  In the latter case let $\mathcal{H}$ be the largest closed subspace of $L^2 (V^k)$ on which $\SO(n,1)$ acts by a multiple of $J_k$. Then $\mathcal{H}$ is stable under $\SO(n,1)$ and $\mathrm{Mp} (2k)$. By \cite[\S 7]{HoweTrans}, the joint action of $\SO(n,1) \times \mathrm{Mp} (2k)$ on $\mathcal{H}$ is a tensor product $J_k \otimes \sigma$, where $\sigma$ is an irreducible unitary representation of $\mathrm{Mp} (2k)$. Now since $k \leq \frac{n}{2}$ it follows from Li, \cite{Li}, Lemma 3.3 b, that the restriction of the Weil representation $\omega$ in $L^2 (V^k)$ to $\mathrm{Mp} (2k)$ is strongly $L^2$. It then follows from Li, \cite{Li}, Lemma 3.5, that the irreducible sub representation $\sigma \subset \omega |_{\mathrm{Mp} (2k)}$ is square integrable. By \cite{LiDuke} $\S 2$ we conclude that $J_k$ and $\sigma$ are the Howe duals of each other in the sense defined in \cite{HoweCorvalis}. The explicit determination of the Howe correspondence in that case is written down in \cite{LiDuke}. We conclude that $\sigma$ is the discrete series representation with parameter $(\frac{n+1}{2}, \cdots, \frac{n+1}{2})$. Since 
$$H^{k} (L^2 (V^k)) = H^{k} (\mathcal{H}) = \sigma \otimes \mathrm{Hom}_{\SO (n)} (\wwedge{k} \mathfrak{p} , J_k) \cong \sigma,$$
the theorem follows.

\end{proof}

\begin{rmk}
There is a very simple description of the invariant harmonic square-integrable form on $E$ of degree $k$.  By \cite{KM2} the Hodge star of the pullback $\pi^* \vol_{C_U}$ of the volume form $\vol_C$  of  $C$  is harmonic of degree $\ell=k$ and is square integrable if $k < \frac{n}{2}$. This form coincides with  the form $\phi_k$ of  Part 3.
\end{rmk}

In the next section we deal with the middle dimension cohomology, i.e. the case $\ell \in [\frac{n-1}{2}, \frac{n+1}{2}]$.

\section{The middle dimensional cohomology}

We distinguish two cases:
\begin{enumerate}
\item when $n=2m$ is even, then it only remains to deal with the case $\ell =m$;
\item when $n=2m+1$ is odd, then it remains to deal with the two dual cases $\ell = m$ and $m+1$.
\end{enumerate}

\subsection{The middle  dimensional cohomology  for general $k$}
Assume $n = 2m$ and let $\mathcal{H}_{(2)}^m(D)$ be the space of square integrable harmonic $m$-forms on $D$.  By \cite{MP}, Theorem 3.13, $\mathcal{H}_{(2)}^m(D)$ is infinite dimensional.  We will use the convention that if $k > n$ then $\SO(n-k,1)$ and $\SO(n-k+1)$ are the trivial groups.  By \cite{MP}, Theorem 3.13, $\mathcal{H}_{(2)}^m(E_k)$ is infinite dimensional.

By \cite{Pedon}, Proposition 3.1, we have an equality of unitary $\SO(n,1)$-representations
$$\mathcal{H}_{(2)}^m(D) = \pi_+ \oplus \pi_-$$
where $\pi_+$ and $\pi_-$ are the discrete series with trivial infinitesimal characters and lowest $K$-types $\wwedge{m}(V_+ \otimes \C)^+$ and $\wwedge{m}(V_+ \otimes \C)^-$ respectively.  Here the superscripts denote the $\pm i$ (respectively $\pm 1$) eigenspaces of the Hodge star operator for $m$ odd (respectively $m$ even).

We will describe two $\theta$-stable parabolics $\mathfrak{q}_+$ and $\mathfrak{q}_-$ such that in the notation of Vogan-Zuckerman, \cite{VZ},
\begin{equation*}
\pi_+ = A_{\mathfrak{q}_+} \text{ and } \pi_- = A_{\mathfrak{q}_-}.
\end{equation*}

We define a new ordered basis $\{u_1, \ldots, u_m, v_m, \ldots, v_1 \}$ consisting of isotropic vectors for $V_+ \otimes \C$ where $u_j = \frac{e_j + i e_{n+1-j}}{\sqrt{2}}, 1 \leq j \leq m$ and $v_j = \frac{e_j - i e_{n+1-j}}{\sqrt{2}}, 1 \leq j \leq m$.  Then note $(u_j, v_k) = \delta_{jk}$.  Let $E^\prime = \mathrm{span}(u_1, \ldots, u_m), F^\prime = \mathrm{span}(v_m, \ldots, v_1)$ and $E^{\prime \prime} = \mathrm{span}(u_1, \ldots, u_{m-1}, v_m), F^{\prime \prime} = \mathrm{span}(u_m, v_{m-1}, \ldots, v_1)$.  Then we define $\mathfrak{q}_+$ to be the stabilizer of $E^\prime$ and $\mathfrak{q}_-$ to be the stabilizer of $E^{\prime \prime}$.  We recall the identification of $\wwedge{2} V$ with $\mathfrak{so}(n,1)$ given by $(u \wedge v)(w) = (u,w)v - (v,w)u$.  Then $\mathfrak{q}_+ = (E^\prime \otimes F^\prime) \oplus (\wwedge{2} E^\prime) \oplus (E^\prime \otimes V_-)$ and $\mathfrak{q}_- = (E^{\prime \prime} \otimes F^{\prime \prime}) \oplus (\wwedge{2} E^{\prime \prime}) \oplus (E^{\prime \prime} \otimes V_-)$.  Then $\mathfrak{l}_+ = E^\prime \otimes F^\prime$ and $\mathfrak{u}_+ = \wwedge{2}E^\prime \oplus (E^\prime \otimes V_-)$ with $\mathfrak{k} \cap \frak{u}^\prime = \wwedge{2}E^\prime$ and $\mathfrak{p} \cap \frak{u}^\prime = E^\prime \otimes V_-$.  There are similar formulas for $\mathfrak{l}_-$ and $\mathfrak{u}_-$. Put $\mathbf{t} = (t_1, \ldots, t_m)$ with $t_j \in \R, 1 \leq j \leq m$.  We choose a Cartan subalgebra $\mathfrak{t} \subset \mathfrak{so}(n, \C)$ to be
$$\mathfrak{t} = \{ h(\mathbf{t}) = \sum_{j=1}^m t_j v_j \wedge u_j \}.$$
Then relative to the above ordered basis, the element $h(\mathbf{t})$ is the diagonal matrix with diagonal elements $(t_1, \ldots, t_m, -t_m, \ldots, -t_1)$.

We define an ordered basis $\{\epsilon_1, \ldots, \epsilon_m \}$ of $\mathfrak{t}^*$ where $\epsilon_j(\mathbf{t}) = t_j$.  Then in the notation of \cite{VZ}, we have $\Delta(\mathfrak{u}_+ \cap \mathfrak{p}) = \{ \epsilon_1, \ldots, \epsilon_{m-1}, \epsilon_m \}$ and $2 \rho( \mathfrak{u}_+ \cap \mathfrak{p}) = \epsilon_1 + \cdots + \epsilon_{m-1} + \epsilon_m$.  Also, $\Delta(\mathfrak{u}_- \cap \mathfrak{p}) = \{ \epsilon_1, \ldots, \epsilon_{m-1}, -\epsilon_m \}$ and $2 \rho( \mathfrak{u}_- \cap \mathfrak{p}) = \epsilon_1 + \cdots + \epsilon_{m-1} - \epsilon_m$.  We now recall Theorem 2.5 (c) of  \cite{VZ}, which states that the highest weights $\delta_+$ of the irreducible constituents of $A_{\mathfrak{q}_+}$ are of the form $\delta_+ = 2 \rho( \mathfrak{u}_+ \cap \mathfrak{p}) + \displaystyle \sum_{i=1}^m n_i \epsilon_i$.  Hence
\begin{equation} \label{deltaweight}
\delta_+ = \sum_{j=1}^m a_j \epsilon_j \text{ where } a_j > 0, 1 \leq j \leq m.
\end{equation}
The corresponding statement for $A_{\mathfrak{q}_-}$ is
$$\delta_- = \sum_{j =1}^m a_j \epsilon_j \text{ where } a_j > 0, \ 1 \leq j \leq m-1 \text{ and } a_m < 0.$$

Note that $\wwedge{m}\C^{2m} = \delta_+ + \delta_-$ is contained in the restriction of $\pi_+ \oplus \pi_-$ to $\SO(2m)$.

\begin{lem} \label{restrictioninvariance}
\hfill

\begin{enumerate}
\item $\mathcal{H}_{(2)}(D)^{\SO(p)} \neq 0$ if $p \leq m $.
\item $\mathcal{H}_{(2)}(D)^{\SO(p)} = 0$ if $p > m $.
\item $\pi_+^{\SO(p)} = \pi_-^{\SO(p)}0$ if $p > m $.
\end{enumerate}
\end{lem}

\begin{proof}
We first prove (1).  We have seen $\wwedge{m}\C^{2m} \subset \mathcal{H}_{(2)}(D)$.  But if we embed $\SO(m)$ into $\SO(2m)$ as the subgroup that fixes $e_1, \ldots, e_m$, then $\SO(m)$ fixes $e_1 \wedge \cdots \wedge e_m \in \wwedge{m}\C^{2m}$.  Hence $(\wwedge{m}\C^{2m})^{\SO(m)} \neq 0$ and the same holds for $\SO(p), p < m$.

We now prove (2).  It suffices to prove that the trivial representation of $\SO(m+1)$ does not occur in $\pi_+ \oplus \pi_-$.  We first show it does not occur in $\pi_+$.  Then it suffices to show $V_\delta^{\SO(m+1)} = 0$ for any irreducible constituent $V_\delta$ of $\pi_+|_K$.  Suppose $\delta = (a_1, \ldots, a_m)$.  Note that $a_m > 0$.

We apply the branching formulas from $\SO(2m)$ to $\SO(2m-1)$, see Boerner, \cite{Boerner}, page 269, Theorem 12.1a, to find that $V_\delta|_{\SO(2m-1)}$ contains $V_\tau$ if and only if the entries $b_1, \ldots, b_{m-1}$ of $\tau$ interlace the entries of $\delta$.  By this, we mean $a_1 \geq b_1 \geq a_2 \geq b_2 \geq \cdots \geq a_{m-1} \geq b_{m-1} \geq |a_m|$.  Note that $\tau$ has $m-1$ non-zero entries.  We apply branching again to compute the restriction of $V_\tau$ to $\SO(2m-2)$.  We find that the irreducible representation $V_\gamma$ with highest weight $\gamma = (c_1, \ldots, c_{m-1})$ of $\SO(2m-2)$ occurs in $V_\tau|_{\SO(2m-2)}$ if and only if the entries of $\gamma$ interlace the entries of $\tau$.  Thus we find the irreducible representation with highest weight $(a_1, a_2, \ldots, a_{m-2}, 0)$ occurs in $V_\delta |_{\SO(2m-2)}$.  Note that there are now only $m-2$ non-zero entries. Applying the branching rule 2 more times, we find the representation with highest weight $(a_1, \ldots, a_{m-4}, 0 ,0)$ occurs in $V_\delta |_{\SO(2m-4)}$.  Note that this weight has $m-4$ non-zero entries.  We successively apply the branching rule
\begin{equation*}
\SO(2m) \downarrow \SO(2m-1) \downarrow \SO(2m-2) \downarrow \cdots \downarrow \SO(m)
\end{equation*}
to the highest weight $\delta$ to find the highest weight $(\underbrace{0, 0, \ldots, 0}_{\lfloor \frac{m}{2} \rfloor})$ after branching $m$ times.  The reader will verify that this is the smallest number of branchings needed to obtain the zero weight.

A similar argument proves the statement for $\pi_-$.

\end{proof}

\begin{thm}
\hfill

\begin{enumerate}
\item $H^m(  \mathfrak{so}(n,1), \SO (n); L^2(\Omega_{k-1,1}))$ is non-zero if and only if $k > m$.
\item $H^m(  \mathfrak{so}(n,1), \SO (n); L^2(\Omega_{k,0}))$ is non-zero if and only if $m \leq k \leq n$.
\item $H^m(  \mathfrak{so}(n,1), \SO (n); L^2(V^k)) \neq 0$ if and only if $k \geq m$. In this case, as a $\mathrm{Mp}(2k)$-module, it decomposes as the sum of two irreducible unitary representations $M_{k, m}^+$ and $M_{k, m}^-$ of $\mathrm{Mp}(2k)$.  In case $k=m$, these representations are discrete series representations.
\end{enumerate}
\end{thm}
\begin{proof}

Let $\mathbf{x} \in \Omega_{k-1,1}$.  Then the transfer map gives an isomorphism
$$H^m( \mathfrak{so}(n,1), \SO (n); L^2(G \mathbf{x})) = H_{(2)}^m(D)^{G_\mathbf{x}}.$$

By \cite{Pedon}, Corollary 4.4, the spectrum of the Laplacian on $A^m_{(2)}(D)$ is $\{0\} \cup [\frac{1}{4}, \infty)$.  Then we apply Proposition \ref{Mateisprop2} and $H^m_{(2)}(D) \cong \mathcal{H}_{(2)}^m(D)$.  Hence $H^m_{(2)}(D)^{G_\mathbf{x}} \cong \mathcal{H}_{(2)}^m(D)^{G_\mathbf{x}}$, but $G_\mathbf{x} \cong \SO(n-k+1)$.  Hence (1) follows from Lemma \ref{restrictioninvariance}.

Now let $\mathbf{x} \in \Omega_{k,0}$ and let $E_\mathbf{x}$ be the associated tube.  Then the transfer map gives an isomorphism
$$H^m( \mathfrak{so}(n,1), \SO (n); L^2(G \mathbf{x})) \cong \mathcal{H}_{(2)}^m(E_\mathbf{x})^{G_\mathbf{x}}.$$
But by \cite{MP}, $\mathcal{H}_{(2)}^m(E_\mathbf{x})$ is infinite dimensional.  By Proposition \ref{Nicolaslowerbound1} there is a non-zero lower bound on the non-zero spectrum of the Laplacian on the tubes $E_\mathbf{x}$ which is uniform in $\mathbf{x}$.  We apply Proposition \ref{tubevanishing} and this proves that $H^m(L^2(\Omega_{k,0}))$ is non zero iff $\mathcal{H}_{(2)}^m(E)^{\SO(n-k,1)}$ is non zero.  If $m \leq k \leq n$ then $\mathcal{H}_{(2)}^m(E)^{\SO(n-k,1)}$ contains  $\mathcal{H}_{(2)}^m(E)^{\SO(m,1)}$.  But in \cite{KM1}, page 219, it is shown that the pullback of the volume form of the cycle $C_U$ is harmonic, square integrable and $\SO(m,1)$-invariant.  Hence
$$\mathcal{H}_{(2)}^m(E)^{\SO(m,1)} \neq 0.$$
In case $k>n$, $\Omega_{k,0} = \emptyset$, see below.

Suppose now $k < m$ and $\mathcal{H}_{(2)}^m(E)^{\SO(n-k,1)} \neq 0$. Let $\pi = \pi_+$ or $\pi_-$. Using the identification 
$$A^m_{(2)} (E)= \mathrm{Hom}_{\SO(n)} \left( \wwedge{m} \mathfrak{p} , L^2 (\Gamma_{\mathbf{x}} \backslash G) \right),$$ 
we define a linear map 
$$T_\pi : \mathrm{Hom}_K \left( \wwedge{*} \mathfrak{p} , \pi \right) \otimes \mathrm{Hom}_{\mathfrak{so}(n,1) , \SO (n)} \left( \pi , L^2 (\Gamma_{\mathbf{x}} \backslash G) \right) \to  A^m_{(2)} (E)$$
by $\psi \otimes \varphi \mapsto \varphi \circ \psi$. Since $\pi_+$ and $\pi_-$ are the only cohomological representation in degree $m$, Matsushima's formula states that 
$$\mathcal{H}_{(2)}^m(E) \cong \mathrm{Image} \ T_{\pi_+} \oplus \mathrm{Image} \ T_{\pi_-}.$$
Now let $\omega$ be a nonzero  $\SO(n-k,1)$-invariant vector in $\mathrm{Image} \ T_{\pi} \subset \mathcal{H}_{(2)}^m(E)$ with $\pi = \pi_+$ or $\pi_-$. 
Pulling $\omega$ back to $\Gamma_{\mathbf{x}} \backslash G$ by $p :\Gamma_{\mathbf{x}} \backslash G$ we obtain a left $\SO(n-k,1)$ -invariant right $K$-equivariant $\wwedge{m} \mathfrak{p}^*$-valued  function 
$F$ on $\Gamma_{\mathbf{x}} \backslash G$ which is square integrable. Then under the right action of $G$,  $F$ generates a unitary representation $\mathcal{V}$ isomorphic to  $\pi$. 
Note that, being harmonic, the form $\omega$ is smooth. But then evaluation of  the smooth vectors of $\mathcal{V}$ at the identity gives rise to  an  $\SO(n-k,1)$-invariant linear functional $\alpha$  on the smooth vectors of $\mathcal{V}$.  Note that the smooth vectors of $\mathcal{V} \subset L^2(G) \otimes \wwedge{m} \mathfrak{p}^*$ are realized as smooth functions by Dixmier and Malliavin \cite{DM} \cite{KM1}. Then $\alpha$ is nonzero because we have
$$ \alpha(F) \neq 0. $$ 
But $\alpha$ is $\SO(p)$-invariant with $p > m$ and by Lemma \ref{restrictioninvariance}, $(\pi_+ + \pi_-)^{\SO(p)} = 0$.  Hence $\alpha$ annihilates the $K$-finite vectors in $\pi_+ + \pi_-$ and hence $\alpha$ is identically zero which gives a contradiction.

The first part of (3) is an immediate consequence of (1) and (2). Now $\pi_+$ and $\pi_-$ are the only irreducible representations of $\SO (n,1)$ that contribute to cohomology in degree $m$.  Let $\mathcal{H}_\pm$ be the largest closed subspace of $L^2 (V^k)$ on which $\SO(n,1)$ acts by a multiple of $\pi_\pm$. Then $\mathcal{H}_\pm$ is stable under $\SO(n,1)$ and $\mathrm{Mp} (2k)$. By \cite[\S 7]{HoweTrans}, the joint action of $\SO(n,1) \times \mathrm{Mp} (2k)$ on $\mathcal{H}_\pm$ is a tensor product $\pi_\pm \otimes \sigma_\pm$, where $\sigma_\pm$ is an irreducible unitary representation of $\mathrm{Mp} (2k)$. Now since $\pi_\pm$ is square integrable, by \cite[\S 2]{LiDuke} we conclude that $\pi_\pm$ and $\sigma_\pm$ are the Howe duals of each other in the sense defined in \cite{HoweCorvalis}. The explicit determination of the Howe correspondence in that case is written down in 
\cite{LiDuke}. Since 
$$H^{k} (\mathfrak{so}(n,1) , \mathrm{SO} (n) ; L^2 (V^k)) = H^{k} (\mathfrak{so}(n,1) , \mathrm{SO} (n) ; \mathcal{H}_+ \oplus \mathcal{H}_-),$$
the theorem follows.

\end{proof}

\begin{rmk}
The representations $M_{k, m}^+$ and $M_{k, m}^-$ are the Howe duals of resp. $\pi^+$ and $\pi^-$. 
If $k=m$ both $M_{k, m}^+$ and $M_{k, m}^-$ are discrete series of $\mathrm{Mp}(2k)$.
\end{rmk}

\subsection{The odd dimensional case for general $k$}
Now assume $n=2m+1$ and $k$ arbitrary.  Then $\overline{H}^m( L^2(V^k)) \cong \overline{H}^{m+1}( L^2(V^k))$

\begin{thm} \label{n=2m+1cases}
\hfill

\begin{enumerate}
\item If $k \neq m$ then
$$H^m( L^2(V^k)) = \overline{H}^m( L^2(V^k)) = 0.$$

\item If $k = m$ then
$$H^m( L^2(V^k)) = \overline{H}^m( L^2(V^k)) \neq 0.$$
Furthermore,
$$H^m( L^2(V^m))^{(\mathrm{MU}(m))} = \mathcal{R}_m(V) \varphi_m.$$

\item If $k < m$ then
$$H^{m+1}(L^2(V^k)) = 0.$$

\item If $k \geq m$ then
$$H^{m+1}(L^2(V^k)) \neq 0.$$

\item If $k=m$ then
$$H^{m+1}(L^2(V^k)) = \overline{H}^{m+1}( L^2(V^k)) \cong \overline{H}^m( L^2(V^k)) \neq 0.$$
\end{enumerate}
\end{thm}
The theorem is a consequence of the following lemmas and propositions

\begin{lem} \label{L196}
$$H^m( L^2(V^k)) = \overline{H}^m( L^2(V^k)).$$
\end{lem}

\begin{proof}
By Proposition \ref{Mateisprop4}, it suffices to show that there is a spectral gap on each of $C^{m-1}(L^2(\Omega_{k-1,1}))$ and $C^{m-1}(L^2(\Omega_{k,0}))$.  The first follows from \cite{Pedon}, Corollary 4.4, and the second follows from Proposition \ref{Nicolaslowerbound1}.

\end{proof}

\begin{prop}
\hfill

\begin{enumerate}
\item $\overline{H}^m( L^2(\Omega_{k-1,1})) = 0$.
\item If $k \neq m$ then $\overline{H}^m( L^2(\Omega_{k,0})) = 0$.
\item If $k = m$ then $\overline{H}^m( L^2(\Omega_{k,0})) \neq 0$.
%= \mathcal{H}^m(L^2(\Omega_{k,0})) \cong \mathcal{H}^m(L^2(G \mathbf{x})) \otimes M_{k, m}$.
%\item If $k \geq n$ then
%$$H^{m+1}(L^2(V^k)) \neq 0.$$
\end{enumerate}

\end{prop}

\begin{proof}
%\textbf{We need a proof that $H^m(L^2(\Omega_{k-1,1}))$ is non-zero and a Borel reference for infinite dimensionality}

We prove $\overline{H}^m( L^2(\Omega_{k-1,1})) = 0$.  We use the transfer.  By \cite{Pedon}, Corollary 4.4, the spectrum of the Laplacian on square integrable $m$-forms is $[0, \infty)$ and by \cite{Pedon} Theorem 3.2 there is no discrete spectrum.  Hence we may apply Proposition \ref{Mateisprop3} to conclude that
$$\overline{H}^m( L^2(\Omega_{k-1,1})) = 0.$$

We now prove (2).  By the transfer map, we have
\begin{equation*}
\overline{H}^m(L^2(G \mathbf{x})) \cong \overline{H}_{(2)}^m(E_\mathbf{x})^{G_\mathbf{x}}.
\end{equation*}
But by \cite{MP}, Theorem 3.13,
\begin{equation*}
\overline{H}_{(2)}^m(E_\mathbf{x})^{G_\mathbf{x}} \cong H_c^m(E_\mathbf{x})^{G_\mathbf{x}}.
\end{equation*}
Since $E_\mathbf{x}$ may be thought of as a vector bundle of fiber dimension $k$ over $C_U$, we may apply the Thom Isomorphism Theorem and obtain
\begin{equation*}
H_c^m(E_\mathbf{x})^{G_\mathbf{x}} \cong H^{m-k}(C_\mathbf{x})^{G_\mathbf{x}}
\end{equation*}
But $H^{m-k}(C_\mathbf{x})^{G_\mathbf{x}} \neq 0$ if and only if $k = m$ or $m-k = n-k$, and the second of these cases is impossible.  We find
\begin{equation*}
\overline{H}^m(L^2(G \mathbf{x})) = \begin{cases} 0 \text{ if } k \neq m \\ \text{non-zero} \text{ if } k = m \end{cases}
\end{equation*}
Then we obtain (2) and (3) by Proposition \ref{tubevanishing}. 

%We now prove (4).  Let $\mathbf{x} \in \Omega_{k-1,1}$.  Then by transfer
%$$H^{m+1}(L^2(G_\mathbf{x})) \cong H^{m+1}_{(2)}(D)^{G_\mathbf{x}}$$
%with $G_\mathbf{x} \cong \SO(n-k+1)$.  If $k \geq n$ then $G_\mathbf{x}$ is trivial so
%$$H^{m+1}(L^2(G_\mathbf{x})) \cong H^{m+1}_{(2)}(D).$$
%But by Borel, \cite{Borel1}, Theorem B (ii), $H^{m+1}_{(2)}(D)$ is infinite dimensional.

\end{proof}

%\section{Nicolas}
%Let $n=2m+1$. Recall that $G = \mathrm{SO} (n,1)$ and $H =G_{\mathbf{x}} \cong \mathrm{SO} (n-k+1)$ or $\mathrm{SO} (n-k,1)$ according to which cone $\mathbf{x}$ belongs to.

\begin{lem}
$H^m(L^2(V^m))^{\mathrm{MU}(m)}$ is isomorphic to the space of $\mathrm{MU}(m)$-finite vectors in the discrete series representatino of $\mathrm{Mp}(m)$ with parameter $(\frac{n+1}{2}, \ldots, \frac{n+1}{2})$.
\end{lem}

\begin{proof}
The proof is analogous to the last paragraph of the proof of Theorem \ref{T:ksmall}.  Note that if $k = \frac{n+1}{2}$, the action of $\mathrm{Mp}(m)$ on $L^2(V^m)$ is strongly $L^2$ by \cite{Li}, Lemma 3.3 b.

\end{proof}

This completes the proof of (1) and (2) of Theorem \ref{n=2m+1cases}.  (3) and (4) are consequences of the following lemma and proposition

\begin{lem} \label{homnonzerobranch}
We have:
$$ (\wwedge{m} \C^n)^H \neq 0$$
if and only if $k>m$.
\end{lem}
\begin{proof} We successively apply the branching rule
\begin{equation*}
\SO(2m+1) \downarrow \SO(2m) \downarrow \SO(2m-2) \downarrow \cdots \downarrow \SO(2m+2-k)
\end{equation*}
to the highest weight $(1, \ldots ,1)$. We get the trivial weight if and only if there are at least $m$ steps (see Lemma \ref{restrictioninvariance}). 

\end{proof}

\begin{prop}
\hfill

\begin{enumerate}
\item $H^{m+1} (L^2 (V^k)) = \{ 0 \}$ if $k < m$.
\item $H^{m+1} (L^2 (V^k) )$ is nonzero if $k > m$.
\item $H^{m+1} (L^2 (V^m)) = \overline{H}^{m+1} (L^2 (V^m))$ is nonzero. 
\end{enumerate}
\item
\end{prop}

\begin{proof}
We first prove (1).  We claim that we have
$$C^m(\mathfrak{so} (n,1) , \mathrm{SO} (n) ; L^2 (V^k)) = 
C^{m+1}(\mathfrak{so} (n,1) , \mathrm{SO} (n) ; L^2 (V^k)) =0.  $$
Note that the Hodge star gives an isomorphism between the two cochain groups. 
It suffices to prove
$$\Hom_{\SO(n)}( \wedge^m(V_+ \otimes \C), L^2(V^k)) = 0.$$
To prove this it suffices to prove 
$$\Hom_{\mathrm{O}(n)}( \wedge^m(V_+), L^2(V^k)) = \Hom_{\mathrm{O}(n)}( \wedge^m(V_+) \otimes \det, L^2(V^k)) =0.$$
Now Li, in Section 5 of \cite{LiDuke}, describes explicitly the induced correspondence between $K$-types and infinitesimal characters. It first follows from his computations (see especially case (A) on page 927) that the $\mathrm{O}(n)$-type $\wedge^m(V_+)$ can only occur in the theta correspondence from $\mathrm{Mp} (2k)$ if $k \geq m$ and the $\mathrm{O}(n)$-type 
$\wedge^m(V_+) \otimes \det$ never occurs.
 In Li's notation our $m$ is $r$ and our $k$ is his $n$. In order that his $\sigma'$ exist we must have $r \leq n$. This concludes the proof of (1).

We now prove (2). It is enough to prove that for $\mathbf{x} \in \Omega_{k-1,1}$
$$H^{m+1} (L^2 (G \mathbf{x})) \cong H^{m+1} (L^2 (G/H))$$ 
is nonzero; here $G = \SO (n,1)$ and $H=\SO (n-k+1)$. 

Recall that the Plancherel formula for $G$ decomposes $L^2 (G)$ as 
$$L^2 (G) \cong \int_{\widehat{G}}^{\oplus} d\nu (\pi) \mathcal{H}_\pi \widehat{\otimes} \mathcal{H}_{\pi}^*.$$
This provides an analogous decomposition of $L^{2} (G/H)$ as (recall that $H$ is compact here):
\begin{equation*}
\begin{split}
L^2 (G/H) & = L^2 (G)^H \\
& \cong  \int_{\widehat{G}}^{\oplus} d\nu (\pi) \mathcal{H}_\pi \widehat{\otimes} (\mathcal{H}_{\pi}^*)^H \\
& \cong  \int_{\widehat{G}}^{\oplus} d\nu (\pi) \mathcal{H}_\pi \widehat{\otimes} \mathrm{Hom}_H (\mathcal{H}_{\pi} , \C).
\end{split}
\end{equation*}
Now the Plancherel measure $d\nu$ is supported on principal series $(\pi_{\sigma , \lambda} , \mathcal{H}_{\sigma , \lambda})$ (see e.g. \cite{Pedon} note that $n$ being odd there is no discrete series). Letting $P=MAN$ be the usual (minimal) parabolic subgroup of $G$ we get:
$$L^2 (G/H) \cong  \int_{W \backslash (\widehat{M} \times \mathfrak{a}^*)}^{\oplus} d\nu (\sigma, \lambda ) \mathcal{H}_{\sigma , \lambda} \widehat{\otimes}  \mathrm{Hom}_H (\mathcal{H}_{\sigma , \lambda} , \C),$$
where $W$ is the Weyl groupe $W(\mathfrak{g} , \mathfrak{a})$. 

Now $\mathcal{H}_{\sigma , \lambda} \cong L^2 (K,M, \sigma)$ as a $K$-module. The subgroup $H$ being compact we conclude that
$$\mathrm{Hom}_H (\mathcal{H}_{\sigma , \lambda} , \C) \cong \mathrm{Hom}_H (\C , L^2 (K,M, \sigma)).$$
We apply this to the representation $\sigma=\sigma_m$ of $M = \SO (n-1) = \SO (2m)$ of weight $(1, \ldots , 1)$ corresponding in the decomposition of $\wedge^m \C^{2m}$ into irreducibles to 
the eigenspace $i^{\left(\frac{n}{2}\right)^2}$ for the Hodge operator $*$. By the Frobenius reciprocity formula as a $K$-module $L^2 (K , M , \sigma)$ contains the  irreducible $
\wedge^m (V_+\otimes \C) = \wedge^{m} \C^{n}$. We conclude that 
$$\mathrm{Hom}_H (\mathcal{H}_{\sigma , \lambda} , \C ) \supset  (\wedge^m \C^n)^H.$$
This is non-zero by Lemma \ref{homnonzerobranch} provided $k>m$. Then 
%$$\mathrm{Hom}_H (\mathcal{H}_{\sigma , \lambda} , \C ) \supset  \mathrm{Hom}_H (\C , \wedge^m \C^n) \neq 0$$
%and 
$$\int_{\widehat{G}}^{\oplus} d\nu (\pi) \mathcal{H}_\pi \widehat{\otimes} \mathrm{Hom}_H (\mathcal{H}_{\pi} , \C)
\supset \int_{\widehat{G}}^{\oplus} d\nu (\pi) \mathcal{H}_\pi \otimes (\wedge^m \C^n)^H.$$
Hence $\int_{\mathfrak{a}^*}^{\oplus} d\nu_\sigma \mathcal{H}_{\sigma , \lambda} \otimes (\wedge^m \C^n)^H$
occurs as a direct summand in $L^2 (G/H)$.  Thus, to prove non-vanishing it suffices to show that $H^{m+1}(\mathfrak{g}, K; \int_{\mathfrak{a}^*}^{\oplus} d\nu_\sigma \mathcal{H}_{\sigma , \lambda} ) \neq 0$.  

We now show how the second part of the proof of Proposition 2.8 of \cite{Borel1}, see in particular Eq. (5)--(8),  shows that
$$H^{m+1}(\mathfrak{g}, K; \int_{\mathfrak{a}^*}^{\oplus} d\nu_\sigma \mathcal{H}_{\sigma , \lambda} ) \neq 0.$$
First, by \cite{BW}, Chapter I, Section 2, we have that for a $\mathfrak{g}, K$-module $\mathcal{V}$
$$\mathrm{Ext}^\bullet_{(\mathfrak{g},K)}(\C, \mathcal{V}) \cong H^\bullet(\mathfrak{g}, K; \mathcal{V}).$$
Now by applying Equation (6) of \cite{Borel1} we have
\begin{align*}
H^i(\mathfrak{g}, K; \int_{\mathfrak{a}^*}^{\oplus} d\nu_\sigma \mathcal{H}_{\sigma , \lambda} ) &=  \C \otimes H^i(\mathfrak{a}; \int_{\mathfrak{a}^*}^\oplus \C_{i \lambda} d\nu_{\sigma, \lambda})[-m] \\
&= H^{i-m}(\mathfrak{a}; \int_{\mathfrak{a}^*}^\oplus \C_{i \lambda} d\nu_{\sigma, \lambda}).
\end{align*}
By Equation (8) of \cite{Borel1}, $H^1(\mathfrak{a}; \int_{\mathfrak{a}^*}^\oplus I_\lambda d\nu_{\sigma, \lambda})$ is infinite dimensional.  Hence, setting $i = m+1$ we see this cohomology group is non-zero.

We finally prove (3). We have already proved that $\overline{H}^{m+1} (L^2 (V^k))$, which equals $\overline{H}^{m} (L^2 (V^k))$ by Poincar\'{e} duality, is nonzero when $k=m$.  It remains to prove that $H^{m+1} (L^2 (V^k))= \overline{H}^{m+1} (L^2 (V^k))$. By Proposition \ref{Mateisprop4}, it suffices to show that there is a spectral gap on each of $C^{m}(L^2(\Omega_{k-1,1}))$ and $C^{m}(L^2(\Omega_{k,0}))$. Recall any unitary irreducible representation of $G$ that is weakly contained in $L^2 (V^k)$, and therefore in $L^2(\Omega_{k-1,1})$ or $L^2(\Omega_{k,0})$, is contained in the image of the theta correspondence which maps a subset of the unitary dual of $\mathrm{Mp} (2k)$ into the unitary dual of $G$.  We will prove that the (unique) irreducible unitary representation that is cohomological in degree $m$ is isolated in this image. This will conclude the proof.  

We first recall from e.g. \cite{Pedon} that the unique irreducible unitary representation that is cohomological in degrees $m$ can be obtained as the Langlands' quotient $J(\sigma_m , 0)$ of the representation unitarily induced from the minimal parabolic $P=MAN$ of $G$ from the representation $\sigma_m$ of $M=\SO(n-1) = \SO (2m)$ of weight $(1, \ldots , 1)$, as above, and the trivial character of $A$. The infinitesimal character of $J(\sigma_m , 0)$ is 
$$\rho = (m , m-1 , \ldots , 1 , 0)$$
(considered up to the action of the Weyl group) the infinitesimal character of the trivial representation.   

The representation $J(\sigma_m , 0)$ is not isolated in the unitary dual of $G$ (see e.g. \cite{BSB,Vogan,JIMJ}): it is a limit point of both the unitary principal series $J(\sigma_m , it)$ as $t\to 0$ and the complementary series $J(\sigma_{m-1}, s)$ ($s\in (0,1)$) as $s \to  1$, where $\sigma_{m-1}$ is the standard representation of $M=\SO(2m)$ in $\wedge^{m-1} \C^{2m}$. But $J(\sigma , 0)$ is isolated from any other family of irreducible unitary representations. Indeed: if $\pi_i$ is a sequence that converges toward $J(\sigma_m , 0)$ then the smallest $K$-type $\tau_m $ of $J(\sigma_m, 0)$ eventually occurs as a $K$-type of $\pi_i$. Using the classification of the unitary dual of $\SO(n,1)$ we get (using Frobenius reciprocity as explained in \cite{Pedon} for example) that this forces $\pi_i$ to be the Langlands' quotient of an induced representation of the form $\sigma_m \otimes \lambda$ or $\sigma_{m-1} \otimes \lambda$. Next we use that the infinitesimal character of $\pi_i$ converges toward the infinitesimal character of the trivial representation and reach the conclusion.\footnote{In terms of differential forms, the representations $J(\sigma_m , it)$ correspond to co-closed forms of degree $m$, or closed forms of degree $m+1$, and eigenvalue $t^2$, whereas 
the representations $J(\sigma_{m-1} , s)$ correspond to closed forms of degree $m$ of co-closed forms of degree $m-1$, and eigenvalue $1-s^2$; see e.g. \cite[\S 6.5]{inventiones}.}

The complementary series $J(\sigma_{m-1}, s)$ also approaches the unique irreducible unitary representation that is cohomological in degree $m-1$. But we already have seen that there is a spectral gap in degree $m-1$; see Lemma \ref{L196}. It follows that there exists some positive $\varepsilon$ such that if $s \in (1-\varepsilon , 1)$ then $J(\sigma_{m-1}, s)$ is not contained in the image of the theta correspondence from $\mathrm{Mp} (2k)$ to $G$. 

It therefore remains to prove that the representations $J(\sigma_m , it)$ are not contained in the image of the theta correspondence when $t \neq 0$. To do so we note that the infinitesimal character of $J(\sigma_m , it)$ is 
$$(m , m-1 , \ldots , 2 , 1, it)$$ 
(considered up to the action of the Weyl group).  But Li \cite[\S 5]{LiDuke} describes explicitely the theta correspondence between infinitesimal characters: denoting by $\lambda' = (\lambda_1 ' , \ldots , \lambda_k ')$ the infinitesimal character of an irreducible representation $\pi'$ of $\mathrm{Mp} (2k)$ that occurs in the theta correspondence to $G$ (with $k \leq m$), we recall from \cite[Eq. (30) on p.   926]{LiDuke} that the infinitesimal character of the image of $\pi'$ by the theta correspondence is 
$$\lambda = ( \lambda_1 ' , \ldots , \lambda_k ' , m-k, m-k-1 , \ldots , 1 , 0)$$
(considered up to the action of the Weyl group). In particular if $k \leq m$ and $t \neq 0$ we have:
$$\lambda \neq (m , m-1 , \ldots , 1 , it)$$
and $J(\sigma_m , it)$ ($t\neq 0$) does not occur in the theta correspondence from $\mathrm{Mp} (2k)$ if $k \leq m$.  These groups are non-zero by the existence of the harmonic form $\phi_k$ of Kudla and Millson \cite{KM1}, see Part 3 below for details.  This concludes the proof of (3).

\end{proof}

Combining the results of the previous sections we obtain Theorem \ref{L^2mainthm}.

\bigskip

\newpage

\part{Comparison of the cohomology with $\mathcal{P}(V^k)$ coefficients and $L^2(V^k)$ coefficients}

The purpose of Part 3 is to prove Theorem \ref{part3theorem} of the introduction which we now restate for the convenience of the reader

\textbf{Theorem}
\hfill
\begin{enumerate}
\item Suppose $\ell \neq n$ and $k \leq \frac{n}{2}$, then the map from $H^\ell( \mathcal{P}(V^k))$ to $H^\ell( L^2(V^k))$ is an injection.  Furthermore, if $k < \frac{n-1}{2}$ then the map from $H^k( \mathcal{P}(V^k))$ to $H^k( L^2(V^k))$ is a surjection onto the $\mathrm{MU}(k)$-finite vectors.
\item If $k > \frac{n}{2}$ or $\ell = n$ then the map $H^\ell( \mathcal{P}(V^k))$ to $H^\ell( L^2(V^k))$ is the zero map.
\end{enumerate}

We begin by proving statement $(2)$.  By Theorem \ref{main} $H^\ell( \mathcal{P}(V^k)) = 0$ unless $\ell = k$ or $n$.  Next note that for all $k$ we have $H^n( L^2(V^k)) = H^0( L^2(V^k)) = 0$ by Poincar\'{e} duality.  Finally, we need only show that the map is zero when $\ell = k$.  But, by Theorem \ref{L^2mainNicolas}, when $k > \frac{n+1}{2}$ we have $H^k( L^2(V^k)) = 0$.

Our goal for the remainder of the paper is to prove statement $(1)$.  Thus, we assume from now on that
\begin{equation} \label{part3k}
k \leq \frac{n}{2}.
\end{equation}

\section{The harmonic cocycle $\phi_k$}
Recall $\beta = (\mathbf{x}, \mathbf{x})$ and set $R = (\mathbf{x}_0, \mathbf{x}_0)$.  We now define a $k$-cochain $\phi_k$ with values in the functions on $V^k$ by
\begin{equation}\label{definingequationforphi}
\phi_k(\mathbf{x}) =  \frac{\det(\beta)^{\frac{n-k}{2}}} {\det(R)^{\frac{n-k+1}{2}}} (\eta_1 \wedge \cdots \wedge \eta_k).
\end{equation}

Hence the extension $\widetilde{\phi_k}$ of $\phi_k$ to a $G$-invariant form on $D$ with values in the functions on $V^k$ is given by
\begin{equation} \label{phitilde}
\widetilde{\phi_k}(\mathbf{x},z) =  \frac{\det(\beta)^{(\frac{n-k}{2})}}{\det((\mathbf{x}_z, \mathbf{x}_z))^{(\frac{n-k+1}{2})}} (\eta_1 \wedge \cdots \wedge \eta_k).
\end{equation}

Note that $\widetilde{\phi_k}(\mathbf{x},z)$ is the form $(\varphi_{0,\mathbf{x}})_z$ of \cite{KM1}, page 230.  Now fix $\mathbf{x} \in \Omega_{k,0}$ and consider $\phi_k$ as an element of $C^k(L^2(G\mathbf{x}))$.  Then the transfer of $\phi_k$ to the tube $E$, which we again denote $\phi_k$, was studied in \cite{KM1}, pages 218-219, where it is denoted $\varphi$.  From \cite{KM1}, page 219, since $k \leq \frac{n}{2}$ by assumption \eqref{part3k}, we see $\phi_k$ is harmonic and $L^2$.

Our goal in the following section is to prove the following theorem. 
\begin{thm} \label{Harmonicprojectiontoromanphi}
If $k \leq \frac{n}{2}$ then the harmonic projection of the restriction of $\varphi_k$ to every orbit in $\Omega_{k,0}$ is nonzero. Hence the image of the cohomology class of $\varphi_k$ in the cohomology with $L^2(V^k)$ coefficients is nonzero.
\end{thm}
\begin{cor}
If $k \leq \frac{n}{2}$ then except for top dimensional cohomology, the map from polynomial Fock space cohomology to $L^2$ cohomology is injective.
\end{cor}

\section{The harmonic projection of $\varphi_k$ is a multiple of $\phi_k$}

We now use the transfer $T_{k,0}$ to transfer $\varphi_k$ and $\phi_k$ to $E$.  We will abuse notation and use $\varphi_k$ and $\phi_k$ to denote their transfers to the tube.

The following proposition is essentially a special case of \cite{M}, Lemma III.3.2 and the Main Lemma on page 35.  The only point that requires attention is that there is a linear change of variable required to pass from the formulas of \cite{M} to the formulas of this paper.  Put $\beta = (\mathbf{x}, \mathbf{x})$.

\begin{lem} \label{varphikduality}
For any closed, smooth, bounded $(n-k)$-form $\eta$ on $E$, we have
\begin{equation*}
\int_E \varphi_k(\mathbf{x}) \wedge \eta = \kappa_1(\mathbf{x}) \int_{C_U} \eta
\end{equation*}
where $\kappa_1(\mathbf{x}) = (\frac{\sqrt{\pi}}{2})^k e^{-\frac{ tr \beta}{2}}$.
\end{lem}

\begin{cor} \label{hvarphikduality}
For any closed, smooth, bounded, square-integrable $(n-k)$-form $\eta$ on $E$, we have
\begin{equation*}
\int_E \mathcal{H}(\varphi_k)(\mathbf{x}) \wedge \eta = \kappa_1(\mathbf{x}) \int_{C_U} \eta
\end{equation*}
where $\kappa_1(\mathbf{x}) = (\frac{\sqrt{\pi}}{2})^k e^{-\frac{ tr \beta}{2}}$.
\end{cor}

\begin{proof}
By the Hodge decomposition  we have $\mathcal{H}(\varphi_k) - \varphi_k = d \tau$ where $\tau = d^{\ast} \mathcal{G} \varphi_k$ is square integrable and smooth.  Then we have
\begin{align*}
\int_E (\mathcal{H}(\varphi_k)(\mathbf{x}) - \varphi_k(x)) \wedge \eta &= \int_E (d \tau) \wedge \eta \\
&= \int_E d( \tau \wedge \eta) = 0.
\end{align*}
The last equality follows because $\tau \wedge \eta$ is integrable and smooth.

\end{proof}

Recall from \cite{KM1} we have an analytic continuation $\phi_{k,s}$ of $\phi_k$ given by
\begin{equation} \label{phiks}
\phi_{k,s} = \det(\beta)^{n-k+2s} \det(R)^{-(n-k+2s+1)} \eta_1 \wedge \cdots \wedge \eta_k
\end{equation}

Now fix $\mathbf{x}$ and transfer to the tube $E$.  Then by \cite{KM1}, Lemma 3.2, we obtain
\begin{equation}
\phi_{k,s} = \cosh(t)^{-(n-k+2s)} \sinh(t)^{k-1} dv_2 \wedge dt.
\end{equation}
From this formula we see that $\phi_{k,s}$ is closed for all $s$, is  integrable if $\mathrm{Re}(s) > \frac{k-1}{2}$, and is square-integrable if $\mathrm{Re}(S) \geq 0$.

Then by \cite{KM1} Lemma 3.3 and Proposition 3.4 we obtain
\begin{lem} \label{phiksduality}
For any closed, smooth, square integrable $(n-k)$-form $\eta$ on $E$, we have, for $\mathrm{Re}(s) > \frac{k-1}{2}$
\begin{equation*}
\int_E \phi_{k,s} \wedge \eta = \kappa_2(\mathbf{x},s) \int_{C_U} \eta
\end{equation*}
where $\kappa_2(\mathbf{x},s) = \frac{1}{2} \vol(S^{k-1}) \frac{\Gamma(\frac{k}{2}) \Gamma(s+ \frac{m}{2} - k)}{\Gamma(s+\frac{m}{2} - \frac{k}{2})} = \pi^\frac{k}{2} \frac{\Gamma(s+ \frac{m}{2} - k)}{\Gamma(s+\frac{m}{2} - \frac{k}{2})}$
\end{lem}

The right-hand side of the identity in Lemma \ref{phiksduality} can be meromorphically continued to the entire plane in $s$ and is regular at $s=0$.  Since $\eta$ is $L^2$ and $\phi_{k,s}$ is square integrable for $\mathrm{Re}(s) \geq 0$, the left-hand side of the identity can be meromorphically continued to the closed half plane $\mathrm{Re}(s) \geq 0$.  Evaluating both sides at $s=0$, and noting that $\phi_k = \phi_{k,0}$, we obtain

\begin{lem} \label{phikduality}
For any closed, smooth, square integrable $(n-k)$-form $\eta$ on $E$, we have
\begin{equation*}
\int_E \phi_k \wedge \eta = \kappa_2(\mathbf{x}) \int_{C_U} \eta
\end{equation*}
where $\kappa_2(\mathbf{x}) = \kappa_2(\mathbf{x},0)$.
\end{lem}

\begin{prop} \label{varphi=phi}
\begin{equation*}
\mathcal{H}(\varphi_k) = \frac{\kappa_1(\mathbf{x})}{\kappa_2(\mathbf{x})} \phi_k = \frac{\Gamma(\frac{m}{2}-\frac{k}{2})}{2^k \Gamma(\frac{m}{2}-k)} e^{- \frac{tr \beta}{2}} \phi_k.
\end{equation*}
\end{prop}

\begin{proof}
\begin{align*}
&||\mathcal{H}(\varphi_k) - \frac{\kappa_1(\mathbf{x})}{\kappa_2(\mathbf{x})} \phi_k||^2 \\
&= \int_E (\mathcal{H}(\varphi_k) - \frac{\kappa_1(\mathbf{x})}{\kappa_2(\mathbf{x})}\phi_k) \wedge *(\mathcal{H}(\varphi_k) - \frac{\kappa_1(\mathbf{x})}{\kappa_2(\mathbf{x})}\phi_k) \\
&= \int_E \mathcal{H}(\varphi_k) \wedge *(\mathcal{H}(\varphi_k) - \frac{\kappa_1(\mathbf{x})}{\kappa_2(\mathbf{x})}\phi_k) - \int_E \frac{\kappa_1(\mathbf{x})}{\kappa_2(\mathbf{x})}\phi_k \wedge *(\mathcal{H}(\varphi_k) - \frac{\kappa_1(\mathbf{x})}{\kappa_2(\mathbf{x})}\phi_k) \\
&= \kappa_1(\mathbf{x}) \int_{C_U} * \bigg(\mathcal{H}(\varphi_k) - \frac{\kappa_1(\mathbf{x})}{\kappa_2(\mathbf{x})}\phi_k \bigg) - \kappa_2(\mathbf{x}) \frac{\kappa_1(\mathbf{x})}{\kappa_2(\mathbf{x})} \int_{C_U} * \bigg( \mathcal{H}(\varphi_k) - \frac{\kappa_1(\mathbf{x})}{\kappa_2(\mathbf{x})}\phi_k \bigg) = 0
\end{align*}
Here the passage from the second line to the third line follows from Corollary \ref{hvarphikduality} and Lemma \ref{phikduality}.

\begin{cor} \label{imagenonzero}
\begin{equation*}
\mathcal{H}(f(\beta) \varphi_k) = \frac{\kappa_1(\mathbf{x})}{\kappa_2(\mathbf{x})} f(\beta) \phi_k.
\end{equation*}
\end{cor}

\end{proof}

\section{Proof of Theorem \ref{part3theorem}}
We now prove Theorem \ref{part3theorem} $(1)$.  First, suppose for the purpose of contradiction $\varphi_k$ is the coboundary of a square-integrable cochain.  Then the restriction of $\varphi_k$ to $\Omega_{k,0}$ is exact.  From the argument of Proposition \ref{Mateisprop1} we find that $\varphi_k$ has the canonical primitive $\tau = d^* \mathcal{G} \varphi_k$.  Then $\tau$ is defined and satisfies
$$d(\tau|_{G\mathbf{x}}) = \varphi_k|_{G \mathbf{x}} \text{ for all } x \in \Omega_{k,0}.$$
Hence $\mathcal{H}((\varphi_k)|_{G \mathbf{x}}) = 0$ for all $x \in \Omega_{k,0}$, which contradicts Proposition \ref{varphi=phi}.
\begin{rmk}
{\rm If we replace $\varphi_k$ with $f(\beta) \varphi_k$ in the above argument, we find $\mathcal{H}((f(\beta) \varphi_k)|_{G \mathbf{x}}) = 0$ for all $x \in \Omega_{k,0}$, which contradicts Corollary \ref{imagenonzero}.}
\end{rmk}

By Theorem \ref{main}, every nonzero element of $H^k(\mathcal{P}(V^k))$ is of the form $f(\beta) \varphi_k$, and hence the kernel of the map induced by $\mathcal{P}(V^k) \hookrightarrow L^2(V^k)$ is zero and thus is an injection.

To show this map is a surjection onto the $\mathrm{MU}(k)$-finite vectors, we note that as a representation of $\mathrm{Mp}(2k, \R)$, $H^k(L^2(V^k))$ is the holomorphic discrete series representation of weight $(\frac{n+1}{2}, \ldots, \frac{n+1}{2})$ and the class of $\varphi_k$ is a lowest weight vector because it is the image of a lowest weight vector.

%We now prove Theorem \ref{part3theorem} $(1)$.  It suffices to show that the image of $f(\beta) \varphi_k$ is non-zero.  Suppose there exists $\tau \in C^{k-1}(L^2(V^k))$ with
%\begin{equation} \label{contradiction}
%d \tau = f(\beta) \varphi_k.
%\end{equation}
%Choose $\mathbf{x} \in \Omega_{k,0}$ such that $\tau$ is defined on $G\mathbf{x}$ and
%$$d(\tau|_{G\mathbf{x}}) = f(\beta) \varphi_k|_{G\mathbf{x}}.$$
%Hence $\mathcal{H}(f(\beta) \varphi_k|_{G\mathbf{x}}) = 0$ which contradicts Proposition \ref{varphi=phi}.

%\textbf{we are working here} Harmonic representatives.  We are assunming $k \leq \frac{n}{2}$ and $\ell =k$.  For every $\mathbf{x} \in \Omega_{k,0}$, the image of $\varphi_k$ in $L^2(G\mathbf{x})$ is non-zero in $H^k(L^2(G\mathbf{x}))$. 

\end{document}